\documentclass[11pt]{amsart}
\usepackage{amsmath, amsfonts, amsthm, amssymb}
\usepackage[all,cmtip,line]{xy}
\usepackage{array}
\usepackage{enumerate}
\usepackage{srcltx} 

\newtheorem{thm}{Theorem}[section]
\newtheorem{lem}{Lemma}[section]
\newtheorem{prop}[lem]{Proposition}
\newtheorem{cor}[lem]{Corollary}

\newtheorem{rem}[lem]{Remark}

\numberwithin{equation}{section}

\newcommand{\bR}{ \mathbb{R}} 
\newcommand{\bN}{ \mathbb{N}} 
\newcommand{\bZ}{ \mathbb{Z}} 
\newcommand{\bT}{ \mathbb{T}} 

\newcommand{\defd}{:=}

\newcommand{\Lm}{ {\mathcal L}}

\newcommand \ON{\;\text{ on }}
\newcommand \AND{\;\text{ and }\;}

\newcommand \IN{\;\text{ in }}

\newcommand \FOR{\;\text{ for }}
\newcommand \FORA{\;\text{ for all }}

\newcommand \eps{\varepsilon}

\newcommand \epsTwo{\varepsilon_3}
\newcommand \epsTT{\varepsilon_2}
\newcommand \epsThree{\varepsilon_4}
\newcommand \epsFour{\varepsilon_5}

\newcommand \CSix{C_7}
\newcommand \CSeven{C_8}
\newcommand \CEight{C_9}
\newcommand \CNine{C_{10}}
\newcommand \CTen{C_{11}}
\newcommand \CElev{C_{12}}
\newcommand \CTwelve{C_{13}}
\newcommand \CThirt{C_6}

\newcommand \pN{p}   
\newcommand \pNP{q}  
\newcommand \FNP{F}  
\newcommand \Mp{\mathcal G}
\newcommand \hOne{w}

\newlength{\originalbase}
\newcommand{\spacing}[1]{\setlength{\baselineskip}{#1\originalbase}}

\begin{document}               

\newcommand{\avint}{{- \hspace{-3.5mm} \int}}

\spacing{1}

\title{Semi-geostrophic system with variable Coriolis parameter}
\author{Jingrui Cheng}
\address{Department of Mathematics\\
         University of Wisconsin\\
         Madison WI 53706\\
USA}
\email{jrcheng@math.wisc.edu}
\author{Michael Cullen}
\address{
Met Office\\ Fitzroy Road\\ Exeter Devon EX1 3PB\\ United Kingdom}
\email{mike.cullen@metoffice.com}
\author{Mikhail Feldman}
\address{Department of Mathematics\\
         University of Wisconsin\\
         Madison WI 53706\\
USA}
\email{feldman@math.wisc.edu}
\date{\today}
\pagestyle{myheadings}

\maketitle
\begin{abstract}
We prove short time existence and uniqueness of smooth solutions
(in $C^{k+2,\alpha}$ with $k\geq 2$) to the 2-D semi-geostrophic system and
semi-geostrophic shallow water system with variable Coriolis parameter $f$ and
periodic boundary conditions, under the natural convexity condition on the
initial data. The dual space used in analysis of the semi-geostrophic system with
constant $f$ does not exist for the variable Coriolis parameter case, and we
develop a time-stepping procedure in Lagrangian coordinates
to overcome this difficulty.
\end{abstract}

\section{Introduction}
\label{section-1}

The semi-geostrophic system (abbreviated as SG) is a model of large-scale
atmospheric/oceanic flows, where "large-scale" means that the flow is rotational
dominated. All previous works on analysis of the SG system have been restricted
to the case of constant Coriolis force, where the ability to write the equations
in 'dual' coordinates enables the equations to be solved in that space and then
mapped back to physical space.  Examples are the results of
Benamou and Brenier \cite{benamou}, Cullen and Gangbo \cite{cul01},
Cullen and Feldman \cite{culfeld},
Ambrosio, Colombo, De Philippis and Figalli\cite{Euler solution}.
All these solve SG subject to a convexity
condition introduced by Cullen and Purser \cite{cul84}.
Convexity condition allows to  interpret the mapping  between the physical and
dual spaces as an optimal map for Monge-Kantorovich mass transport problem,
which makes possible the use of methods of Monge-Kantorovich theory in the study of SG with
constant Coriolis force.

The background and applicability of this model is reviewed by Cullen \cite{cul06}. In the atmosphere, this model is applicable on scales larger than about 1000km, which is comparable to the radius of the Earth. Thus the variations of the vertical component of the Coriolis force have to be taken into account, as these are a fundamental part of atmospheric dynamics on this scale. Thus SG with variable rotation (i.e. Coriolis parameter) is more physically realistic.

Attempts to extend the theory to the case of variable rotation were made by Cullen {\em et al.} \cite{cul05} and Cullen \cite {cul08}. These included formal arguments why the equations should be solvable. In particular, they derived a solvability condition in the form of the positive definiteness of a stability matrix which generalises the convexity condition used in the constant rotation case. They also showed that geostrophic balance could be defined by the condition that the energy was stationary under a  certain class of Lagrangian displacements. These properties suggest that a rigorous existence proof should be possible.

In this paper we prove short time existence and uniqueness of smooth solutions to SG with variable Coriolis parameter subject to the strict positive definiteness of the stability matrix. Since dual variables are not available, the result has to be proved working directly in the physical coordinates. Somewhat surprisingly, Monge-Ampere type equations appear in this process, even though we do not use Monge-Kantorovich mass transport as in case of constant Coriolis parameter. We consider two versions of the SG equations. The simplest to analyse is two-dimensional incompressible SG flow. However, this is not a physically relevant model. We therefore also analyse the SG shallow water equations, which are an accurate approximation to the full shallow water equations on large scales.

\section{Formulation of the problems and main results}
\label{FromMainTh-Sect}
We consider the SG equation with non-constant Coriolis force on a 2-D flat torus:
\begin{align}\label{1.1}
&(u_1^g,u_2^g)=f^{-1}(-\frac{\partial p}{\partial x_2},
\frac{\partial p}{\partial x_1}), \\
&\label{1.2}
D_tu_1^g+\frac{\partial p}{\partial x_1}-fu_2=0,\\
&\label{1.3}
D_tu_2^g+\frac{\partial p}{\partial  x_2}+fu_1=0,\\
&\label{1.4}
\nabla\cdot\mathbf{u}=0,
\end{align}
with initial data $p|_{t=0}=p_0(x)$. Here $\mathbf{u}=(u_1,u_2)$ is the physical velocity and  $D_t=\partial_t+\mathbf{u}\cdot\nabla$, the material derivative. $\mathbf{u}_g=(u_1^g,u_2^g)$ is the
geostrophic wind velocity, $p$ is the pressure, and $f=f(x)$ is the Coriolis parameter,
which  is a given smooth positive function.

In this paper we consider two-dimensional periodic case. That is, all the functions appearing above
are assumed to be defined on $\bR^2$ and periodic with respect to $\bZ^2$,
hence can be thought of defined on a 2-D torus.

Physically interesting solutions of the SG system must satisfy the convexity principle introduced by Cullen and Purser \cite{cul84}. In the case when $f\equiv1$, the convexity condition means that the modified pressure function
$P(x_1,x_2)=p(x_1,x_2)+\frac{1}{2}(x_1^2+x_2^2)$ is convex.
We will introduce the analogue of this convexity condition when $f$ is not
a constant, see (\ref{1.7}) below, and prove short time existence and
 uniqueness of solutions when this condition is satisfied by the initial data.
 Before we state the main results of this paper, we first
 introduce some notations:

In the following, we identify $\bT^2$ with $\bR^2/\bZ^2$.

We will denote $C^{k,\alpha}(\bT^2)$ to be the space of $C^{k,\alpha}$ functions on $\bR^2$ and periodic with respect to $\bZ^2$, which is equipped with the norm
$$\|p\|_{k,\alpha}=\sum_{0\leq|\beta|\leq k}\|D^{\beta}p\|_0+\sum_{|\beta|=k}[D^{\beta}u]_{\alpha},
$$
where
$$\|v\|_0=\max_{x\in\bR^2}|v(x)|,\qquad
[v]_{\alpha}=\sup_{x,y\in\bR^2}\frac{|v(x)-v(y)|}{|x-y|^{\alpha}}.
$$
Sometimes we will write $C^{k,\alpha}$ instead of $C^{k,\alpha}(\bT^2)$ for
simplicity. Similarly define $L^2(\bT^2)$ which consists of periodic functions
which is in $L^2_{loc}(\bR^2)$. All these spaces can be equivalently understood as corresponding spaces on 2-d torus $\bT^2$.

Now we state the main result of this paper:
\begin{thm}\label{1.5}
Let $k\ge 2$ be integer.
 Let $f\in C^{k,\alpha}(\mathbf{\bT^2})$ with $f(x)>0$ on $\bT^2$.
Let $p_0\in C^{k+2,\alpha}(\mathbf{\bT^2})$ with $\int_{\bT^2}p_0(x)dx=0$.
Suppose also the following convexity-type condition is satisfied
\begin{equation}\label{1.6}
I+f^{-1}D(f^{-1}Dp_0)\ge c_0 I \ON\bT^2\textrm{ for some } c_0>0.
\end{equation}
Then there exists $T_0>0$, depending on $\|p_0\|_{k+2,\alpha}$, $c_0$, $f$ and $k$, such that there exists a solution $(p,u_g,v_g,\mathbf{u})$ to (\ref{1.1})-(\ref{1.4}) with initial data $p_0$ on $[0,T_0]\times\bT^2$ which satisfies
\begin{align}\label{1.7}
&I+f^{-1}D(f^{-1}Dp)>0\ON[0,T_0]\times\bT^2,\\
\label{1.7-1}
&\int_{\bT^2}p(t,x)dx=0 \FORA t\in[0,T_0],
\end{align}
and the following regularity
\begin{equation}\label{1.8}
\partial_t^mp\in L^{\infty}(0,T_0;C^{k+2-m,\alpha}(\bT^2)) \FOR
0\leq m\leq k+1,\;  \mathbf{u}\in L^{\infty}(0,T_0;C^{k,\alpha}(\bT^2)).
\end{equation}
Moreover, any solution $(p,\mathbf{u}_g,\mathbf{u})$ to (\ref{1.1})-(\ref{1.4}) with initial data $p_0$, defined on $[0,T]\times\bT^2$ for some $T>0$, which satisfies (\ref{1.7}), (\ref{1.7-1}) and has regularity $p\in L^{\infty}(0,T;C^3(\bT^2))$, $\partial_tp\in L^{\infty}(0,T;C^2(\bT^2))$ is unique.
\end{thm}
Similar results hold for the semi-geostrophic shallow water system (\ref{5.1})-(\ref{5.4}).
\begin{thm}\label{1.9}
Let $k\ge 2$ be integer.
 Let $f\in C^{k,\alpha}(\mathbf{\bT^2})$ with $f(x) >0$ on $\bT^2$.
Let $h_0\in C^{k+2,\alpha}(\bT^2)$ with $\int_{\mathbf{T^2}}h_0(x)dx=1$. Suppose also the following convexity and positivity conditions are satisfied for initial data:
\begin{equation}\label{1.10}
I+f^{-1}D(f^{-1}Dh_0)\ge c_0 I \AND h_0\ge c_1  \ON\bT^2\;\textrm{ for some } c_0,\,c_1>0.
\end{equation}
Then there exists $T_0>0$, depending on $\|h_0\|_{k+2,\alpha}$, $c_0,c_1$, $f$ and $k$, such that there exists a solution $(h,u_g,v_g,\mathbf{u})$ to (\ref{5.1})-(\ref{5.4}) with initial data $h_0$ on $[0,T_0]\times\bT^2$ which satisfies
\begin{align}\label{1.11}
&I+f^{-1}D(f^{-1}Dh)>0,\; h>0\ON[0,T_0]\times\bT^2,\\
\label{1.11-1}
& \int_{\bT^2}h(t,x)dx=1 \FORA t\in[0,T_0],
\end{align}
and the following regularity
\begin{equation}
\partial_t^mh\in L^{\infty}(0,T_0;C^{k+2-m,\alpha}(\bT^2)) \FOR 0\leq m\leq k+1,  \mathbf{u}\in L^{\infty}(0,T_0;C^{k,\alpha}(\bT^2)).
\end{equation}
Moreover, any solution $(h,\mathbf{u}_g,\mathbf{u})$ to (\ref{5.1})-(\ref{5.4}) with initial data $h_0$, defined on $[0,T]\times\bT^2$ for some $T>0$ which satisfies (\ref{1.11}), (\ref{1.11-1}) and has regularity $h\in L^{\infty}(0,T;C^{3}(\bT^2))$, $\partial_th\in L^{\infty}(0,T;C^2(\bT^2))$ is unique.

\end{thm}

All previous works on existence of solutions for the SG system concern the case when the Coriolis parameter $f$ is constant (and then by rescaling
we can set $f\equiv1$), and make use of the dual space. Namely, we introduce the "geopotential" $P=p(x_1,x_2)+\frac{1}{2}(x_1^2+x_2^2)$, then the system
(\ref{1.1})-(\ref{1.4}) can be put in the form.
\begin{align}\label{1.13}
&D_t(\nabla P)=J(\nabla P-id),\\
\label{1.14}
&\nabla\cdot\mathbf{u}=0,
\end{align}
with initial conditions
\begin{equation}\label{1.15}
P|_{t=0}=p_0+\frac{1}{2}(x_1^2+x_2^2),
\end{equation}
where
\begin{displaymath}
J =
\left( \begin{array}{cc}
0 &-1  \\
1& 0  \\
\end{array} \right).
\end{displaymath}
Cullen-Purser convexity condition is that $P$ is convex, which
coincides with condition (\ref{1.7}) for $f\equiv 1$.
For each $t>0$,
introduce the measure $\nu_t=\nabla P(t, \cdot)_\#(\Lm_{\bT^2})$ (i.e. the push-forward of the
Lebesgue measure
$\Lm_{\bT^2}$ on the torus by the map $\nabla P(t, \cdot)$,
and let $\nu$ be a measure on $[0,\infty)\times \bT^2$ defined by
$d\nu=d\nu_tdt$.
Then the measure $\nu$ will satisfy the equation
\begin{equation}\label{1.16}
\partial_t\nu+\nabla\cdot(U\nu)=0,
\end{equation}
with
\begin{equation}\label{1.17}
U=J(id-\nabla P^*),
\end{equation}
with also the  initial condition
\begin{equation}\label{1.18}
\nu|_{t=0}=\nu_0:=\nabla P_{0\#}(\Lm_{\bT^2}).
\end{equation}
Here $P^*$ is the Legendre transform of the convex function $P$. Notice the vector field (\ref{1.17}) is divergence-free.

	To prove existence of the SG system (\ref{1.1})-(\ref{1.4}) when $f\equiv1$, or equivalently (\ref{1.13})-(\ref{1.14}), it is easier to first consider the existence of solutions to the system (\ref{1.16})-(\ref{1.18}).
In general, thes solution has low regularity which makes it difficult to
transform that solution back to the physical space. For general initial data
with $\nu_0\in L^p$,
it is shown in \cite{culfeld} that a solution in physical space exists in
Lagrangian sense. In \cite{FeldTud-15,FeldTud-15a} a weaker form of
Lagrangian solutions  in physical space was obtained in case when $\nu_0$ is a general measure.
If
 the solution $(\nu, P)$ in dual space has enough regularity, then such solutions can be transformed back to physical variables and give Eulerian solutions to the original equation (\ref{1.13})-(\ref{1.14}). In the case when the density of the initial measure $\nu_0$ is
between two positive constants:
  $0<\lambda\leq\nu_0\leq\Lambda$ on $\bT^2$,
Ambrosio et al \cite{Euler solution} obtained a solution to (\ref{1.16})-(\ref{1.18}) with $P^*\in W^{2,1}(\bT^2)$,  and this regularity turns out to be sufficient to transform back and give a weak solution to (\ref{1.1})-(\ref{1.4})
in the sense of distributions. For smooth solutions, Loeper \cite{Loeper} obtained short time existence of smooth solutions to (\ref{1.16})-(\ref{1.18}) when $\nu_0$ is smooth and positive. And because of smoothness, there is no difficulty to rewrite the equation in terms of original physical variables.

The  approach described above does not work for the system (\ref{1.1})-(\ref{1.4}) when $f$ is not a constant,  because a dual space cannot be defined in such case. Therefore,
we work directly with system (\ref{1.1})-(\ref{1.4}).

As a first attempt, we may try the following argument. Note
that (\ref{1.2}), (\ref{1.3}) is a linear
 algebraic system for the physical velocity $\mathbf{u}$.
 Assuming $I+f^{-1}D(f^{-1}Dp)>0$, solve for $\mathbf{u}$ and use
 the definition of $\mathbf{u}_g$ by (\ref{1.1}) to obtain
\begin{equation}\label{2.19new}
\mathbf{u}=(I+f^{-1}D(f^{-1}Dp))^{-1}(f^{-1}J\nabla p-f^{-2}\partial_t\nabla p).
\end{equation}
Substituting (\ref{2.19new}) into (\ref{1.4}), we obtain an elliptic equation in $\partial_tp$:
\begin{equation}\label{2.20new}
\nabla\cdot\big[(I+f^{-1}D(f^{-1}Dp))^{-1}f^{-2}\partial_t\nabla p\big]=\nabla\cdot\big[(I+f^{-1}D(f^{-1}Dp))^{-1}f^{-1}J\nabla p\big].
\end{equation}
Then we may try to solve (\ref{1.1})-(\ref{1.4}) by a fixed point argument.
For a given $p$, we solve the elliptic equation:
\begin{equation}\label{2.21new}
\nabla\cdot\big[(I+f^{-1}D(f^{-1}Dp))^{-1}f^{-2}\nabla w\big]=\nabla\cdot\big[(I+f^{-1}D(f^{-1}Dp))^{-1}f^{-1}J\nabla p\big].
\end{equation}
We expect the solution $w$ to give $\partial_tp$.
Hence we define $\hat{p}(t,x)=p_0(x)+\int_0^tw(s,x)ds$. This procedure gives a
map $p\rightarrow\hat{p}$. If this map has a fixed point $p$ and it is smooth,
then it gives a solution to the system (\ref{1.1})-(\ref{1.4}). This approach runs
into a serious difficulty because of loss of derivative.
Indeed, if we assume $p$ to be in $C^{k+2,\alpha}$ in spatial variables, then
the coefficients and the right-hand side  of the divergence form elliptic
equation (\ref{2.21new}) are in $C^{k,\alpha}$. Then, from the standard
elliptic estimates, the solution $w$ will be
$C^{k+1,\alpha}$ in $x$-variables. Next we integrate $w$ in time, and the resulting
 $\hat{p}$
has regularity $C^{k+1,\alpha}$ in spatial variables.
Thus we lose one derivative (in space) by
performing this procedure. For this reason, we take a different approach.

We will construct solutions
 using a time-stepping procedure  in the Lagrangian coordinates in  physical space.
 In the rest of the section we will give some motivation for the time-stepping procedure
 to be used in the next sections.

System  (\ref{1.1})-(\ref{1.4}) can be written in Lagrangian coordinates as following.
First we use (\ref{1.1}) and write the equation (\ref{1.2})-(\ref{1.3}) as
\begin{equation}\label{1.19}
D_t\mathbf{u}_g-fJ\mathbf{u}_g+fJ\mathbf{u}=0.
\end{equation}
 Denote by $\phi(t,x)$ the flow map generated by $\mathbf{u}$. Then
$\phi(t,x)$ satisfies
\begin{equation*}
\begin{split}
&\partial_t \phi(t,x) = \mathbf{u}(t, \phi(t,x))\IN \bR\times\bR^2,\\
&\phi(0,x)=x \ON \bR^2.
\end{split}
\end{equation*}
Then from standard ODE theory we can see $\phi(t,x+h)=\phi(t,x)+h$ for any $h\in\bZ^2$. Equation (\ref{1.4}) implies that for each $t$ the map $\phi(t,\cdot)$
is  Lebesgue measure preserving: $\phi(t,\cdot)_\#\mathcal{L}^2_{\bR^2} =\mathcal{L}^2_{\bR^2}$,
where the left-hand side denotes the push-forward of the Lebesgue measure
$\mathcal{L}^2_{\bR^2}$ by the map $\phi(t,\cdot)$.
Express the geostrophic wind velocity in Lagrangian variables:
\begin{equation}\label{LagrGeostrVel}
\mathbf{v}_g(t,x_0)=\mathbf{u}_g(t,\phi(t,x_0))
\FOR (t, x_0)\in \bR_+\times\bR^2,
\end{equation}
where $x_0=(x_{01}, x_{02})$ is the spatial coordinate at time $t=0$. Then $\mathbf{v}_g$ is periodic with respect to $\bZ^2$ since $\mathbf{u}_g$ is assumed to be periodic in spatial arguments.
Now the equation (\ref{1.19}) can be written as
\begin{equation}\label{1.20}
\partial_t\mathbf{v}_g(t,x_0)-f(\phi(t,x_0))J\mathbf{v}_g(t,x_0)
+f(\phi(t,x_0))J\partial_t\phi(t,x_0)=0.
\end{equation}

We thus have rewritten system (\ref{1.1})--(\ref{1.4}) in the following Lagrangian
form: for $T>0$, find a function $p\in C^1([0,T)\times\bT^2)$ and a family of maps
$\phi\in C^1([0,T)\times\bR^2;\,\bR^2)$ such that:
\begin{equation}\label{lagrangianSG}
\begin{split}
& \textrm{Equation (\ref{1.20}) holds on $(0,T)\times\bT^2$
with $\mathbf{u}_g$,
$\mathbf{v}_g$ defined by (\ref{1.1}), (\ref{LagrGeostrVel})};\\
&\phi(t, \cdot)_\# \mathcal{L}^2_{\bR^2}=\mathcal{L}^2_{\bR^2} \FORA t\in (0,T),\,\phi(t,x+h)=\phi(t,x)+h\textrm{ for any $h\in\bZ^2$};\\
& p(0, x)=p_0(x),\;\;\phi(0,x)=x \ON \bR^2,
\end{split}
\end{equation}
where $p_0(x)$ is a given periodic function. For sufficiently small
$T>0$, we will find a smooth solution of (\ref{lagrangianSG})
such that $\phi(t, \cdot):\bR^2\to \bR^2$ is a diffeomorphism for each $t\in[0, T]$.
This determines a solution of (\ref{lagrangianSG}) by defining
$\mathbf{u}(t, x)=(\partial_t \phi)(t, \phi_t^{-1}(x))$,
where $\phi_t(\cdot)\defd\phi(t, \cdot)$.

In the rest of this section, we briefly describe the plan of the paper.

In Section \ref{TimeSteppingProcSect} we define
a time-stepping approximation of the system (\ref{lagrangianSG}).
For a time step size $\delta t>0$ and $n=0, 1,\dots, N$, a periodic function
$p_n$ on $\bR^2$ is an approximation of $p(n\delta t, \cdot)$,
and a measure-preserving map
$F_{n+1}: \bR^2 \to \bR^2$ with $F(\cdot+h)=F(\cdot)+h$ for any $h\in\bZ^2$ is an approximation of the flow map
 connecting time steps $n\delta t$ and $(n+1)\delta t$.
Then $p_0$ is given by the initial data.
On $n$-th step of  iteration, assuming that $p_n$ is known, we
define equations for  $F_{n+1}$ and $p_{n+1}$.
Equation for $p_{n+1}$ is of Monge-Ampere type.

In Section \ref{constMaps-Sect}, we show that for any $p_n,p_{n+1}$ close
enough to $p_0$ in $C^{2,\alpha}$ norm, and $\delta t$ small depending
on $p_0$, it is always possible to define a map $F_{n+1}$ which is a $C^{1,\alpha}$
diffeomorphism and satisfies equation (\ref{1.30}), and such a map is unique among all maps
close enough to the identity.

In Sections \ref{constApprSol-Sect}, \ref{constApprSol-1-Sect} we show
that if the step
size $\delta t$ is sufficiently small, then
for any $p_n$ which is close $p_0$ in $C^{3,\alpha}$-norm,
 we can find $p_{n+1}$ close to $p_0$ in
$C^{2,\alpha/2}$, such that the map defined in Section \ref{constMaps-Sect} is measure
preserving. This is done by solving
the iteration equation of Monge-Ampere type using
the implicit function theorem. Also we establish an estimate which shows
that if $p_0\in C^{k+2,\alpha}$, $k\ge 1$,
 then $p_{n+1}$ will remain
bounded in $C^{k+2,\alpha}$ and will
be $C\delta t$ close to $p_n$ in $C^{k+1,\alpha}$. This allows to define
time-stepping solution on time interval $[0, T]$, independent of (small) $\delta t$.

In Section \ref{passageLim-Sect}, we pass time-stepping solutions to the limit
as
$\delta t \to 0$,
 and show that the limit is
  a smooth solution of system  (\ref{1.1})-(\ref{1.4}) in Lagrangian coordinates.
  Because of smoothness, there is no difficulty to transform to Eulerian coordinates.

In Section \ref{varCoriolisPar-sect}, we extend above discussion to the SG
shallow water case.

In Section \ref{uniq-Sect}, we prove uniqueness of solutions, both for SG and SG shallow water equations, under the assumptions of Theorem \ref{1.5} and \ref{1.9} respectively.

\section{Time-stepping in Lagrangian coordinates}
\label{TimeSteppingProcSect}
In this section
 we define a time-stepping approximation of the system (\ref{lagrangianSG}).
We first give a heuristic motivation for the equations defined below.

In the following argument it will be more convenient to work with periodic functions
on $\bR^2$, instead of functions on $\bT^2$.

Discretize the time at $t_0$ with step size $\delta t$. Then the time difference equation corresponding to (\ref{1.20}) is
\begin{equation}\label{1.21}
\begin{split}
&\mathbf{v}_g(t_0+\delta t,x_0)-\mathbf{v}_g(t_0,x_0)-
f(\phi(t_0, x_0))J\mathbf{v}_g(t_0,x_0)\delta t\\
&\qquad\qquad+f(\phi(t_0,x_0))J(\phi(t_0+\delta t,x_0)-\phi(t_0,x_0))=0.
\end{split}
\end{equation}
On the other hand, we have
\begin{equation}\label{1.22}
R_{f(\phi(t_0,x_0))\delta t}\mathbf{v}_g(t_0, x_0)=\mathbf{v}_g(t_0, x_0)+f(\phi(t_0, x_0))J\mathbf{v}_g(t_0,x_0)\delta t+O(\delta t^2)
\end{equation}
where
\begin{equation}\label{1.22pr}
R_{a}=
\left( \begin{array}{cc}
\cos a &-\sin a  \\
\sin a& \cos a  \\
\end{array} \right)
\end{equation}
is the matrix defining a rotation by angle $a$.
Then we can replace (\ref{1.21}) with
\begin{equation}\label{1.23}
\mathbf{v}_g(t_0+\delta t,x_0)-R_{f\delta t}\mathbf{v}_g(t_0,x_0)+f(\phi(t_0,x_0))J(\phi(t_0+\delta t,x_0)-\phi(t_0,x_0))=0,
\end{equation}
where $R_{f\delta t}=R_{f(\phi(t_0,x_0))\delta t}$.
Write the flow map from time $t_0$ to $t_0+\delta t$ as $F$, then $\phi(t_0+\delta t,x_0)=F\circ\phi(t_0, x_0)$. Write $x=\phi(t_0+\delta t)$,
then $\phi(t_0, x_0)=F^{-1}(x)$. With these notations and recalling (\ref{LagrGeostrVel}), equation (\ref{1.23})
becomes
\begin{equation}\label{1.24}
\mathbf{u}_g(t_0+\delta t,x)-R_{f\delta t}\mathbf{u}_g(t_0,F^{-1}(x))+f(F^{-1}(x))J(x-F^{-1}(x))=0,
\end{equation}
where in $R_{f\delta t}$ function $f$ is evaluated  at $F^{-1}(x)$.
Recalling that $\mathbf{u}_g=f^{-1}J\nabla p$ and noting that
$R_{f\delta t}J=J R_{f\delta t}$, we obtain
\begin{equation}\label{1.25}
f^{-1}(x)J\nabla p(t_0+\delta t,x)-f^{-1}JR_{f\delta t}\nabla p(t_0, F^{-1}(x))+f(F^{-1}(x))J(x-F^{-1}(x))=0.
\end{equation}
In the second term on the left hand side above, $f$ is evaluated at $F^{-1}(x)$.

Let $t_0=n\delta t$, and write $p_{n+1}=p((n+1)\delta t),p_n=p(n\delta t)$, and $F_{n+1}$ for the flow map connecting time step $n\delta t$ and $(n+1)\delta t$, we obtain from (\ref{1.25}):
\begin{equation}\label{1.30}
x+f^{-1}(x)f^{-1}(F_{n+1}^{-1}(x))\nabla p_{n+1}(x)=F_{n+1}^{-1}(x)+(f^{-2}R_{f\delta t}\nabla p_n)(F_{n+1}^{-1}(x)).
\end{equation}
In the second term of the right-hand side, all functions are evaluated at $F_{n+1}^{-1}(x)$.

We require the map $F_{n+1}^{-1}$ to be a measure preserving
 diffeomorphism of $\bT^2$.

Next we will set up the iteration scheme based on the ideas described above. Let $p_0$ be the
initial data.
Then we define $p_1, p_2, \dots, N$ inductively as following.
Let $n\in\{0, 1,2\dots, N\}$, and a function $p_n$ is given. We look for a function $p_{n+1}$ and a measure
preserving map $F_{n+1}$ such that (\ref{1.30}) holds.
Since we want $F_{n+1}$ to be measure preserving, we take gradient on both sides of (\ref{1.30}), and collect terms involving $DF_{n+1}^{-1}$:
\begin{equation}\label{1.31}
\begin{split}
&\big(I+f^{-1}\nabla p_{n+1}\otimes\nabla(f^{-1})(x)+f^{-2}D^2p_{n+1}\big)(x)
+A_{n+1}(x)\\
&\qquad=[\big(I+f^{-1}\nabla p_n\otimes\nabla (f^{-1})+
f^{-2}D^2p_n\big)(F_{n+1}^{-1}(x))
+B_{n+1}(x)]DF_{n+1}^{-1}(x),
\end{split}
\end{equation}
where
\begin{align}\label{1.32}
\begin{split}
A_{n+1}(x)=&[f^{-1}(F_{n+1}^{-1}(x))-f^{-1}(x)]\nabla p_{n+1}(x)\otimes
\nabla (f^{-1})(x)\\
&\qquad+f^{-1}(x)(f^{-1}(F_{n+1}^{-1}(x))-f^{-1}(x))D^2p_{n+1}(x),
\end{split}
\\
\label{1.33}
\begin{split}
B_{n+1}(x)=&[(f^{-1}\nabla p_n)(F_{n+1}^{-1}(x))-(f^{-1}\nabla p_{n+1})(x)]\otimes\nabla(f^{-1})(F_{n+1}^{-1}(x))\\
&\qquad +D[f^{-2}(R_{f\delta t}-I)\nabla p_n](F_{n+1}^{-1}(x)).
\end{split}
\end{align}
Taking determinant on both sides of (\ref{1.31}), we see $\det DF_{n+1}^{-1}\equiv1$ if and only if
\begin{equation}\label{1.34}
\begin{split}
&\det[\big(I+f^{-1}\nabla p_{n+1}\otimes\nabla(f^{-1})
+f^{-2}D^2p_{n+1}\big)(x)+A_{n+1}(x)]\\
&\qquad=\det[\big(I+f^{-1}\nabla p_n\otimes\nabla (f^{-1})
+f^{-2}D^2p_n\big)(F_{n+1}^{-1}(x))+B_{n+1}(x)].
\end{split}
\end{equation}
Number $N$ will be defined below so that $N\delta t$ is small and $p_1,\dots, p_N$ are close to
$p_0$ in the norms specified below.
After (\ref{1.30}), (\ref{1.31})--(\ref{1.34}) is solved for $n=1, \dots, N$,
we define
approximate solution $(p_{\delta t}, \phi^{\delta t})$
 of (\ref{lagrangianSG})
 with step size $\delta t$ to be $p_{\delta t}(t)=p_n$ if $t\in[n\delta t,(n+1)\delta t)$, the approximate flow map $\phi^{\delta t}(t)=F_{n}\circ F_{n-1}\circ\cdots F_1$, for $t\in[n\delta t,(n+1)\delta t)$.

In the following, to simplify notations, we write $q(x)=p_{n+1}(x)$, $p(x)=p_n(x)$, and $F=F_{n+1}$,
 $A(x)=A_{n+1}(x)$, $B(x)=B_{n+1}(x)$ for functions and maps used in (\ref{1.31})--(\ref{1.34}).
 In the present notations (\ref{1.30}) becomes
\begin{equation}\label{3.2}
x+f^{-1}(x)f^{-1}(F^{-1}(x))\nabla q(x)=F^{-1}(x)+
(f^{-2}R_{f\delta t}\nabla p)(F^{-1}(x)).
\end{equation}
Here, in the last term, all functions are evaluated at $F^{-1}(x)$.
Equation (\ref{1.34}) in the present notation is the following:
\begin{equation}\label{3.3}
\begin{split}
&\det[\big(I+f^{-1}\nabla q\otimes\nabla(f^{-1})
+f^{-2}D^2 q\big)(x)+A(x)]\\
&\qquad=\det[\big(I+f^{-1}\nabla p\otimes\nabla (f^{-1})
+f^{-2}D^2p\big)(F^{-1}(x))+B(x)],
\end{split}
\end{equation}
where the expressions of $A(x)$, $B(x)$ are
given by (\ref{1.32}), (\ref{1.33}) with $p_n=p$, $p_{n+1}=q$, $F_{n+1}=F$.

In next two sections, for a given $\pN$ which is close to $p_0$ in $C^{2,\alpha}$ and small
$\delta t>0$,
 we find $q$
and $F$ which satisfy (\ref{3.2}), (\ref{3.3}). Here $\alpha\in (0,1)$ is fixed from now on.

\section{Construction of maps}
\label{constMaps-Sect}
Let $p_0\in C^{3,\alpha}(\mathbf{\bT^2})$ satisfy $\int_{\bT^2}p_0(x)dx=0$
and (\ref{1.6}).
In this section we show that, given $\pN,\pNP$ which are
 close to $p_0$  and small $\delta t>0$, the map $\FNP^{-1}$ satisfying
 (\ref{3.2}) exists and is invertible.  For this we use implicit function theorem.

We continue to work with periodic functions on $\bR^2$, instead of working directly on $\bT^2$.
Then, for $k=0,1,\dots$ and $\alpha\in (0, 1)$, we denote  by $C^{k,\alpha}(\bT^2)$ the space of functions
$\varphi:\bR^2\rightarrow\bR$ which are in $C^{k,\alpha}_{loc}(\bR^2)$
 and  $\bZ^2$-periodic:
 $$
C^{k,\alpha}(\bT^2)=\{\varphi\in C^{k,\alpha}_{loc}(\bR^2)\;|\;
\varphi(x+h)=\varphi(x) \FORA h\in\bZ^2\}.
 $$
 Then $C^{k,\alpha}(\bT^2)$ is a closed subspace of $C^{k,\alpha}(\bR^2)$,
 thus $C^{k,\alpha}(\bT^2)$ with norm $\|\varphi\|_{k,\alpha}\defd \|\varphi\|_{k,\alpha, \bR^2}$
 is a Banach space.

 Let $U_1\subset C^{2,\alpha}$ be the open subset defined by:
\begin{equation}\label{2.1}
U_1=\{p\in C^{2,\alpha}:I+f^{-1}D(f^{-1}Dp)>\frac{c_0}{2}I\}.
\end{equation}

Also, we denote by $C^{k,\alpha}(\bT^2; \bR^2)$ the space
 of mappings $w:\bR^2\rightarrow\bR^2$ which are in $C^{k,\alpha}_{loc}(\bR^2; \bR^2)$
 and  $\bZ^2$-periodic:
\begin{equation}\label{periodicPlusId-set}
C^{k,\alpha}(\bT^2; \bR^2)=\{w\in C^{k,\alpha}_{loc}(\bR^2; \bR^2)\;|\;
w(x+h)=w(x) \FORA x\in\bR^2,\,h\in\bZ^2\}.
\end{equation}
 We also consider mappings $w:\bR^2\rightarrow\bR^2$ with the following periodicity
 property:
  $$
C^{k,\alpha}_{p}(\bR^2; \bR^2)=\{w\in C^{k,\alpha}_{loc}(\bR^2; \bR^2)\;|\;
w(x+h)=w(x)+h \FORA h\in\bZ^2\}.
 $$
 Note that $C^{k,\alpha}_{p}(\bR^2; \bR^2)$ is not a subspace. We also note  that
\begin{equation}\label{periodicPlusId-set-shift}
C^{k,\alpha}_{p}(\bR^2; \bR^2)=Id + C^{k,\alpha}(\bT^2; \bR^2),
\end{equation}
 where $Id$ is the identity map $Id(x)=x$ in $\bR^2$. Indeed, if $w(x+h)=w(x)+h$ for all
 $x\in\bR^2$, $h\in\bZ^2$ if and only if $v\defd w(x)-x$ is $\bZ^2$-periodic.
The reason to introduce $C^{k,\alpha}_{p}(\bR^2; \bR^2)$ is the following: if
$\pN, \pNP$ are periodic, and $\|\pN -\pNP\|_{2,\alpha}$ and $\delta t$ are small,
then the map $z=F^{-1}$ solving (\ref{3.2}) and close to $Id$, satisfies
$z\in C^{k,\alpha}_{p}(\bR^2; \bR^2)$, see Lemma \ref{perOfMap} below.

We first rewrite equation  (\ref{3.2}) as following.

For fixed $\pN, \pNP\in C^2(\bT^2)$, and $ \delta t\in(-1,1)$ , consider  the map
\begin{align}\nonumber
&\hat Q_{\pN, \pNP, \delta t}=\hat Q:\bR^2\times\bR^2\rightarrow\bR^2\;\mbox{ defined by}\\
\label{2.2-11}
&\hat Q(x,w)=
x+f^{-1}(x)f^{-1}(w)\nabla \pNP(x)-w-\big(f^{-2}R_{f\delta t}\nabla \pN\big)(w).
\end{align}

Solving (\ref{3.2}) for $F$, with given $\pN, \pNP, \delta t$, is equivalent to solving
\begin{equation}\label{5.10-eq0}
\hat Q_{\pN, \pNP, \delta t}(x,z) = 0
\end{equation}
for $z$ for each $x\in \bR^2$, then $F^{-1}(x)=z(x)$.
Also, we note that for any $\pN\in C^{2,\alpha}(\bT^2)$
\begin{equation}\label{5.10-eq0-Zero}
\hat Q_{\pN, \pN, 0}(x,x)=0  \FORA x\in\bR^2,
\end{equation}
which is obtained directly from (\ref{2.2-11}) using that $R_0=I$ in (\ref{1.22pr}).
Thus we expect that $z(x)-x$ is small if $\|p-q\|_{2,\alpha}$ and $\delta t$ are small.

In next lemma we use the set $U_1$ defined by (\ref{2.1}):
\begin{lem}\label{injectivity}
For any $\pN_0\in U_1$ there exists $\eps>0$ such that for any
$\pN,\pNP\in C^{2,\alpha}(\bT^2)$ satisfying $\|\pN -\pN_0\|_{2,\alpha,\bR^2}\le \eps$,
$\|\pNP -\pN_0\|_{2,\alpha,\bR^2}\le \eps$, and
any $\delta t\in (-\eps, \eps)$
\begin{enumerate}[(i)]
\item\label{injectivity-i1}
$D_w\hat Q_{\pN, \pNP, \delta t}(x,w)<-\frac{c_0}{4}I$
if $ |x-w|<\eps$;
\item\label{injectivity-i2}
For any $x\in\bR^2$, map
$\hat Q_{\pN, \pNP, \delta t}(x,\cdot): \bR^2\to \bR^2$ is injective on $B_\eps(x)$.
\end{enumerate}
\end{lem}
\begin{proof}
We note first that for $\pNP=\pN=\pN_0$ and $\delta t=0$, we get for any $x\in \bR^2$:
\begin{equation*}
\begin{split}
D_w\hat Q_{\pN_0, \pN_0, 0}(x,x)&=
-I-\big(f^{-1}\nabla \pN_0\otimes\nabla (f^{-1})+f^{-2}D^2\pN_0\big)(x)\\
&= -I-\big(f^{-1}D(f^{-1}D\pN_0)\big)(x)
<-\frac{c_0}{2}I.
\end{split}
\end{equation*}

Now consider any $\pN$, $\pNP$, $\delta t$ satisfying conditions of Lemma,
and any $x,w\in \bR^2$ with $|x-w|<\eps$. Then, assuming that $\eps\in (0,1)$,
we have $\|\pN\|_{2,\alpha}, \|\pNP\|_{2,\alpha} \le \|\pN_0\|_{2,\alpha}+1$, thus
we get:
\begin{equation*}
\begin{split}
|D_w\hat Q_{\pN_0, \pN_0, 0}(x,x)- D_w\hat Q_{\pN, \pNP, \delta t}(x,w)|
&\le
|D_w\hat Q_{\pN_0, \pN_0, 0}(x,x)- D_w\hat Q_{\pN, \pN_0, \delta t}(x,w)|\\
&\qquad+
|D_w\hat Q_{\pN, \pN_0, \delta t}(x,w)- D_w\hat Q_{\pN, \pNP, \delta t}(x,w)|
\\
&\le C(\|\pN-\pN_0\|_{2,0}+\|\pNP-\pN_0\|_{1,0}+|x-w|^\alpha+|\delta t|)\\
&\le C\eps^\alpha,
\end{split}
\end{equation*}
where $C$ depends on $\|f^{-1}\|_{C^{1,\alpha}(\bT^2)}$ and
$\|\pN_0\|_{C^{2,\alpha}(\bT^2)}$,
and $C$ may be different in different occurrences.
Thus, choosing $\eps$ small depending only on $\|f^{-1}\|_{C^{1,\alpha}(\bT^2)}$
and  $\|\pN_0\|_{C^{2,\alpha}(\bT^2)}$, we
get assertion (\ref{injectivity-i1}) of Lemma.

Now we prove  assertion (\ref{injectivity-i2}) of Lemma.
For $w, \hat w \in  B_\eps(x)$ with $w\ne \hat w$
we have $\tau w + (1-\tau) \hat w \in  B_\eps(x)$ for any $\tau\in[0,1]$, and then
denoting
${\bf e}\defd w-\hat w$,
we get ${\bf e}\ne 0$ and thus
\begin{equation*}
\begin{split}
\big(\hat Q(x,w)-\hat Q(x,\hat w)\big)\cdot {\bf e}
=\int_0^1 \big(D_w \hat Q(x, \tau w + (1-\tau) \hat w){\bf e}\big) \cdot {\bf e} \,d\tau
\le -\frac{c_0}{4} |{\bf e}|^2<0.
\end{split}
\end{equation*}

\end{proof}

Next we show
that solutions $z=F^{-1}$ of (\ref{3.2}), which are close to the identity map,
lie in the set $C^{1,\alpha}_{p}(\bR^2; \bR^2)$. 
We first note the property
\begin{equation}\label{5.10s-hat}
\hat Q(x+k, w+h)=\hat Q(x,w)+k-h \textrm{  for any $x, w\in\bR^2$, $h,k\in\bZ^2$},
\end{equation}
which follows from (\ref{2.2-11}).
\begin{lem}\label{perOfMap}
For any $\pN_0\in U_1$ there exists $\eps>0$ such that if
$\pN, \pNP\in C^{2,\alpha}(\bT^2)$, $z:\bR^2\to\bR^2$ and $\delta t\in (-\eps, \eps)$
satisfy (\ref{3.2}) with  $F^{-1}\defd z$, and also satisfy
$\|\pN -\pN_0\|_{2,\alpha,\bR^2}\le \eps$, $\|\pNP -\pN_0\|_{2,\alpha,\bR^2}\le \eps$,
$\|z- Id\|_{L^\infty(\bR^2)}< \eps$,
$|\delta t|<\eps$, then $z\in C^{1,\alpha}_{p}(\bR^2; \bR^2)$.
\end{lem}
\begin{proof}
From (\ref{5.10s-hat}), if $\hat Q(x,z(x))=0$, then
 $\hat Q(x+h, z(x)+h)=\hat Q(x,z(x))=0$  for any $x\in\bR^2$, $h\in\bZ^2$.
Combined with property $|z(x)-x|<\eps$ and injectivity
of $\hat Q(x,\cdot)$ on the $B_\eps(x)$ shown in
Lemma \ref{injectivity}(\ref{injectivity-i2}),
we obtain $z(x+h)=z(x)+h$.

Finally, the fact $z\in C^{1,\alpha}_{loc}(\bR^2; \bR^2)$ follows from
the Implicit Function Theorem applied to the equation $\hat Q(x,w)=0$, using
nondegeneracy of $D_w \hat Q(x,z(x))$ for any  $x$ which follows
from Lemma \ref{injectivity}(\ref{injectivity-i1})
since $|z(x)-x|<\eps$, and using regularity
$\hat Q\in C^{1,\alpha}_{loc}(\bR^2\times\bR^2; \bR^2)$ which follows from (\ref{2.2-11})
for $\pN, \pNP\in C^{2,\alpha}$.
\end{proof}

Now we show that (\ref{3.2}) has a solution $F^{-1}=z\in
C^{1,\alpha/2}_p(\bR^2; \bR^2)$. For that, we use
Implicit Function Theorem in the following spaces.
 Define a map
\begin{align}\label{2.2-pr-00}
&Q:C^{3,\alpha}(\bT^2)\times C^{2,\alpha/2}(\bT^2)\times C^{1,\alpha/2}_p(\bR^2; \bR^2)
\times(-1,1)
\rightarrow C^{1,\alpha/2}(\bT^2; \bR^2)
\quad\text{ by}\\
\label{2.2-pr-00-1}
&Q(\pN,\pNP, z, \delta t)(x)=\hat Q_{\pN, \pNP, \delta t}(x,z(x)).
\end{align}
Thus,  $Q$ is given by the expression:
\begin{align}
\label{2.2}
&Q(\pN,\pNP, z, \delta t)(x)=
x+f^{-1}(x)f^{-1}(z(x))\nabla \pNP(x)-z(x)-\big(f^{-2}R_{f\delta t}\nabla \pN\big)(z(x)).
\end{align}
The fact that $Q$ in (\ref{2.2}) acts into $C^{1,\alpha}(\bT^2; \bR^2)$
is seen as following: Regularity
$Q(\pN,\pNP, z, \delta t)(\cdot)\in C^{1,\alpha}_{loc}(\bR^2; \bR^2)$
follows directly from the choice of spaces in the domain of $Q$ and the explicit
expression (\ref{2.2}).  The $\bZ^2$-periodicity
of
$Q(\pN,\pNP, z, \delta t)(\cdot)$  follows from
the property
\begin{equation}\label{5.10s}
Q(\pN,\pNP, z+h, \delta t)(x+k)=Q(\pN,\pNP, z, \delta t)(x)+h-k \textrm{  for any }
x\in\bR^2,\;h,k\in\bZ^2,
\end{equation}
where (\ref{5.10s}) follows from (\ref{2.2})
using $\bZ^2$-periodicity of $f$, $\pN$, $\pNP$ and property (\ref{periodicPlusId-set-shift})
for $z\in C^{1,\alpha}_{p}(\bR^2; \bR^2)$.

Given $(\pN, \pNP, \delta t)$, solving (\ref{3.2}) for $F^{-1}$  is equivalent to
solving
\begin{equation}\label{5.10-eq}
Q(\pN,\pNP, z, \delta t) = 0
\end{equation}
for $z$, then $F^{-1}=z$ solves (\ref{3.2}).
From (\ref{5.10-eq0-Zero}) and (\ref{2.2-pr-00-1}), we have
 for any $\pN\in C^{3,\alpha}(\bT^2)$
\begin{equation}\label{5.10-eq-Zero}
Q(\pN,\pN, Id, 0)=0,
\end{equation}
where $Id$ is the identity map in $\bR^2$. Then
we will solve (\ref{5.10-eq}) for $z(\cdot)$ when $\pN\in U_1\cap C^{3,\alpha}$, and
when
$\|\pN-\pNP\|_{2,\alpha}$ and $\delta t$ are small.

Since the set of functions
$C^{1,\alpha}_p(\bR^2; \bR^2)$ is not a space, it is convenient to replace $z(\cdot)$
by $w(x)=z(x)-x$ in $Q$, since then $w\in  C^{1,\alpha}(\bT^2; \bR^2)$
by (\ref{periodicPlusId-set-shift}). Thus we define
\begin{equation}\label{2.2-pr-Q1}
\begin{split}
&Q_1:C^{3,\alpha}(\bT^2)\times C^{2,\alpha/2}(\bT^2)\times C^{1, \alpha/2}(\bT^2; \bR^2)
\times(-1,1)
\rightarrow C^{1,\alpha/2}(\bT^2; \bR^2)\\
&\mbox{by }\; Q_1(\pN,\pNP, w, \delta t)=Q(\pN,\pNP, w+Id, \delta t),
\end{split}
\end{equation}
that is
\begin{equation}\label{2.2-Q1}
\begin{split}
Q_1(\pN,\pNP, w, \delta t)(x)=&
f^{-1}(x)f^{-1}(x+w(x))\nabla \pNP(x)-w(x)\\
&\quad-\big(f^{-2}R_{f\delta t}\nabla \pN\big)(x+w(x)).
\end{split}
\end{equation}

Expressing equation (\ref{5.10-eq}) in terms of $Q_1$, we see that
solving (\ref{3.2}) for $F$, for a given $\pN, \pNP, \delta t$, is equivalent to solving
\begin{equation}\label{5.10-eq-Q1}
Q_1(\pN,\pNP, w, \delta t) = 0
\end{equation}
for $w$, then $F=(Id+w)^{-1}$.
From (\ref{5.10-eq-Zero}), for any $\pN\in C^{3,\alpha}(\bT^2)$
\begin{equation}\label{5.10-eq-Zero-Q1}
Q_1(\pN,\pN, w_0, 0)=0,
\end{equation}
where $w_0$ is the zero map in $\bR^2$, i.e. $w_0:\bR^2\to\bR^2$ is given by $w_0(x)=0$.
Then
we will solve (\ref{5.10-eq-Q1}) for $w(\cdot)$ with small
$\|w\|_{2,\alpha}$, when $\pN\in U_1\cap C^{3,\alpha}$, and when
$\|\pN-\pNP\|_{2,\alpha/2}$ and $\delta t$ are small.

Now, in order to solve (\ref{5.10-eq-Q1}), we will apply Implicit Function Theorem
in spaces given in (\ref{2.2-pr-Q1}), near the background solution
given in (\ref{5.10-eq-Zero-Q1}).

For that we first note that the higher regularity of $\pN$ implies that the map $Q_1$
is smooth:
\begin{lem}\label{MapForMapDifferentiable}
Map
\begin{align}\label{2.2-pr-1}
&Q_1:C^{3,\alpha}(\bT^2)\times C^{2,\alpha/2}(\bT^2)\times C^{1,\alpha/2}(\bT^2; \bR^2)
\times(-1,1)
\rightarrow C^{1,\alpha/2}(\bT^2; \bR^2),
\end{align}
defined by (\ref{2.2}), is continuously Frechet-differentiable.
\end{lem}
\begin{proof}
Lemma follows directly from the expression (\ref{2.2}) and Lemma \ref{CompositionFrechetDif},
proved in Appendix.
\end{proof}

Now we prove existence of solution of $(\ref{5.10-eq-Q1})$ near $(\pN,\pN, w_0, 0)$.
For $\pN_0\in C^{3,\alpha}(\bT^2)$ and $\eps>0$, denote:
\begin{equation}\label{V-W-def}
\begin{split}
V_{\eps}(\pN_0)&\defd\{(\pN, \pNP)\in  C^{3,\alpha}(\bT^2)\times C^{2,\alpha/2}(\bT^2)\;|
\;\|\pN-\pN_0\|_{3,\alpha},
\|\pNP-\pN_0\|_{2,\alpha/2}<\eps\}\\
&\subset  C^{3,\alpha}(\bT^2)\times C^{2,\alpha/2}(\bT^2);\\
W_{\eps}&\defd\{w\in C^{1,\alpha/2}(\bT^2; \bR^2)\;|
\;\|w\|_{1,\alpha/2}<\eps\}\subset C^{1,\alpha/2}(\bT^2; \bR^2).
\end{split}
\end{equation}

\begin{lem}\label{2.11}
For any $\pN_0\in U_1\cap C^{3,\alpha}(\bT^2)$ there exist $\eps_1, \epsTT>0$ such that
for any $(\pN,\pNP, \delta t)\in V_{\eps_1}(\pN_0)\times(-\eps_1, \eps_1)$
there exists a unique $w\in W_{\epsTT}$ such that $(\pN,\pNP,w, \delta t)$
satisfy (\ref{5.10-eq-Q1}). The map
$\Mp:V_{\eps_1}(\pN_0)\times(-\eps_1, \eps_1)\to W_{\epsTT}$, defined by
$\Mp(\pN, \pNP, \delta t)=w$, is continuously Frechet-differentiable.
\end{lem}
\begin{proof}
This follows directly from the Implicit Function Theorem in Banach spaces.
Indeed, by Lemma \ref{MapForMapDifferentiable}, the map $Q_1$ is
continuously Frechet-differentiable, and (\ref{5.10-eq-Zero-Q1}) holds.
Using (\ref{2.2-Q1}), we find that the linear map
\begin{equation}\label{Q1-W-Deriv}
D_w Q_1(\pN_0, \pN_0, w_0, 0):
C^{1,\alpha/2}(\bT^2; \bR^2)
\rightarrow C^{1,\alpha/2}(\bT^2; \bR^2)
\end{equation}
is given, for $h\in C^{1,\alpha/2}(\bT^2; \bR^2)$, by
\begin{equation*}
\begin{split}
\big(D_w Q_1(\pN_0, \pN_0, w_0, 0)\big) h&=
\big(-I-\big(f^{-1}\nabla \pN_0\otimes\nabla (f^{-1})+f^{-2}D^2\pN_0\big)\big)h\\
&= \big(-I-\big(f^{-1}D(f^{-1}D\pN_0)\big)\big)h.
\end{split}
\end{equation*}
Since the matrix $\big(-I-\big(f^{-1}D(f^{-1}D\pN_0)\big)\big)(x)$ is nondegenerate for
each $x\in \bR^2$ (which holds because $\pN\in U_1$),
and since $-I-\big(f^{-1}D(f^{-1}D\pN_0)\big) \in C^{1,\alpha}(\bT^2; \bR^2)$,
it follows that the map (\ref{Q1-W-Deriv}) is a linear isomorphism.
Now the lemma follows from the Implicit Function Theorem.
\end{proof}
\begin{rem}\label{identitySolnRemark}
From Lemma \ref{2.11} and (\ref{5.10-eq-Zero-Q1}), it follows that
$$
\Mp(\pN, \pN, 0)=w_0 \FORA (\pN, \pN)\in V_{\eps_1}(\pN_0).
$$
Then, since $\Mp:V_{\eps_1}(\pN_0)\times(-\eps_1, \eps_1)\to W_{\epsTT}$
is continuous, it follows that for any $\epsTT'\in (0, \epsTT)$
there exists $\eps_1'\in (0, \eps_1)$ such that
$\Mp(\pN, \pNP,\delta t)\in W_{\epsTT'}$ if
$(\pN, \pNP,\delta t)\in V_{\eps_1'}(\pN_0)\times(-\eps_1', \eps_1')$.
\end{rem}
\begin{lem}\label{mapIsDiffeo}
For any $\pN_0\in U_1\cap C^{3,\alpha}(\bT^2)$ there exist $\eps_1'\in (0, \eps_1]$
such that
for each $(\pN,\pNP, \delta t)\in V_{\hat\eps_1'}(\pN_0)\times(-\hat\eps_1', \hat\eps_1')$
the map
$z=Id+\Mp(\pN,\pNP, \delta t):\bR^2\to\bR^2$ is a diffeomorphism.
\end{lem}
\begin{proof}
Let $\eps_1$ and $\epsTT$ are so small that the map $\Mp$ is defined by
 Lemma \ref{2.11}.

 Then, by (\ref{2.2-pr-00}) and (\ref{2.2-pr-Q1})
 \begin{equation}\label{5.10-eq0-11}
\hat Q_{\pN, \pNP, \delta t}(x,z(x)) = 0 \FORA x\in\bR^2.
\end{equation}

We show that $z(\bR^2)=\bR^2$
for each $(\pN,\pNP, \delta t)\in V_{\eps_1}(\pN_0)\times(-\eps_1, \eps_1)$.

First we note that,  after possibly reducing $\eps_1$, we have
 $z(\bR^2)$ is an open set.
Indeed, using Remark \ref{identitySolnRemark},
we find that for any
$\epsTT'\in (0, \epsTT)$
there exists $\eps_1'\in (0, \eps_1)$ such that, if
$(\pN, \pNP,\delta t)\in V_{\eps_1'}(\pN_0)\times(-\eps_1', \eps_1')$,
we get
\begin{equation}\label{5.10-eq0-111}
\|z-Id\|_{L^\infty(\bR^2)}= \|\Mp(\pN,\pNP, \delta t)\|_{L^\infty(\bR^2)}\le \epsTT'.
\end{equation}
Then choosing $\epsTT'$ smaller than $\eps/2$ in
Lemma \ref{injectivity}(\ref{injectivity-i1}) and choosing the corresponding
 $\eps_1'$, we get
$D_w\hat Q_{\pN, \pNP, \delta t}(x,z(x))<-\frac{c_0}{4}I$ for each $x\in \bR^2$
if $(\pN, \pNP,\delta t)\in V_{\eps_1'}(\pN_0)\times(-\eps_1', \eps_1')$.
Then, fixing $\hat x\in\bR^2$, we obtain by Implicit Function Theorem that
there exists a neighborhood $B_r(z(\hat x))$ of $z(\hat x)$,
where $r>0$, and a $C^{1,\alpha/2}$ map
$g: B_r(z(\hat x)) \to \bR^2$ with $g( z(\hat x))=\hat x$, such that
 \begin{equation}\label{5.10-eq0-12}
\hat Q_{\pN, \pNP, \delta t}(g(v),v) = 0 \FORA v\in {\mathcal N}.
\end{equation}
Since
$$\|g(z(\hat x)) - z(\hat x)\|=\|\hat x - z(\hat x)\|\le \epsTT'<\eps/2,$$
then reducing $r$,
we get
$\|g(v) - v\| <\eps$ for each $v\in B_r(z(\hat x))$.
Then, by (\ref{5.10-eq0-11}), (\ref{5.10-eq0-111}), (\ref{5.10-eq0-12}) and
Lemma \ref{injectivity}(\ref{injectivity-i2}), it follows that
$v=z(g(v)$, i.e. $B_r(z(\hat x))\subset z(\bR^2)$. Thus, the set $z(\bR^2)$ is open.
Also, from now on we set $\hat \eps_1$  to be equal to $\eps_1'$
chosen above.

Next, we show that the set $z(\bR^2)$ is closed. If $z(x_i)\to \hat v\in\bR^2$
for some points $x_i\in\bR^2$, then from (\ref{5.10-eq0-111}), it follows that
there exists a positive $N$ such that
$x_i\in B_{2\epsTT'}(\hat v)$ for all $i>N$. Thus there exists a convergent subsequence
$x_{i_j}\to \hat x\in\bR^2$. Since $z(\cdot)$ is continuous, then $z(\hat x)=\hat v$,
thus $z(\bR^2)$ is closed.

Now, $z(\bR^2)$ is an open, closed, and non-empty set, thus $z(\bR^2)=\bR^2$.
Also, by Lemma \ref{injectivity}(\ref{injectivity-i2}), $z(\cdot)$ is injective
on $\bR^2$. Thus the map $z^{-1}:\bR^2\to \bR^2$ is uniquely defined. Also, locally this
map is determined by the implicit function theorem as we discussed above:
$z^{-1}=g$ locally, where $g(\cdot)$ is from (\ref{5.10-eq0-12}).
Thus $z^{-1}\in C^{1,\alpha/2}_{loc}$.
\end{proof}


\section{Solving iteration equations}
\label{constApprSol-Sect}

Let $\hat\eps_1, \pNP,\pN, \delta t$ be as in Lemma \ref{mapIsDiffeo}.
Then we can define the map $F(x)=(Id+\Mp(\pN,\pNP, \delta t))^{-1}$,
so that $F^{-1}(x)=Id+\Mp(\pN,\pNP, \delta t)$. Then
$(\pNP,\pN, F^{-1}, \delta t$ satisfy
equation (\ref{3.2}),
by Lemma \ref{2.11} and (\ref{2.2-Q1}).

To make $F^{-1}$ measure preserving, we solve  equation (\ref{3.3}),
with $F^{-1}= Id+\Mp(\pN,\pNP, \delta t)$,  for $\pNP$.
We will use Implicit Function Theorem in the setting described below.
\begin{lem}\label{l5.1}
Let $\pN_0\in U_1$.
There exists $\epsTwo\in (0, \hat\eps_1)$, such that for any
$(p,q)\in V_{\epsTwo}(p_0)$, $|\delta t|<\epsTwo$, one has
$\|A\|_0, \|B\|_0<\frac{c_0}{8}$. Here $A(x),B(x)$ are given in (\ref{1.32}),(\ref{1.33}) with $p=p_n$, $q=p_{n+1}$,
and $F_{n+1}^{-1}=id+\Mp(p,q,\delta t)$.
\end{lem}

\begin{proof}
Let $0<\eps<\hat\eps_1$ be small and assume $(p,q)\in V_{\eps}(p_0)$. We first estimate the
matrix $A$. From (\ref{1.32}), for some constant $C$ which depends only on $f$ we get:
$$\|A\|_0\leq C\|q\|_{2,\alpha/2}\|\Mp\|_0\leq C(\|q_0\|_{2,\alpha/2}+1)\|\Mp\|_0.
$$
Next we can write the expression of $B$ in (\ref{1.33}) as
\begin{equation*}
\begin{split}
B=&\big[\big(f^{-1}\nabla (p-q)\big)(F^{-1}(x))+\big(f^{-1}\nabla q\big)(F^{-1}(x))
-\big(f^{-1}\nabla q\big)(x)\big]\otimes\nabla(f^{-1})(F^{-1}(x))\\
&+D[f^{-2}(R_{f\delta t}-I)\nabla p](F^{-1}(x)).
\end{split}
\end{equation*}
Since one has $\|q-p\|_{2,\alpha/2}\leq \epsTwo$, $|\delta t|\le \epsTwo$, one can estimate
\begin{equation*}
\begin{split}
\|B\|_0  &\leq  C[\|p-q\|_{C^1}+\|f\|_{C^1}\|q\|_{C^2}\|\Mp\|_0+|\delta t\|p\|_{C^2}],\\
& \leq C[\epsTwo+(\|p_0\|_{C^2}+1)(\|\Mp\|_0+\epsTwo)].
\end{split}
\end{equation*}
Now from Remark \ref{identitySolnRemark}, $\|\Mp(p,q,\delta t)\|_0$ can be made as small as we want as long
as $(p,q)\in V_{\epsTwo}(p_0)$ and $|\delta t|<\epsTwo$ with $\epsTwo$ chosen small enough. It follows that as long as $\eps$ is chosen small enough,
we can make $\|A\|_0,\|B\|_0<\frac{c_0}{8}$.
\end{proof}

 Denote
$$
C_0^{k,\alpha}\defd\{\varphi\in C^{k,\alpha}(\bT^2)\;:\;\int_{[0,1)^2}\varphi(x)\,dx= 0\}.
 $$
For the rest of this section, we
fix  $\pN_0\in U_1\cap C^{k+2,\alpha}$ for some $k\geq 2$.

Given $p$ near $p_0$ and small $\delta t$, we solve the equation (\ref{3.3}) (with $F$ defined by
$F^{-1}= Id+\Mp(\pN,\pNP, \delta t)$) for $q$
using implicit function theorem.
 Consider the following open subset $U_2$ of $C_0^{2,\alpha/2}$
\begin{equation}\label{4.4b}
U_2=\{w\in C_0^{2,\alpha/2}:I+f^{-1}D(f^{-1}Dw)>\frac{c_0}{2}, \|w-p_0\|_{2,\alpha/2}<\epsTwo\}.
\end{equation}
where $\epsTwo$ is chosen in Lemma \ref{l5.1}. Let $\tilde{U_2}\subset C_0^{3,\alpha}$ be defined similarly, namely
\begin{equation}\label{4.5-0-1}
\tilde{U_2}=\{w\in C_0^{3,\alpha}:I+f^{-1}D(f^{-1}Dw)>\frac{c_0}{2}, \|w-p_0\|_{3,\alpha}<\epsTwo\}.
\end{equation}
Lemma \ref{l5.1} implies that
\begin{equation}\label{4.5-1}
\begin{split}
&(I+f^{-1}\nabla p\otimes\nabla(f^{-1})+f^{-2}D^2p)\circ(id+\Mp(p,q,\delta t))+B \ge \frac {c_0} 4\\
&\quad\mbox{ for all }\;(q,p,\delta t)\in U_2\times\tilde{U_2}\times(-\epsTwo,\epsTwo),
\end{split}
\end{equation}
where $B$ is as in Lemma \ref{l5.1}. Also, by (\ref{V-W-def}),
\begin{equation}\label{4.5-Veps}
U_2\times \tilde{U_2} \subset V_{\epsTwo}.
\end{equation}

Then we can define the following map:
\begin{equation}\label{3.4}
\begin{split}
P: & U_2\times\tilde{U_2}\times(-\epsTwo,\epsTwo)\rightarrow C^{0,\alpha/2} \quad\mbox{ by}\\
& P(q,p,\delta t)=\frac{\det[I+f^{-1}\nabla q\otimes\nabla(f^{-1})+f^{-2}D^2q+A]}
{\det[(I+f^{-1}\nabla p\otimes\nabla(f^{-1})+f^{-2}D^2p)\circ(id+\Mp(p,q,\delta t))+B]}-1,
\end{split}
\end{equation}
where $A$ and $B$ are as in Lemma \ref{l5.1}.
\begin{lem}\label{l5.2}
Map (\ref{3.4}) has the following properties:
\begin{enumerate}[(i)]
\item\label{l5.2-i1}
$\displaystyle\int_{[0,1)^2}\big(P(q,p,\delta t)\big)(x)\,dx=0$ for any
$(q,p,\delta t)\in U_2\times\tilde{U_2}\times(-\epsTwo,\epsTwo)$.
Thus $P(q,p,\delta t)\in C_0^{0,\alpha/2}$, which means that
$P$ acts in the following spaces:
\begin{equation}\label{3.4-pr}
P:  U_2\times\tilde{U_2}\times(-\epsTwo,\epsTwo)\rightarrow C_0^{0,\alpha/2}
\end{equation}
\item\label{l5.2-i1.1}
$ (q,p,\delta t)\in U_2\times\tilde{U_2}\times(-\epsTwo,\epsTwo)$
satisfies equations (\ref{3.2}), (\ref{3.3}) with
$F^{-1}=Id+\Mp(\pN,\pNP, \delta t)$ if and only if
$P(q,p,\delta t)=0$ on $\bT^2$.
\item\label{l5.2-i2}
P is continuously Frechet differentiable in the spaces given in (\ref{3.4}), or
equivalently in (\ref{3.4-pr}).
\end{enumerate}
\end{lem}
\begin{proof}
First we show $P$ maps into $C_0^{0,\alpha/2}$.

Fix $ (q,h,\delta t)\in U_2\times\tilde{U_2}\times(-\epsTwo,\epsTwo)$.
From (\ref{1.31}) with $p_n=p$, $p_{n+1}=q$, $F_{n+1}^{-1}=id+\Mp(p,q,\delta t)$,
one sees that the right hand side of (\ref{3.4}) at $x\in\bR^2$ is
exactly $\det D_x(id+(\Mp(p,q,\delta t)(x)))-1$.  Denote $G:=id+\Mp(h,q,\delta t)$.
Then
$G: \bR^2\to \bR^2$ is a diffeomorphism by
Lemma \ref{mapIsDiffeo}, and
the right hand side of (\ref{3.4}) is  $\det DG (x)-1$. Then we calculate,
changing variables:
\begin{equation*}
\begin{split}
\int_{[0,1)^2}\det DG(x)\,dx=
\int_{G([0,1]^2)}dy=
\int_{[0,1]^2} dx=1,
\end{split}
\end{equation*}
where the second equality follows from $\bZ^2$-periodicity of
$G-id= \Mp(h,q,\delta t)$, and from the fact that
$G: \bR^2\to \bR^2$ is a diffeomorphism, see Lemma \ref{chgVarOnTorus}
(applied now with $h\equiv 1$).
This completes the proof of assertion (\ref{l5.2-i1}) of Lemma.

\medskip
Assertion (\ref{l5.2-i1.1}) of Lemma
follows from Lemma \ref{2.11} and (\ref{3.4}).

\medskip
Now we prove assertion (\ref{l5.2-i2}) of Lemma.

Using (\ref{4.5-1}),
Lemma \ref{l9.2} and Corollary \ref{c9.3} we see  that it is sufficient to show
that for any $(i,j)\in\{1,2\}^2$, the following
maps acting in the spaces $U_2\times\tilde{U_2}\times(-\epsTwo,\epsTwo)\rightarrow C^{0,\alpha/2}$
are continuously Frechet differentiable:
\begin{align}\label{firstMapInFrechetDiff-proof}
&(p,q,\delta t)\longmapsto \delta_{ij}+f^{-1}\partial_ip\cdot\partial_j(f^{-1})+f^{-2}\partial_{ij}q
+A_{ij},\\
\label{secndMapInFrechetDiff-proof}
&(p,q,\delta t)\longmapsto\delta_{ij}+[f^{-1}\partial_ip\cdot\partial_j(f^{-1})+f^{-2}
\partial_{ij}p](id+\Mp(p,q,\delta t))+B_{ij}.
\end{align}
Here $A_{ij}$ and $B_{ij}$ are elements of the matrices $A(x)$ and $B(x)$ which  given in
(\ref{1.32}),(\ref{1.33}) with $p=p_n$, $q=p_{n+1}$,
and $F_{n+1}^{-1}=id+\Mp(p,q,\delta t)$.
We now show differentiability of maps (\ref{firstMapInFrechetDiff-proof}), (\ref{secndMapInFrechetDiff-proof}).
From Lemma \ref{l9.4} with Lemma \ref{2.11}, the terms $f^{-1}(id+\Mp(p,q,\delta t)$
and $\nabla(f^{-1})(id+\Mp(p,q,\delta t))$ are Frechet differentiable, where we include
the terms in expressions of $A$ and $B$. Then by Lemma \ref{l9.2}, one can see that
the mapping (\ref{firstMapInFrechetDiff-proof})
is Frechet differentiable. From Lemma \ref{CompositionFrechetDif}(ii),
the terms  $D^2p(id+\Mp(p,q,\delta t))$ are also differentiable. Then we obtain differentiability of
the map (\ref{secndMapInFrechetDiff-proof}).
\end{proof}

Now we will show that the partial Frechet derivative
 $D_qP(p_0, p_0,0):C_0^{2,\alpha/2}\rightarrow C_0^{0,\alpha/2}$ is invertible.

First we can calculate from (\ref{2.2-Q1}), (\ref{5.10-eq-Q1}) amd Lemma \ref{2.11}:
\begin{equation}\label{3.6}
D_q\Mp(p_0,p_0,0)h_1=
[I+f^{-2}D^2p_0+f^{-1}\nabla p_0\otimes\nabla (f^{-1})]^{-1}(f^{-2}\nabla h_1).
\end{equation}
Then by explicit calculation, we find that $D_qP$ is
\begin{equation}\label{3.7}
\begin{split}
D_qP(p_0,p_0,0)&:C_0^{2,\alpha/2}\rightarrow C_0^{0,\alpha/2}
\\&h\longmapsto L(h),
\end{split}
\end{equation}
where
\begin{equation}\label{3.8}
\begin{split}
L(h)&= \frac{\sum_{i,j=1}^2
M_{ij}[f^{-2}\partial_{ij}h-\partial_j(f^{-1}\nabla(f^{-1}\partial_i p_0))
\cdot(D_q\Mp(p_0,p_0,0)h)]}
{\det(I+f^{-1}D(f^{-1}Dp_0))}.
\end{split}
\end{equation}

Here $M=M_{ij}$ is the cofactor matrix of $I+f^{-1}D(f^{-1}D\pN_0)$, which is strictly positive definite due to (\ref{4.5-0-1}).
Notice we already computed $D_q\Mp(p_0,p_0,0)$ in (\ref{3.6}).
\begin{rem}
Note that  the operator (\ref{3.8}) acts in spaces given in (\ref{3.7}),
i.e. that $\displaystyle\int_{[0,1)^2}\big(L(h)\big)(x)\,dx=0$ for any
$h\in C_0^{2,\alpha/2}$. This follows from \eqref{3.4-pr}, since $L=D_q P(p_0,p_0,0)$.
\end{rem}

Next we argue the linear operator $L$ defined above is invertible and the inverse is a bounded linear operator. First we
observe that $L$ can be put in the form
$$L(h)=a_{ij}\partial_{ij}h+b_i\partial_ih.
$$
with coefficients $a_{ij},b_i\in C^{\alpha/2}(\bT^2)$, with the norms depending only on $\|\pN_0\|_{C^{3,\alpha}}$
and $\frac 1{c_0}$.
Also it follows from (\ref{3.8}), $L$ is uniformly elliptic,
precisely $\displaystyle a_{ij}=\frac{M_{ij}}{f^2\det(I+f^{-1}D(f^{-1}Dp_0))}$ and thus,
ellipticity follows from (\ref{4.5-0-1}) and regularity of $p_0, f, f^{-1}$.

Now invertibility of $D_qP(p_0, p_0,0):C_0^{2,\alpha/2}\rightarrow C_0^{0,\alpha/2}$
follows from the following lemma.

\begin{lem}\label{3.9}
Let $L(h)=a_{ij}\partial_{ij}h+b_i\partial_ih$
be a uniformly elliptic operator
on $\bT^2$, with coefficients $a_{ij},b_i\in C^{\alpha}(\bT^2)$. Suppose that
coefficients $a_{ij},b_i$ satisfy the following additional property:
$L(h)\in C_0^{\alpha}(\bT^2)$ for any $h\in C_0^{2,\alpha}(\bT^2)$.
Then $L:C_0^{2,\alpha/2}\rightarrow C_0^{0,\alpha/2}$
is isomorphism.
\end{lem}
\begin{proof}
The injectivity follows from strong maximum principle.
Indeed, if $L(h)=0$ for some $h\in C^{2,\alpha}$, then by strong maximum principle,
 $h$ must be a constant. Since $h\in C_0^{2,\alpha}$, i.e.
 $\displaystyle\int_{[0,1)^2} h \,dx=0$,
 this constant must be zero.

To show surjectivity, we use the method of continuity. We consider the following family of operators:
\begin{equation}\label{3.10}
\begin{split}
&L_t:C_0^{2,\alpha}\rightarrow C_0^{0,\alpha}\textrm{ with $t\in[0,1]$}
\\&L_t(h)=(1-t)\Delta h+tL(h).
\end{split}
\end{equation}

When $t=0$, $L_0=\Delta$. Equation
$\Delta h=k$ has a unique solution in $C_0^{2,\alpha}$ with any $k\in C_0^{0,\alpha}$.
Uniqueness is again the result of strong maximum principle.
Existence can be obtained by minimizing the functional
$I[v]=\int_{\mathbb{T}^2}\frac{1}{2}|\nabla v|^2+kv$ over the space
$H_0^1(\mathbb{T}^2):=\{v\in H_{loc}^1(\mathbb{R}^2):\textrm{ $v$
is  $\mathbb{Z}^2$-periodic, and }\int_{\mathbb{T}^2}v=0\}$.

By Theorem 5.20 in \cite{GT}, to see that $L_1$ is surjective, we just have to show
the estimate
\begin{equation}\label{3.11}
\|h\|_{2,\alpha}\leq C\|L_th\|_{0,\alpha}\;\;\textrm { for all $ t\in[0,1]$
and $h\in C_0^{2,\alpha}$}.
\end{equation}
By Schauder`s estimate, we have
\begin{equation}\label{3.12}
\|h\|_{2,\alpha}\leq C(\|h\|_0+\|L_th\|_{0,\alpha}).
\end{equation}
Here $C$ depends on the $C^{\alpha}$ norm of the coefficients and the ellipticity
constant of operator $L$. Both are independent of $t$. So we just need to show
\begin{equation}\label{3.13}
\|h\|_0\leq C\|L_th\|_{0,\alpha}\textrm{  $h\in C_0^{2,\alpha}$ and
$\forall t\in[0,1]$}.
\end{equation}

We use compactness and argue by contradiction.

If (\ref{3.13}) were false, then for any $n\geq1$, there exists $t_n\in[0,1]$,
$h_n\in C_0^{2,\alpha}$, such that $\|h_n\|_0\geq n\|L_{t_n}h_n\|_{0,\alpha}$.
After normalization, we can assume
$\|h_n\|_0\equiv1$ and $\|L_{t_n}h_n\|_{0,\alpha}\rightarrow0$.
By Schauder`s estimate, $h_n$ is bounded in $C^{2,\alpha}$. So up to a subsequence,
we can assume $t_n\rightarrow t_*\in[0,1]$, $h_n\rightarrow h_*$ in $C^2$,
and $h_*\in C_0^{2,\alpha}$. Then we will have $L_{t_*}h_*=0$. By strong maximum
principle, we have $h_*\equiv0$. On the other hand $h_n\rightarrow h_*$ uniformly,
we have $\|h_*\|_0=1$. This is a contradiction.

\end{proof}

 Hence we can conclude the following:
\begin{prop}\label{5.4spm}
There exist $\epsThree, \epsFour\in (0, \epsTwo]$ with
$\epsFour\le \epsThree$, such that for any $p\in C^{3,\alpha}_0(\bT^2)$ with
$\|p-p_0\|_{3,\alpha}<\epsFour$ and $\delta t\in (-\epsFour, \epsFour)$, there exists a unique
$q\in C^{2,\alpha/2}(\bT^2)$ which solves (\ref{3.3})
with $F^{-1}=Id+\Mp(\pN,\pNP, \delta t)$
 and satisfies
$\|q-p_0\|_{2,\alpha/2}<\epsThree$.

Thus, denoting $q:=\mathcal{H}(p,\delta t)$ and
$U_3=\left\{p\in C^{3,\alpha}_0(\bT^2)\;:\; \|p-p_0\|_{3,\alpha}<\epsFour\right\}$,
we obtain a map
$$\mathcal{H}:U_3\times(-\epsFour,\epsFour)\rightarrow U_2,
$$
such that for each for any $(\pN, \delta t)\in U_3\times(-\epsFour,\epsFour)$,
defining $q=\mathcal{H}(\pN,\delta t)$ and
$F^{-1}=Id+\Mp(\pN,\mathcal{H}(\pN,\delta t), \delta t)$,
we get solution $(p, q, F^{-1}, \delta t)$ of (\ref{3.3}).
\end{prop}
\begin{proof}
  This follows from
Lemma \ref{l5.2}(\ref{l5.2-i1.1}), and (\ref{3.7}), (\ref{3.8})  with Lemma \ref{3.9} by
 implicit function theorem.

\end{proof}

\begin{rem}\label{3.14}
If $p_0$ is chosen in a compact subset of $C^{3,\alpha}(\bT^2)$, one can see such choice of $\epsFour$ is actually uniform. In particular, this choice of $\epsFour$ is uniform on any bounded subsets in $C^{4,\alpha}$.
\end{rem}

We also prove the following lemma which will be used below:
\begin{lem}\label{5.4spm-lem}
If $\epsThree$ hence $\epsFour$ are sufficiently small,
 then $C(x)\geq\frac{c_0}{4}$
for any $p$ and $\delta t$ as in Proposition \ref{5.4spm} and
$q=\mathcal{H}(p,\delta t)$.
Here $C(x)$ is:
\begin{equation}\label{FirstC}
\begin{split}
C(x)=&I+\int_0^1D[f^{-2}Dp]((1-\theta)F^{-1}(x)+\theta x)d\theta\\
&-f^{-1}(x)\int_0^1\nabla(f^{-1})((1-\theta)F^{-1}(x)+\theta x)d\theta\otimes\nabla p,\\
&\text{with }\;F^{-1}=Id+\Mp(\pN,\pNP, \delta t).
\end{split}
\end{equation}
\end{lem}
\begin{proof}
\medskip
We start by observing that
\begin{equation*}
\begin{split}
C(x)-(I+f^{-1}D(f^{-1}Dp))(x)=&\int_0^1\big[D(f^{-2}Dp)((1-\theta)F^{-1}(x)+\theta x)-
D(f^{-2}Dp)(x)\big]d\theta\\
&-f^{-1}(x)\int_0^1\big[\nabla(f^{-1})((1-\theta)F^{-1}(x)+\theta x)-
\nabla(f^{-1})(x)\big]d\theta\otimes\nabla p.
\end{split}
\end{equation*}
Hence, recalling $F^{-1}=Id+\Mp(\pN,\pNP, \delta t)$, we get:
\begin{equation*}
\begin{split}
|C(x)-(I+f^{-1}D(f^{-1}Dp)(x)|&\leq\|D^2(f^{-2}Dp)\|_0
\|\Mp(p,q,\delta t)\|_0+\|f^{-1}\|_0\|D^2(f^{-1})\|_0\|\Mp(p,q,\delta t)\|_0\\
&\leq \big(10\|f^{-2}\|_{C^2}(\|p_0\|_{3,\alpha}+\epsThree)+\|f^{-1}\|_0\|\nabla(f^{-1})\|_0\big)\|\Mp(p,q,\delta t)\|_0.
\end{split}
\end{equation*}
By Remark \ref{identitySolnRemark}, we can make $\|\Mp(p,q,\delta t)\|_0$ as small
as we wish as long as we choose $\epsThree$ small. In particular, we can make above
line $\leq\frac{c_0}{4}$. Now since $p\in\tilde{U_2}$,
we know $I+f^{-1}D(f^{-1}Dp)>\frac{c_0}{2}$, it follows that $C(x)\geq\frac{c_0}{4}$.
\end{proof}

From now on, we fix $\epsThree$ and $\epsFour$ such that Lemma \ref{5.4spm-lem} holds.

\section{Estimates of solutions on time steps}
\label{constApprSol-1-Sect}

 Suppose the initial data satisfies
 $p_0\in C^{k+2,\alpha}_0$ and $I+f^{-1}D(f^{-1}Dp)>c_0$.


Fix $\delta t\in (0, \epsFour)$, and define $p_1, p_2, \dots$  as following.
Assume that for $n=0, 1, \dots$,
we have defined
 $p_n\in U_3$, where $U_3$ is from Proposition \ref{5.4spm}.
Then we can define $p_{n+1}:=\mathcal{H}(p_n,\delta t)$.
 Thus
we have $p_{n+1}\in U_2$, and by Lemma \ref{mapIsDiffeo}
and (\ref{4.5-Veps}) with $\epsTwo$ determined by Lemma \ref{l5.1}, we  can define the flow map $F_{n+1}$
  which is a diffeomorphism and solves
\begin{align}\label{3.1}
&x+f^{-1}(x)f^{-1}(F_{n+1}^{-1}(x))\nabla p_{n+1}(x)=
F_{n+1}^{-1}(x)+(f^{-2}R_{f\delta t}\nabla p_n)(F_{n+1}^{-1}(x));\\
\label{3.1new}
&\det DF_{n+1}=1,
\end{align}
where (\ref{3.1new}) follows from (\ref{3.3}) written in the form
(\ref{1.34}). By the definition of $U_2$, we have $|A_{n+1}(x)|\leq\frac{c_0}{4}$.
   Hence $I+f^{-1}D(f^{-1}Dp_{n+1})+A_{n+1}>\frac{c_0}{4}$.

In order to continue the process, we need to show that $p_{n+1}\in U_3$.
We will show that this is true if $n\delta t$ is sufficiently small,
i.e. if $n\delta t \le T$, where $T>0$ does not depend on $\delta t$.
In order to show this, we establish some estimates for the approximate
solutions $p_n$.

\begin{lem}\label{3.16}
Let $p_{n+1}=\mathcal{H}(p_n,\delta t)$ with $\|p_n-p_0\|_{3,\alpha}<\epsFour$,
$|\delta t|<\epsFour$. Then
\begin{equation}\label{3.17}
\|p_{n+1}\|_{k+2,\alpha}\leq C_0(\|p_n\|_{k+2,\alpha}).
\end{equation}
\end{lem}
\begin{proof}
First we show
\begin{equation}\label{4.20p}
\|p_{n+1}\|_{k+2,\alpha/2}\leq \CTwelve(\|p_n\|_{k+2,\alpha/2}).
\end{equation}
This follows from differentiating (\ref{3.3}). Indeed, we have from our assumption and the definition of the map $\mathcal{H}$ that $\|p_{n+1}\|_{2,\alpha/2}\leq\|p_0\|_{2,\alpha/2}+1$. Also it follows from Lemma \ref{2.11} that  $\|\Mp(p_n,p_{n+1},\delta t)\|_{1,\alpha/2}\leq 1$. Now one differentiate (\ref{3.3}) to see $Dp_{n+1}$ solves an elliptic equation with main coefficients given by $M_{ij}$, the $(i,j)$ entry of the cofactor matrix of $I+f^{-1}\nabla p_{n+1}\otimes\nabla(f^{-1})+f^{-2}D^2p_{n+1}(x)+A_{n+1}(x)$. The main coefficients are uniformly elliptic, with ellipticity constant depending on $c_0$ and $\|p_0\|_{2,\alpha/2}$, because by Lemma \ref{l5.1},  $I+f^{-1}D(f^{-1}Dp_{n+1})+A_{n+1}>\frac{c_0}{4}$. One also sees all the coefficients of this equation are in $C^{0,\alpha/2}$, with norm bounded by $\|p_n\|_{3,\alpha/2}$ and $\|p_{n+1}\|_{2,\alpha/2}$. This follows from calculation based on (\ref{1.32}),(\ref{1.33}). So one can apply Schauder estimates to conclude $Dp_{n+1}$ is bounded in $C^{2,\alpha/2}$, or $p_{n+1}$ bounded in $C^{3,\alpha/2}$. Now one looks at (\ref{1.31}) to conclude $\|\Mp(p,q,\delta t)\|_{2,\alpha/2}$ can be bounded by $\|p_{n+1}\|_{3,\alpha/2}$, $\|p_n\|_{3,\alpha/2}$. Then differentiate (\ref{3.3}) twice to see $D^2p_{n+1}$ solves a uniformly elliptic equation with coefficients bounded in $C^{0,\alpha/2}$ by $\|p_{n+1}\|_{3,\alpha/2}$ and $\|p_n\|_{4,\alpha/2}$. One can further differentiate (\ref{3.3}) and use Schauder`s estimates again and again to get (\ref{4.20p}).

Then (\ref{4.20p}) gives a bound for $\|p_{n+1}\|_{2,\alpha}$, since $k\geq1$. Therefore by looking at (\ref{1.31}), one sees that $\|\Mp(p_{n},p_{n+1},\delta t)\|_{1,\alpha}$ can be bounded by $\|p_{n+1}\|_{2,\alpha}$ and $\|p_n\|_{2,\alpha}$. So the same argument as in previous paragraph gives the desired conclusion.

\end{proof}
\begin{lem}\label{3.18}
Under conditions of Lemma \ref{3.16},
\begin{equation}\label{3.19}
\|p_{n+1}-p_n\|_{k+1,\alpha}\leq C_1\delta t,
\end{equation}
where the constant $C_1=C_1(\|p_{n+1}\|_{k+2,\alpha},\|p_n\|_{k+2,\alpha})$.
\end{lem}
\begin{proof}
 In this argument, all the constants $C$ depends only on $\|p_{n+1}\|_{k+2,\alpha}$,
 $\|p_n\|_{k+2,\alpha}$ and may change line from line.
 Write $q(x)=p_{n+1}(x)$, $p(x)=p_n(x)$. First we observe that
 $\|\Mp(p,q,\delta t)\|_{k+1,\alpha}=\|F^{-1}-id\|_{k+1,\alpha}\leq C$. Here $C$ has
 the dependence as stated in the lemma. To see this, we need to recall (\ref{1.31}),
 and this estimate follows from differentiating (\ref{1.31}) and a bootstrap argument.
 Indeed, first from
Lemma \ref{2.11}, we know that $\|\Mp(p,q,\delta t)\|_{1,\alpha/2}\leq\eps_2\leq 1$.
Also it follows from  (\ref{4.5-1}) that
$(I+f^{-1}\nabla p\otimes\nabla(f^{-1})+f^{-2}D^2p)(F^{-1})+B\geq\frac{c_0}{4}$,
therefore we can invert and obtain
\begin{equation}\label{new}
DF^{-1}=\big[(I+f^{-1}\nabla p\otimes\nabla(f^{-1})+f^{-2}D^2p)(F^{-1})+
B\big]^{-1}\cdot\big[I+f^{-1}\nabla q\otimes\nabla(f^{-1})+f^{-2}D^2q+A\big].
\end{equation}

Since we already have $p,q\in C^{2,\alpha}$, the formula in (\ref{1.33}) (\ref{1.34})
for $A$ and $B$ with (\ref{new}) gives $D_x\Mp(p,q,\delta t)(x)\in C^{0,\alpha}$,
with a $C^{0,\alpha}$ bound having the stated dependence. This shows
$\Mp(p,q,\delta t)\in C^{1,\alpha}$. Now since $k\geq 2$, we know actually
$p,q\in C^{3,\alpha}$. This implies the right hand side of (\ref{new}) is in
$C^{1,\alpha}$. Therefore we obtain from (\ref{new}) that $\Mp(p,q,\delta t)$ is in
$C^{2,\alpha}$. If it happens that $p,q\in C^{4,\alpha}$, then we know the right
hand side of (\ref{new}) is in $C^{2,\alpha}$, and hence it gives
$\Mp(p,q,\delta t)\in C^{3,\alpha}$. One can repeat this argument and it gives
in general that if $p,q\in C^{k+2,\alpha}$, then $\Mp(p,q,\delta t)\in C^{k+1,\alpha}$,
with an estimate on the $C^{k+1,\alpha}$ norm which has the dependence stated in the
lemma.

We  subtract from both sides of (\ref{3.3}) the quantity $\det(I+f^{-1}D(f^{-1}Dp))$,
and write the resulting equation as a linear equation for $q-p$.
Then the the left hand side of  the resulting equation can be written as
\begin{equation}\label{3.20}
a_{ij}[f^{-2}(x)\partial_{ij}(q-p)+f^{-1}(x)\partial_i(q-p)\partial_j(f^{-1})+A_{ij}],
\end{equation}
where
$$a_{ij}=\int_0^1M_{ij}[(1-\theta)(M_1)+\theta(M_2)]d\theta.
$$
Here $M_{ij}$ denotes the $(i,j)$ entry of the cofactor matrix, and $M_1=I+f^{-1}D(f^{-1}Dq)+A$, $M_2=I+f^{-1}D(f^{-1}Dp)$, $A_{ij}$ is the element of the matrix $A(x)$ from (\ref{1.32}).
The right hand side becomes
\begin{equation}\label{3.21}
\begin{split}
b_{ij}[(f^{-1}\partial_i p& \partial_j(f^{-1})(F^{-1}(x))-f^{-1}\partial_ip\partial_j(f^{-1})(x))+(f^{-2}\partial_{ij}p(F^{-1}(x))-f^{-2}\partial_{ij}p(x))+B_{ij}]
\\&=b_{ij}[\int_0^1\partial_l(f^{-1}\partial_ip\partial_j(f^{-1})+f^{-2}\partial_{ij}p)((1-\theta)x+\theta F^{-1}(x))d\theta(F^{-1}_l-x_l)+B_{ij}],
\end{split}
\end{equation}
where
$$
b_{ij}=\int_0^1M_{ij}[(1-\theta)(M_1^{\prime})+\theta(M_2)]d\theta.
$$
Here $M_1^{\prime}=I+f^{-1}D(f^{-1}Dp)(F^{-1}(x))+B$.
Now observe we can write
\begin{equation}\label{3.22}
A(x)=\int_0^1\partial_k(f^{-1})((1-\theta)F^{-1}(x)+\theta x)d\theta(F_l^{-1}-x_l)\cdot(\nabla q\otimes\nabla (f^{-1})+f^{-1}D^2q).
\end{equation}
As for $B(x)$, the first term can be rewritten as
\begin{equation}\label{3.23}
\begin{split}
&\,\,\,\,\,\,[(f^{-1}\nabla p)(F^{-1}(x))-f^{-1}\nabla q(x)]\otimes\nabla(f^{-1})(F^{-1}(x))\\
&=[(f^{-1}\nabla p)(F^{-1}(x))-f^{-1}\nabla p(x)]\otimes\nabla(f^{-1})(F^{-1}(x))+f^{-1}\nabla(p-q)\otimes\nabla(f^{-1})(F^{-1}(x))\\
&=\int_0^1\partial_l(f^{-1}\nabla p)(\theta F^{-1}(x)+(1-\theta)x)d\theta\otimes\nabla(f^{-1})(F^{-1}(x))(F^{-1}_l-x_l)\\
&+f^{-1}\nabla(p-q)\otimes\nabla(f^{-1})(F^{-1}(x)).
\end{split}
\end{equation}
The second term of $B(x)$ can be rewritten as
\begin{equation}\label{3.24}
D[f^{-2}(R_{f\delta t}-I)\nabla p](F^{-1}(x))=g \delta t,
\end{equation}
with $\|g\|_{k,\alpha}\leq C$.

To summarize, the difference $q-p$ satisfies an equaiton of the following form
\begin{equation}\label{3.25}
a_{ij}\partial_{ij}(q-p)+b_i\partial_i(q-p)+c_l(F_l^{-1}-x_l)=g\delta t,
\end{equation}
with $\|a_{ij}\|_{k,\alpha}, \|b_i\|_{k,\alpha}, \|c_l\|_{k-1,\alpha}\leq C$,

Next we represent $F^{-1}(x)-x$ in terms of $\nabla(q-p)$, with an error term controlled by $\delta t$. For this we need to go back to (\ref{3.2}). Subtract from both sides of (\ref{3.2}) $x+f^{-2}(x)\nabla p(x)$, we obtain
\begin{equation}\label{3.26}
\begin{split}
 f^{-1}(x)(f^{-1}(F^{-1}&(x))\nabla q(x)-f^{-1}(x)\nabla p(x))=F^{-1}(x)-x\\&+(f^{-2}\nabla p)(F^{-1}(x))-f^{-2}\nabla p(x)
 +[f^{-2}(R_{f\delta t}-I)]\nabla p(F^{-1}(x)).
\end{split}
\end{equation}
After rearranging terms, we get
\begin{equation}\label{3.27}
f^{-1}(x)f^{-1}(F^{-1}(x))\nabla(q-p)-[f^{-2}(R_{f\delta t}-I)\nabla p](F^{-1}(x))=C(x)(F^{-1}(x)-x).
\end{equation}
where
\begin{equation}\label{3.28}
\begin{split}
C(x)=&I+\int_0^1D[f^{-2}Dp]((1-\theta)F^{-1}(x)+\theta x)d\theta\\
&-f^{-1}(x)\int_0^1\nabla(f^{-1})((1-\theta)F^{-1}(x)+\theta x)d\theta\otimes\nabla p.
\end{split}
\end{equation}
From Proposition \ref{5.4spm}, we know $C(x)\geq\frac{c_0}{4}$. Now
in (\ref{3.27}) we have $f^{-1}(F^{-1}(x))\in C^{k+1,\alpha}$, with norms controlled
by $\|p\|_{k+2,\alpha}$, $\|q\|_{k+2,\alpha}$. Also the
term $f^{-2}(R_{f\delta t}-I)\nabla p(F^{-1}(x))=m\delta t$ with $\|m\|_{k,\alpha}$
bounded by $C(\|p\|_{k+2,\alpha},\;\|q\|_{k+2,\alpha})$.
Therefore in the equation (\ref{3.25}), $c_l(F_l^{-1}-x_l)$ can be dispensed of
and the result follows from Schauder`s estimate.
\end{proof}

\begin{lem}\label{3.29}
Let $(p_n, p_{n+1}, \delta t)$ be as in Lemma \ref{3.16},
and
\begin{equation}\label{3.30}
\|F^{-1}-id\|_{k,\alpha}\leq C_2\delta t.
\end{equation}
Here $C_2=C_2(\|p_{n+1}\|_{k+2,\alpha},\|p_n\|_{k+2,\alpha})$.
\end{lem}
\begin{proof}
We use (\ref{3.27}),(\ref{3.28}) to get
\begin{equation}\label{3.31}
F^{-1}(x)-x=C(x)^{-1}[f^{-1}(x)f^{-1}(F^{-1}(x))\nabla(q-p)-f^{-2}(R_{f\delta t}-I)\nabla p(F^{-1}(x))].
\end{equation}
We use the inequality $\|fg\|_{k,\alpha}\leq C_k\|f\|_{k,\alpha}\|g\|_{k,\alpha}$.
Notice that $C(x)\in C^{k,\alpha}$, hence $C^{-1}\in C^{k,\alpha}$, with
$C^{k,\alpha}$ norm controlled by $\|p_{i}\|_{k+2,\alpha}$, $i=n,n+1$,
and $\frac{1}{c_0}$, by Proposition \ref{5.4spm}. The $C^{k,\alpha}$ norm of the
square bracket is controlled by $C\delta t$ by Lemma \ref{3.18}.
\end{proof}
For immediate use, we prove the following lemma,
\begin{lem}\label{3.41}
Let $\alpha\in[0, 1)$.
Let $k\geq1$, $G\in C^{k,\alpha}(\bT^2)$, and let $F^{-1}:\bR^2\rightarrow \bR^2$
be a map which satisfies $F^{-1}(x+h)=F^{-1}(x)+h$ for any $h\in \bZ^2$
and $\|F^{-1}-id\|_{k,\alpha}\leq C^*\delta t$, then
\begin{equation}\label{3.42}
\|G\circ F^{-1}\|_{k,\alpha}\leq \|G\|_{k,\alpha}+C\delta t,
\end{equation}
where $C$ depends on $\|G\|_{k,\alpha}$, $C^*$, $\alpha$ and $k$.
\end{lem}
\begin{proof}
In the following argument, constant $C$ has dependence as in the lemma, and may
change from expression
to expression.

We assume $\alpha\in(0,1)$, since the case $\alpha=0$ is simpler, and follows from
the argument below.

It is obvious that
\begin{equation}\label{NormZero}
\|G\circ F^{-1}\|_0\leq \|G\|_0.
\end{equation}
 Thus we need to estimate terms
$D^\beta(G\circ F^{-1})$ for multi-indices $\beta$ satisfying
$1\le|\beta|\le k$.

Using that $D_i(G\circ F^{-1})=\sum_{j=1}^2((D_jG)\circ F^{-1})D_iF_j^{-1}$, we obtain
\begin{equation}\label{ChainRuleId}
D_i(G\circ F^{-1})
=(D_iG)\circ F^{-1}+\sum_{j=1}^2((D_jG)\circ F^{-1})D_i(F^{-1}-id)_j.
\end{equation}
Next we show that for any multi-index $\beta$
with $|\beta|\ge 1$,
\begin{equation}\label{ChainRuleIdHigher}
D^\beta(G\circ F^{-1})
=(D^\beta G)\circ F^{-1}+\sum_{1\le|\gamma|\le|\beta|}\sum_{j=1}^2
a_{\beta, \gamma, j} D^\gamma(F^{-1}-id)_j,
\end{equation}
where $a_{\beta, \gamma, j}$ is a polynomial expression of
$\{(D^{\gamma'} G)\circ F^{-1}\}_{1\le|\gamma'|\le|\beta|}$,
$\{D^{\gamma'} F^{-1}\}_{1\le|\gamma'|\le|\beta|-1}$.

Indeed, we prove (\ref{ChainRuleIdHigher}) by induction over $m=|\beta|$.
Case $m=1$ follows from (\ref{ChainRuleId}). Next, assume that $m\ge 1$ and
that (\ref{ChainRuleIdHigher}) is proved for all $|\beta|\le m$.
Let $|\beta'|=m+1$, then $D^{\beta'}=D_iD^\beta$ for some
$i\in\{1,2\}$ and $|\beta|=m$.
Now, taking $D_i$-derivative of (\ref{ChainRuleIdHigher}) for $\beta$, and applying
(\ref{ChainRuleId}) with $D^\beta G$ instead of $G$ to handle the derivative of
the first
term in the right-hand side of (\ref{ChainRuleIdHigher}),
we obtain (\ref{ChainRuleIdHigher}) for $\beta'$.

Now from (\ref{ChainRuleIdHigher}), recalling the structure
of $a_{\beta, \gamma, j}$, we see that if $1\le|\beta|\le k$, then
\begin{equation}\label{ChainRuleIdHigher-NormZero}
\|D^\beta(G\circ F^{-1})\|_0\le
\|D^\beta G\|_0+C\delta t,
\end{equation}
where dependence of $C$ is as in the lemma.

In order to estimate $[D^\beta(G\circ F^{-1})]_\alpha$ for
$|\beta|=k$, we first estimate
this seminorm for the first term in the right-hand side of
(\ref{ChainRuleIdHigher}):
\begin{align*}
[(D^\beta G)\circ F^{-1}]_\alpha & =
\sup_{x,y}\frac{|(D^\beta G)\circ F^{-1}(x)-(D^\beta G)\circ F^{-1}(y)|}{|x-y|^\alpha} \\
&\leq  [D^\beta G]_\alpha
\sup_{x,y}\frac{\big(|(F^{-1}(x)-x)-(F^{-1}(y)-y)|+
|x-y|\big)^\alpha}{|x-y|^\alpha}\\
&=  [D^\beta G]_\alpha
\left(1+\sup_{x,y}\frac{|(F^{-1}(x)-x)-(F^{-1}(y)-y)|}{|x-y|}
\right)^\alpha\\
&\leq  [D^\beta G]_\alpha (1+\|D(F-Id)\|_{0})^\alpha \\
&\le [D^\beta G]_\alpha (1+C\|F-Id\|_{1})\le [D^\beta G]_\alpha (1+C\delta t).
\end{align*}
Then by (\ref{ChainRuleIdHigher}),
$$
[D^\beta (G\circ F^{-1})]_\alpha\le [D^\beta G]_\alpha (1+C\delta t) +C
\delta t\le [D^\beta G]_\alpha +C\delta t.
$$
Combining this with (\ref{NormZero}) and (\ref{ChainRuleIdHigher-NormZero}),
we conclude the proof of (\ref{3.42}).
\end{proof}

Define
\begin{equation}\label{3.15}
\nu_m:
=\det(I+f^{-1}D(f^{-1}Dp_m)).
\end{equation}
\begin{lem}\label{3.32}
\begin{equation}\label{3.33}
\|\nu_{n+1}\|_{k,\alpha}\leq\|\nu_n\|_{k,\alpha}+C_3\delta t.
\end{equation}
Here $C_3=C_3(\|p_{n+1}\|_{k+2,\alpha},\|p_n\|_{k+2,\alpha})$. Recall $\nu_n=\det(I+f^{-1}D(f^{-1}Dp_n))$.
\end{lem}
\begin{proof}
Write $q(x)=p_{n+1}(x)$, $p(x)=p_n(x)$. Define
\begin{equation*}
\begin{split}
&\nu_n^{\prime}=\det(I+f^{-1}D(f^{-1}Dp)(F^{-1}(x)),\\
&\nu_n^{\prime\prime}=\det(I+f^{-1}D(f^{-1}Dp)(F^{-1}(x))+B(x)),\\
&\nu^{\prime}_{n+1}=\det(I+f^{-1}D(f^{-1}Dq)+A(x)).
\end{split}
\end{equation*}
By the equation, we have $\nu_{n+1}^{\prime}=\nu_n^{\prime\prime}$. \\
We will prove the lemma by showing
\begin{align}\label{3.34}
&\|\nu_n^{\prime}\|_{k,\alpha}\leq\|\nu_n\|_{k,\alpha}+C\delta t,\\
\label{3.35}
&\|\nu_{n+1}^{\prime}-\nu_{n+1}\|_{k,\alpha}\leq C\delta t,\\
\label{3.36}
&\|\nu_n^{\prime}-\nu_n^{\prime\prime}\|_{k,\alpha}\leq C\delta t.
\end{align}
First we observe that (\ref{3.34}) is a direct consequence of Lemma \ref{3.41} applied to $G=\det(I+f^{-1}D(f^{-1}Dp))$ and Lemma \ref{3.29}.

To see (\ref{3.35}), we can write, upon noticing $p_{n+1}=q$,
\begin{equation}\label{3.37}
\nu_{n+1}^{\prime}-\nu_{n+1}=\int_0^1M_{ij}(I+f^{-1}D(f^{-1}Dq)+\theta A(x))d\theta A_{ij}(x).
\end{equation}
Hence
\begin{equation}\label{3.38}
\|\nu_{n+1}^{\prime}-\nu_{n+1}\|_{k,\alpha}\leq C_k\|\int_0^1I+f^{-1}D(f^{-1}Dq)+\theta A(x))d\theta\|_{k,\alpha}\|A_{ij}\|_{k,\alpha}.
\end{equation}
But $\|A_{ij}\|_{k,\alpha}\leq C\delta t$ by (\ref{3.22}) and Lemma \ref{3.29}. So (\ref{3.35}) is proved.
(\ref{3.36}) is proved in a similar way as (\ref{3.35}), by noticing that
$\|B_{ij}\|_{k,\alpha}\leq C\delta t$. This follows from (\ref{3.23}), (\ref{3.24}),
Lemma \ref{3.29}.
\end{proof}
\begin{lem}\label{3.39}
 For $n=0, 1, \dots$
\begin{equation}\label{3.40}
\|p_{n}\|_{k+2,\alpha}\leq C_4(\|\nu_{n}\|_{k,\alpha}+1).
\end{equation}
Here $C_4=C_4(\|p_{n}\|_{2,\alpha})$.

\end{lem}
\begin{proof}
We differentiate both sides of (\ref{3.15}) with $m$ replaced by $n$ and obtain linear
elliptic equations for derivatives of $p_{n}$. We can then use Schauder estimate to
conclude.
\end{proof}

Now we apply these estimates to the approximate solutions.

\begin{prop}\label{approxSolut-est-Prop}
For any $p_0\in C^{k+2,\alpha}$, $k\geq2$, satisfying $I+f^{-1}D(f^{-1}Dp_0)>c_0$,
there exists $T_1>0$, $\eps>0$, depending only on $\|p_0\|_{k+2,\alpha}$, $f$, $c_0$ and $k$,
such that for any $|\delta t|<\eps$, one can construct a sequence
$(p_n, F_n)$ of solutions to (\ref{3.1}), (\ref{3.1new}) for $n\leq \frac{T_1}{\delta t}-1$,
such that $p_n\in C_0^{k+2,\alpha}$ and $F_n:\bT^2\to \bT^2$ is a measure preserving
diffeomorphism.
Moreover, $(p_n, F_n)$ satisfy the following estimates:
\begin{align}\label{3.55}
&\|p_n\|_{k+2,\alpha}\leq C,\\
\label{3.56old}
&\|p_{n+1}-p_n\|_{k+1,\alpha}\leq C\delta t,\\
\label{3.57}
&\|F_{n+1}^{-1}-id\|_{k,\alpha}\leq C\delta t,\\
\label{3.58}
&\|F_{n+1}-id\|_{k,\alpha}\leq C\delta t.
\end{align}
 for $0\leq n\leq \frac{T_1}{\delta t}-1$.
Here $C$ depends only on $\|p_0\|_{k+2,\alpha}$, $f$, $c_0$ and $k$.
\end{prop}

\begin{proof}

 Let $M=4\|\nu_0\|_{k,\alpha}$.

Let $C_5=C_4(\|p_0\|_{2,\alpha}+1)(M+1)$. This is the bound for $\|p_n\|_{k+2,\alpha}$ given by Lemma \ref{3.39} if we have the bound $\|\nu_n\|_{k,\alpha}\leq M$ and the bound $\|p_n\|_{2,\alpha}\leq\|p_0\|_{2,\alpha}+1$.

Define $\CThirt$ to be the bound for $\|p_{n+1}\|_{k+2,\alpha}$ given by Lemma \ref{3.16} if we have $\|p_n-p_0\|_{3,\alpha}<\epsFour$, $\|p_n\|_{k+2,\alpha}\leq C_5$, and $p_{n+1}=\mathcal{H}(p_n,\delta t)$.

 Now let $\CSix$ be the constant given by Lemma \ref{3.32} such that $\|\nu_{n+1}\|_{k,\alpha}\leq \|\nu_n\|_{k,\alpha}+\CSix\delta t$ if we have the bound $\|p_n\|_{k+2,\alpha},\|p_{n+1}\|_{k+2,\alpha}\leq \CThirt$.

 Similarly let $\CSeven$ be the constant given by Lemma \ref{3.18} such that $\|p_{n+1}-p_n\|_{k+1,\alpha}\leq \CSeven\delta t$ if we have the bound $\|p_n\|_{k+2,\alpha},\|p_{n+1}\|_{k+2,\alpha}\leq \CThirt$

Finally let $\CEight=\min(C_5,\CThirt)$.

We will show by induction that
\begin{equation}\label{3.48-ind}
\|\nu_n\|_{k,\alpha}\leq\|\nu_0\|_{k,\alpha}+n\CSix\delta t,
\quad\|p_n\|_{k+2,\alpha}\leq \CEight,
\quad\|p_n-p_0\|_{k+1,\alpha}\leq nC_7\delta t
\end{equation}
 as long as $\delta t\in(0,\epsFour)$
and
\begin{equation}\label{3.48}
n\delta t\leq\min(\frac{\epsFour}{2\CSeven},\frac{M}{2\CSix}):=T_1,
\end{equation}
where $\epsFour$ is chosen according to Proposition \ref{5.4spm}.

We note that (\ref{3.48-ind}) obviously holds for $n=0$. Suppose
(\ref{3.48-ind}) is true for any
$j\leq n$, suppose also (\ref{3.48}) holds for $n+1$. By the
induction hypothesis, we have $\|p_n\|_{k+2,\alpha}\leq \CEight$,  and also
$\|p_n-p_0\|_{3,\alpha}\leq n \CSix\delta t\leq \epsFour/2$,
where we used (\ref{3.48}) in the last inequality.  Hence, by Proposition \ref{5.4spm},
 we can define
$p_{n+1}=\mathcal{H}(p_n,\delta  t)$.
It remains to show (\ref{3.48-ind}) with $n+1$ instead of $n$.

 First by induction hypothesis and Lemma \ref{3.16}, since $\|p_n\|_{k+2,\alpha}\leq \CEight\leq C_5$, we can conclude
\begin{equation}\label{3.49}
\|p_{n+1}\|_{k+2,\alpha}\leq \CThirt.
\end{equation}
By induction hypothesis, we also have $\|p_n\|_{k+2,\alpha}\leq \CEight\leq \CThirt$. Hence we can
apply Lemma \ref{3.32} and the induction hypotheses (\ref{3.48-ind}),
 to conclude
\begin{equation}\label{50}
\|\nu_{n+1}\|_{k,\alpha}\leq\|\nu_n\|_{k,\alpha}+\CSix\delta t
\leq \|\nu_0\|_{k,\alpha}+(n+1)\CSix\delta t.
\end{equation}
Also, from (\ref{50}), using that (\ref{3.48})
holds for $n+1$, and that $M=4\|\nu_0\|_{k,\alpha}$,
we have $\|\nu_{n+1}\|_{k,\alpha}\leq \frac{M}{4}+\frac{M}{2}\leq M$.
Now we can use Lemma \ref{3.39} to conclude
\begin{equation*}
\|p_{n+1}\|_{k+2,\alpha}\leq C_4(\|p_0\|_{2,\alpha}+1)\cdot(M+1)=C_5
\end{equation*}
Combining this with (\ref{3.49}), we get
\begin{equation}\label{3.51}
\|p_{n+1}\|_{k+2,\alpha}\leq\min(C_5,\CThirt)=\CEight.
\end{equation}
Furthermore, from (\ref{3.48-ind}) and (\ref{3.49}) we have
$\|p_n\|_{k+2,\alpha}, \|p_{n+1}\|_{k+2,\alpha}\leq \CThirt$, thus
we can apply Lemma \ref{3.18} to get
\begin{equation}\label{3.52}
\|p_{n+1}-p_n\|_{k+1,\alpha}\leq \CSeven\delta t.
\end{equation}
 Hence
\begin{equation}\label{3.53}
\|p_{n+1}-p_0\|_{2,\alpha}\leq\|p_{n+1}-p_n\|_{2,\alpha}+
nC_7\delta t\leq(n+1)\CSeven\delta t
\end{equation}
since $k\geq 1$.
From (\ref{50}), (\ref{3.51}) and (\ref{3.53}) the induction is complete.

Thus for any $\delta t\in (0, \eps_5)$ we showed existence of
$(F_n, p_n)$ integer $n\ge 1$ satisfying  (\ref{3.1}), (\ref{3.1new}).
Moreover, estimates (\ref{3.48-ind}) hold. We note that
(\ref{3.48-ind}) directly implies (\ref{3.55}), (\ref{3.56old}).
Also, (\ref{3.57}) is proved in Lemma \ref{3.29}.

To see (\ref{3.58}), first observe there is no loss of generality to assume
$C\delta t<\frac{1}{200}$ so that we get
$\|F_{n+1}^{-1}-id\|_{k,\alpha}\leq\frac{1}{200}$ from (\ref{3.57}).  Here $C$ is

the maximum of $C_5$ up to $C_9$ appearing in this proof. Then (\ref{3.58})
follows from Lemma \ref{lem6.8b} below by taking $G=F_{n+1}^{-1}$.

Now Proposition
is proved with $T=T_1$.
\end{proof}
\begin{lem}\label{lem6.8b}
Let $G:\mathbb{R}^2\rightarrow\mathbb{R}^2$ be a diffeomorphism, such that $G(x+h)=G(x)+h$, for any $x\in\mathbb{R}^2$ and $h\in\mathbb{Z}^2$.
 Suppose that $G\in C^{k,\alpha}_{loc}(\mathbb{R}^2;\mathbb{R}^2)$ for some $k\geq1$ and $0<\alpha<1$, with $\|G-id\|_{k,\alpha}\leq\frac{1}{200}$. Then $G^{-1}\in C^{k,\alpha}_{loc}(\mathbb{R}^2;\mathbb{R}^2)$ and $\|G^{-1}-id\|_{k,\alpha}\leq C\|G-id\|_{k,\alpha}$.
Here $C$ depends only on $k,\,\alpha$.
\end{lem}
\begin{proof}
First we observe that $G^{-1}\in C_{loc}^{k,\alpha}$, with $\|D(G^{-1})\|_{k-1,\alpha}\leq C$, with $C$ depend only on $k,\,\alpha$.
Indeed, from $\|G-id\|_{k,\alpha}\leq\frac{1}{200}$, we know $\frac{1}{2}\leq\det DG\leq\frac{3}{2}$.
From implicit function theorem, we know $G^{-1}$ is $C^1$ and $D(G^{-1})=(DG)^{-1}(G^{-1})$. From this we get a universal $C^{0,\alpha}$ bound for $D(G^{-1})$. One keeps differentiating and get $\|D(G^{-1})\|_{k-1,\alpha}\leq C$, with $C$ depends only on $k,\,\alpha$.

To see $\|G^{-1}-id\|_{k,\alpha}\leq C\|G-id\|_{k,\alpha}$, we just need to show $\|D(G^{-1})-I\|_{k-1,\alpha}\leq C\|G-id\|_{k,\alpha}$. Observe $D(G^{-1})-I=D(G^{-1})\big(I-(D(G^{-1}))^{-1}\big)=D(G^{-1})(I-DG(G^{-1}))$. Therefore,
\begin{equation*}
\begin{split}
\|D(G^{-1})-I\|_{k-1,\alpha}&\leq C\|D(G^{-1})\|_{k-1,\alpha}\|I-DG\circ G^{-1}\|_{k-1,\alpha}\\
&\leq C\|I-DG\circ G^{-1}\|_{k-1,\alpha}.
\end{split}
\end{equation*}
Here $C$ depends only on $k,\,\alpha$, and we used the calculation $\|D(G^{-1})\|_{k-1,\alpha}\leq C$.
So we just need to show that $\|I-DG\circ G^{-1}\|_{k-1,\alpha}\leq C\|G-id\|_{k,\alpha}$. But if we calculate the $(k-1)-$th derivative of $I-DG\circ G^{-1}$, it will be sum of terms of the form $D^{\beta'}(I-DG)\circ G^{-1}$ multiplied by derivatives of $G^{-1}$ of order$\leq k-1$, with $|\beta'|\leq k-1$. Hence we have the claimed estimate.
\end{proof}

\begin{rem}\label{3.54}
We remark that the quantity $T_1$ in (\ref{3.48}) can be chosen to be uniform for all initial data $p_0$ in some bounded set of $C^{k+2,\alpha}$. Indeed, the quantity $M$ and hence $C_5$ depends only on $\|p_0\|_{k+2,\alpha}$, $c_0$, and $f$, so will be $\CThirt$. Hence the same will be true for $\CSix,\CSeven,\CEight$. As for $\epsFour$, we have seen in Remarked \ref{3.14} that as long as $p_0$ is taken within some compact set in $\tilde{U_2}\subset C_0^{3,\alpha}$, then we can choose $\epsFour$ uniformly. But we assumed $k\geq 2$, hence any bounded set in $C_0^{k+2,\alpha}$ will be precompact in $C^{3,\alpha}$.
\end{rem}


 Finally,  we estimate the continuity in time of the flow map. Define the flow map at $n-$th step to be:
\begin{equation}\label{3.59}
\phi(n\delta t,\cdot)=F_n\circ F_{n-1}\cdots F_0: \bR^2\to \bR^2,
\end{equation}
where $F_0=Id$.
Notice that $\phi(n\delta t,\cdot)(x+h)=\phi(n\delta t,x)+h$ for any $h\in\bZ^2$ and $\phi(n\delta t,\cdot)$ is measure preserving.

We wish to get a flow map which is Lipschitz in $C^{k,\alpha}$ norm. For this we need to estimate the continuity of the discrete flow map in time. We start with the following elementary lemma:
\begin{lem}\label{3.60}

Let $G_1, G_2\in C_{loc}^{k,\alpha}(\bR^2;\bR^2)$, $F\in C_{loc}^{k,\alpha}(\bR^2;\bR^2)$,
with $G_i(x+h)=G_i(x)+h$, $i=1,2$ and $F(x+h)=F(x)+h$ for any $h\in\bZ^2$.
Then $\|G_1\circ F-G_2\circ F\|_{k,\alpha}\leq \CTen(\|F\|_{k,\alpha})\|G_1-G_2\|_{k,\alpha}$.
\end{lem}
\begin{proof}
First notice the obvious estimate:
$\max_{x\in\bT^2}|G_1\circ F(x)-G_2\circ F(x)|\leq\max_{x\in\bT^2}|G_1(x)-G_2(x)|$.

For the derivatives, notice that if $\beta$ is a multi-index with $|\beta|=k$,
then $D^\beta(G_1\circ F)-D^\beta(G_2\circ F)$ is  a sum  of
the terms of the form $D^{\beta'}(G_1-G_2)\circ F$ multiplied by derivatives of $F$ of order $\leq k$, with $1\leq|\beta'|\leq k$.
\end{proof}

With this lemma, we are ready to prove:
\begin{prop}\label{iterMapLip}
There exist constants $T\in(0, T_1]$, $\CElev>0$
depending
only on $\|p_0\|_{k+2,\alpha}$, $f$, $c_0$ and $k$,  such that,
 for any $|\delta t|<\epsFour$ and any corresponding sequence of approximate solutions constructed
 in Proposition \ref{approxSolut-est-Prop}, the following estimate holds for each  $1\leq n\leq\frac{T}{\delta t}$:
\begin{equation}\label{3.64}
\|\phi(n\delta t,\cdot)-\phi((n-1)\delta t)\|_{k,\alpha}\leq \CElev\delta t.
\end{equation}
\end{prop}

\begin{proof}
Applying previous lemma to $G_1=F_n$, $G_2=id$, and $F=\phi((n-1)\delta t$, we can conclude
\begin{equation}\label{3.61}
\|\phi(n\delta t,\cdot)-\phi((n-1)\delta t)\|_{k,\alpha}\leq \CTen(\|\phi((n-1)\delta t)\|_{k,\alpha})
\CNine\delta t.
\end{equation}
Here $C_{10}$ is the constant given by Proposition \ref{approxSolut-est-Prop}. Notice that $\phi(0)=id$, and put $\phi(-\delta t)=id$. Let $T^{\prime}>0$, and $n_0$ be the largest integer with $n_0\delta t\leq T^{\prime}$, we sum (\ref{3.61}) from 0 to $n_0$ and obtain
\begin{equation}\label{3.62}
\sup_{0\leq n\leq\frac{T^{\prime}}{\delta t}}\|\phi(n\delta t,\cdot)-id\|_{k,\alpha}\leq
\CTen(\sup_{1\leq n\leq\frac{T^{\prime}}{\delta t}}
\|\phi(n\delta t,\cdot)\|_{k,\alpha})\CNine\delta t\cdot\frac{T^{\prime}}{\delta t}.
\end{equation}
for any $T^{\prime}>0$.
Let $\hat C$ be the constant $C_{11}$ given by Lemma \ref{3.60} when $F$ has  the bound
$\|F\|_{k,\alpha}\leq 2+\|id\|_{k,\alpha}$.
Then, choosing
$T'$ is small
so that $\hat C C_{10}T^{\prime}\leq 2$,
 we see by  (\ref{3.62}) that
 \begin{equation}\label{3.63}
\|\phi(n\delta t,\cdot)\|_{k,\alpha}\leq 2+\|id\|_{k,\alpha},
\end{equation}
for $0\leq n\leq \frac{T^{\prime}}{\delta t}$.
Let  $T=\min(T_1,T^{\prime})$. Then $T$ depends
only on $\|p_0\|_{k+2,\alpha}$, $f$, $c_0$ and $k$,
and (\ref{3.63}) holds
for $1\leq n\leq\frac{T}{\delta t}$.
For such $n$,  denoting $\CElev:=\hat C C_{10}$, we get from (\ref{3.61})
\begin{equation*}
\|\phi(n\delta t,\cdot)-\phi((n-1)\delta t)\|_{k,\alpha}\leq \CElev\delta t.
\end{equation*}
Hence the proof is complete.
\end{proof}

\section{Passage to the limit}
\label{passageLim-Sect}
\subsection{Determine flow map and pressure by taking limit of time-stepping approximations}
\label{passageLim-Sect-sub1}

Here we make use of the estimates obtained in (\ref{3.55})--(\ref{3.58}), (\ref{3.63}),
(\ref{3.64}) and pass to the limit in the following equation
\begin{equation}\label{4.1}
x+f^{-1}(x)f^{-1}(F_{n+1}^{-1}(x))\nabla p_{n+1}(x)=F_{n+1}^{-1}(x)+f^{-2}R_{f\delta t}(\nabla p_n)(F_{n+1}^{-1}(x)).
\end{equation}

For $t\in[0,T]$, let $\phi^{\delta t}_t:\mathbb{R}^2\rightarrow\mathbb{R}^2$ be a map
defined by:
\begin{equation}\label{7.2n}
\phi^{\delta t}_t(x)=\phi(n\delta t,x),\;\textrm{ if }\; n\delta t\leq t<(n+1)\delta t,
\end{equation}
where $\phi(n\delta t,\cdot)$ is defined by (\ref{3.59}).
We also define:
\begin{equation}\label{7.3n}
\bar{\phi}^{\delta t}_t(x)=\frac{t-n\delta t}{\delta t}\,\phi((n+1)\delta t,x)+
\frac{(n+1)\delta t-t}{\delta t}\,\phi(n\delta t,x),\;\textrm{ if }\;
n\delta t\leq t<(n+1)\delta t.
\end{equation}

The estimates (\ref{3.64}) show that the map $\bar{\phi}^{\delta t}_t$ is
Lipschitz in time under the norm $C^{k,\alpha}$. Indeed, it follows from
(\ref{3.64}) that
$\|\bar{\phi}^{\delta t}_t-\bar{\phi}^{\delta t}_s\|_{k,\alpha}\leq \CElev|t-s|$,
for any $0\leq s\leq t\leq T$, and $|\delta t|<\epsFour$.  Now we can use
Arzela-Ascoli Theorem applied to $\bar{\phi}^{\delta t}_t$ restricted to $[0,1]^2$
to get a subsequence of $\delta t_j\rightarrow0$, and a limit map $\phi_t$
Lipschitz in time under the norm $C^{k,\alpha}$, such that
\begin{equation}\label{4.2-0}
\bar{\phi}^{\delta t_j}_t\rightarrow\phi_t\textrm{   under the norm $C^{k,\beta}$
for each $0<\beta<\alpha$, uniformly for $t\in[0,T]$}.
\end{equation}
Also one observes that for $n\delta t\leq t<(n+1)\delta t$, one has
$$
\|\bar{\phi}^{\delta t}_t-\phi^{\delta t}_t\|_{k,\alpha}\leq
\|\phi(n\delta t,\cdot)-\phi((n-1)\delta t,\cdot)\|_{k,\alpha}\leq \CElev\delta t.
$$
Therefore the convergence obtained in (\ref{4.2-0}) will hold also for
$\phi^{\delta t}(t)$:
\begin{equation}\label{4.2}
\phi^{\delta t_j}_t\rightarrow\phi_t\textrm{   under the norm $C^{k,\beta}$
for each $0<\beta<\alpha$, uniformly for $t\in[0,T]$}.
\end{equation}

Same argument works also when we use $\phi(n\delta t,\cdot)^{-1}$
in (\ref{7.2n}), (\ref{7.3n}), to define families of maps $\psi^{\delta t}_t$ and
$\bar{\psi}^{\delta t}_t$.
Then we can pass to the limit as
in (\ref{4.2}), and conclude the limiting
map $\psi_t$ will also be Lipschitz in time under the norm $C^{k,\alpha}$,
and that ${\psi}^{\delta t_j}_t\rightarrow\psi_t$.
Since the argument for convergence of ${\psi}^{\delta t_j}$ can be performed by taking a subsequence
of the sequence $\{\delta t_j\}$ used in (\ref{4.2}), we can assume without loss of generality
that convergence ${\phi}^{\delta t_j}_t\rightarrow\phi_t$ and
${\psi}^{\delta t_j}_t\rightarrow\psi_t$ holds along the same sequence $\delta t_j\to 0$.
Then it is not difficult to observe that $\psi_t=\phi_t^{-1}$ for each $t\in[0, T]$.
Indeed, for each
$n$ and $j$, with $n\delta t_j\leq T$, one has
$\phi^{\delta t_j}(n\delta t_j,\psi^{\delta t_j}(n\delta t_j,x))=x$.
Fix any $t\in[0,T]$ and $j\in\bN$, find $n$ such that
$n\delta t_j\leq t<(n+1)\delta t_j$. Now one can conclude that $\phi(t,\psi(t,x))=x$, which follows from the uniform
convergence obtained above and the Lipschitz continuity in time of the
limit. Same argument also shows
$\psi(t,\phi(t,x))=x$. Hence the maps $\psi(t)$ and $\phi(t)$ obtained above
are inverse to each other. Thus we obtained that
$\phi_t:\bT^2\to\bT^2$ is a diffeomorphism  for each $t\in[0, T)$, and
\begin{equation}\label{flowMapDiffeo}
\|\phi_t, \phi^{-1}_t\|_{k,\alpha}\leq C, \textrm{for each $t\in[0,T]$.}
\end{equation}

Also notice that from (\ref{3.64}),
\begin{equation}\label{insert}
\|\phi_t-\phi_s\|_{k,\alpha}\leq C|t-s|, \textrm{for any $s,t\in[0,T]$.}
\end{equation}
 Also since the map $\phi(n\delta t,\cdot)$ in (\ref{3.59})
 is measure-preserving, then by (\ref{7.2n}) and (\ref{4.2}),
\begin{equation}\label{flowMapMeas}
\textrm{map $\phi_t:\bT^2\to\bT^2$ is $\Lm^2$-measure preserving for each $t\in[0,T]$.}
\end{equation}

Similarly we may consider $p^{\delta t}(t)$. By the estimate (\ref{3.55})-(\ref{3.56old}), and similar argument as above, we can use Arzela-Ascoli to get a limit function $p_t$, which is Lipschitz in time under the norm $C^{k+1,\alpha}$ such that
\begin{equation}\label{4.3}
p^{\delta t_j}_t\rightarrow p_t \textrm{  under the norm
$C^{k+1,\alpha}(\bT^2)$,  uniformly in $t\in[0,T]$ }.
\end{equation}
Then, by (\ref{3.56old}),
\begin{equation}\label{p-Lip}
\|p_t-p_s\|_{k+1,\alpha}\leq C|t-s|, \textrm{for any $s,t\in[0,T]$.}
\end{equation}

Denote $\mathbf{u}_g(n\delta t,x)=f^{-1}(x)J\nabla p(n\delta t,x)$. In Lagrangian coordinates, define
\begin{equation}\label{4.4}
\mathbf{v}_{g}^{\delta t}( t,x)=
f^{-1}(\phi^{\delta t}(t,x))J\nabla p^{\delta t}(t,\phi^{\delta t}( t,x)),
\quad
\mathbf{v}_{g}( t,x)=
f^{-1}(\phi(t,x))J\nabla p(t,\phi( t,x)).
\end{equation}
Then $\mathbf{v}_{g}^{\delta t}$ is a piece-wise constant in time function.
From (\ref{4.2}), (\ref{4.3}) we conclude:
\begin{equation}\label{4.5}
\begin{split}
\mathbf{v}_{g}^{\delta t_j}(t,\cdot)\rightarrow \mathbf{v}_g(t,\cdot)\;
&\textrm{ under the norm $C^{k,\beta}(\bT^2;\bR^2)$ for each $0<\beta<\alpha$,}\\
&\textrm{ uniformly in $t\in[0,T]$}.
\end{split}
\end{equation}
Also, by (\ref{insert}) and (\ref{p-Lip})
\begin{equation}\label{4.5-0}
\mathbf{v}_g(t,\cdot) \textrm{ is Lipschitz in time under the norm $C^{k,\alpha}(\bT^2)$.}
\end{equation}

With above preparation, we are ready to derive the equations in the limit.
\subsection{Derive the equation in the limit}
\label{passageLim-Sect-sub2}

In (\ref{4.1}), replace $x$ by $\phi^{\delta t}((n+1)\delta t,x)$, and
recall $\phi^{\delta t}((n+1)\delta t,\cdot)=F_{n+1}\circ\phi^{\delta t}(n\delta t,\cdot)$.
Then,  
recalling the notation $\phi_t(x)=\phi(t,x)$,
we get for each $x\in\bR^2$:
\begin{equation*}
\begin{split}
\phi^{\delta t}_{(n+1)\delta t}+&f^{-1}(\phi^{\delta t}_{(n+1)\delta t})f^{-1}(\phi^{\delta t}_{n\delta t})\nabla p^{\delta t}_{n+1}(\phi^{\delta t}_{(n+1)\delta t})\\
&=\phi^{\delta t}_{n\delta t}+\big[f^{-2}R_{f\delta t}(\nabla p^{\delta t}_n)\big](\phi^{\delta t}_{n\delta t}).
\end{split}
\end{equation*}
In the last term on the right hand side, all the functions are evaluated at $\phi_{n\delta t}^{\delta t}(x)$. Using (\ref{4.4}),
we rewrite the last equality as
\begin{equation}\label{4.7}
\begin{split}
\phi^{\delta t}_{(n+1)\delta t}-&\phi^{\delta t}_{n\delta t}+
f^{-1}(\phi^{\delta t}_{n\delta t})(-J)(\mathbf{v}_{g,(n+1)\delta t}^{\delta t}-
\mathbf{v}_{g,n\delta t}^{\delta t})\\
&=\big[f^{-1}
(R_{f\delta t}-I)\big](\phi^{\delta t}_{n\delta t})(-J\mathbf{v}_{g,n\delta t}^{\delta t}).
\end{split}
\end{equation}
Fix $0\leq s<t< T^{\prime}$, and suppose $k_1\delta t\leq s<(k_1+1)\delta t$,
and $k_2\delta t\leq t<(k_2+1)\delta t$. With $\delta t$ sufficiently small,
we have $k_1<k_2$. Sum (\ref{4.7}) from $k_1$ to $k_2-1$, we get  for each $x\in\bR^2$:
\begin{equation}\label{4.8}
\begin{split}
\phi^{\delta t}_{k_2\delta t}-&\phi^{\delta t}_{k_1\delta t}+
\sum_{n=k_1}^{k_2-1}f^{-1}(\phi^{\delta t}_{n\delta t})(-J)(\mathbf{v}_{g,(n+1)\delta t}^{\delta t}-\mathbf{v}_{g,n\delta t}^{\delta t})\\
&=\sum_{n=k_1}^{k_2-1}\big[f^{-1}
(R_{f\delta t}-I)\big](\phi^{\delta t}_{n\delta t})(-J\mathbf{v}_{g,n\delta t}^{\delta t}).
\end{split}
\end{equation}

Now we take the limit in (\ref{4.8}) along the sequence $\delta t_j\rightarrow 0$
 from (\ref{4.2}) and (\ref{4.3}). We will suppress the index $j$,
 and write $\delta t\rightarrow 0$ meaning the above sequence.

As $\delta t\rightarrow 0$, we get
$\phi^{\delta t}_{k_2\delta t}-\phi^{\delta t}_{k_1\delta t}\rightarrow\phi_t-\phi_s$ in $C^{k,\beta}(\bT^2)$
for any $\beta<\alpha$ by (\ref{4.2}).

For the term on the right hand side of  (\ref{4.8}), one can get  for each $x\in\bR^2$:
\begin{equation}\label{4.9}
\begin{split}
&\qquad\sum_{n=k_1}^{k_2-1}\big[f^{-1}
(R_{f\delta t}-I)\big](\phi^{\delta t}_{n\delta t})(-J\mathbf{v}_{g,n\delta t}^{\delta t})\\
&=\int_{k_1\delta t}^{k_2\delta t}\big[f^{-1}\frac{R_{f\delta t}-I}{\delta t}\big]
(\phi^{\delta t}_{\tau})(-J\mathbf{v}_{g,\tau}^{\delta t})d\tau
\rightarrow\int_s^tf^{-1}(\phi_{\tau}) Jf(\phi_{\tau})(-J\mathbf{v}_{g,\tau})d\tau\\
&=\int_s^t\mathbf{v}_{g,\tau}d\tau\textrm{ as $\delta t\rightarrow0$},
\end{split}
\end{equation}
where the convergence is obtained as following. By (\ref{4.2}), (\ref{4.5})
there exists $C$ such that $\|\phi^{\delta t_j}_t\|_{k, \alpha/2, \bT^2}\le C$
and  $\|\mathbf{v}_{g,t}^{\delta t_j}\|_{k, \alpha/2, \bT^2}\le C$ for each $t\in[0,T]$,
$j=1,2,\dots$.
Then, using (\ref{4.2}), (\ref{4.5}) again, we see that for each $x\in \bT^2$
the integrands (as functions of $t\in[0,T]$) are uniformly with respect to $j$
bounded, and converge pointwise on $[0,T]$
to $f^{-1}(\phi_{\tau}) Jf(\phi_{\tau})(-J\mathbf{v}_{g,\tau})$.
Then we apply the Dominated Convergence Theorem.

It only remains to deal with the second term on the left hand side of (\ref{4.8}).
For this we define
$(\mathcal{D}_t\mathbf{v}_g)_{\delta t}(t,\cdot)=
\frac{\mathbf{v}_g^{\delta t}((n+1)\delta t,\cdot)-\mathbf{v}_g^{\delta t}(n\delta t,\cdot)}{\delta t}$
 for $t\in[n\delta t,(n+1)\delta t)$. Then $(\mathcal{D}_t\mathbf{v}_g)_{\delta t}(t)$ is
 uniformly bounded in $C^{k,\alpha}$ for $0\leq\delta t<\epsFour$ and $t\in[0, T]$.
 This follows from our definition (\ref{4.4}) and estimates (\ref{3.56old}) and (\ref{3.64}). Then we can write
\begin{equation}\label{4.10}
\begin{split}
&\qquad\sum_{n=k_1}^{k_2-1}f^{-1}(\phi^{\delta t}_{n\delta t})(\mathbf{v}_{g,(n+1)\delta t}^{\delta t}-\mathbf{v}_{g,n\delta t}^{\delta t})\\
&=\int_{k_1\delta t}^{k_2\delta t}f^{-1}(\phi^{\delta t}_{\tau})(\mathcal{D}_t\mathbf{v}_g)_{\delta t}(\tau,\cdot)dt\\
&=\int_{k_1\delta t}^{k_2\delta t}[f^{-1}(\phi^{\delta t}_{\tau})-f^{-1}(\phi_{\tau})](\mathcal{D}_t\mathbf{v}_g)_{\delta t}(\tau,\cdot)dt+\int_{k_1\delta t}^{k_2\delta t}f^{-1}(\phi_{\tau})(\mathcal{D}_t\mathbf{v}_g)_{\delta t}(\tau,\cdot)dt.
\end{split}
\end{equation}
In the first term  above,  we have
$f^{-1}(\phi^{\delta t}_{\tau})-f^{-1}(\phi_{\tau})\to 0$  in $C^{k,\beta}(\bT^2)$
for any $\beta<\alpha$ uniformly in $\tau\in[0, T]$
by (\ref{4.2}), and $(\mathcal{D}_t\mathbf{v}_g)_{\delta t}(\tau,\cdot)$ is bounded
in $C^{k,\alpha}(\bT^2)$ uniformly in $\tau\in[0, T]$ as we showed above. Thus the first term
converges to zero for each $x\in\bT^2$.

So we can focus on the second term. We do an "integration by parts":
\begin{equation}\label{4.11}
\begin{split}
&\quad\int_{k_1\delta t}^{k_2\delta t}f^{-1}(\phi_{\tau})(\mathcal{D}_t\mathbf{v}_g)_{\delta t}(\tau,\cdot)d\tau=\sum_{n=k_1}^{k_2-1}\int_0^{\delta t}f^{-1}(\phi_{\tau+n\delta t})\frac{\mathbf{v}_{g,(n+1)\delta t}^{\delta t}-\mathbf{v}_{g,n\delta t}^{\delta t}}{\delta t}d\tau\\
&=\frac{1}{\delta t}\big[\sum_{n=k_1+1}^{k_2}\int_0^{\delta t}f^{-1}(\phi_{\tau+(n-1)\delta t})\mathbf{v}_{g,n\delta t}^{\delta t}-\sum_{n=k_1}^{k_2-1}   \int_0^{\delta t}f^{-1}(\phi_{\tau+n\delta t})\mathbf{v}_{g,n\delta t}^{\delta t}d\tau\big]\\
&=\delta t^{-1}[\int_0^{\delta t}f^{-1}(\phi_{\tau+(k_2-1)\delta t})d\tau\cdot\mathbf{v}_{g,k_2\delta t}^{\delta t}-\int_0^{\delta t}f^{-1}(\phi_{\tau+k_1\delta t})dt\cdot\mathbf{v}_{g,k_1\delta t}^{\delta t}]\\
&+\sum_{n=k_1+1}^{k_2-1}\int_0^{\delta t}\frac{1}{\delta t}[f^{-1}(\phi_{\tau+(n-1)\delta t})-f^{-1}(\phi_{\tau+n\delta t})]dt\mathbf{v}_{g,n\delta t}^{\delta t}.
\end{split}
\end{equation}
The first two terms will converge to $f^{-1}(\phi_t)\mathbf{v}_{g,t}$ and
$f^{-1}(\phi_s)\mathbf{v}_{g,s}$ in $C^{k,\beta}(\bT^2)$ for any $\beta<\alpha$ respectively. For the last term, we fix some $x\in\mathbb{T}^2$ and have
\begin{equation}\label{4.12}
\begin{split}
&\quad\sum_{n=k_1+1}^{k_2-1}\int_0^{\delta t}\frac{1}{\delta t}
[f^{-1}(\phi_{\tau+(n-1)\delta t})-f^{-1}(\phi_{\tau+n\delta t})]dt\mathbf{v}_{g,n\delta t}^{\delta t}\\
&=\sum_{n=k_1+1}^{k_2-1}\int_{(n-1)\delta t}^{n\delta t}\frac{f^{-1}(\phi_{\tau})-f^{-1}(\phi_{\tau+\delta t})}{\delta t}d\tau\mathbf{v}_{g,n\delta t}^{\delta t}\\
&=\int_{k_1\delta t}^{(k_2-1)\delta t}\frac{f^{-1}(\phi_{\tau})-f^{-1}(\phi_{\tau+\delta t})}{\delta t}\mathbf{v}_{g,\tau+\delta t}d\tau.
\end{split}
\end{equation}
Now we wish to take the limit of this term as $\delta t\rightarrow0$. We have seen
in (\ref{insert}) that $\phi_t$ is Lipschitz in $t$ under $C^{k,\alpha}(\bT^2)$ norm,
hence $f^{-1}(\phi_t)$ will be Lipschitz, and for a.e $\tau\in[0,T)$, we have
$\frac{f^{-1}(\phi_{\tau})-f^{-1}(\phi_{\tau+\delta t})}{\delta t}
\rightarrow\partial_t(f^{-1}(\phi_{\tau})$.
Then, using also (\ref{4.5-0}), we find that
 the integrand will converge to
$-\partial_t(f^{-1}(\phi_{\tau}))\mathbf{v}_{g,\tau}$ for a.e $\tau\in [s,t]$.
Also we know the integrand is bounded on $[0, T]$ by some constant independent of $\delta t$,
by the said Lipschitz continuity of $f^{-1}(\phi_t))$ and also the boundedness of
$\mathbf{v}_{g}^{\delta t}$.
Hence we can apply the dominated convergence theorem to conclude that the integral
converges to $-\int_s^t\partial_t(f^{-1}(\phi(t,x))\mathbf{v}_g(t,x)dt$.

Combining all calculations above, we conclude that in the limit we have for any fixed $x\in\mathbb{T}^2$:
\begin{equation}\label{4.13}
\begin{split}
&\phi_t-\phi_s+(-J)[f^{-1}(\phi_t)\mathbf{v}_{g,t}-f^{-1}(\phi_{s})\mathbf{v}_g(s)\\
&\qquad\qquad-\int_s^t\partial_{\tau}(f^{-1}(\phi_{\tau}))\mathbf{v}_{g,\tau}d\tau]=\int_s^t\mathbf{v}_{g,\tau}d\tau.
\end{split}
\end{equation}
Or equivalently
\begin{equation}\label{4.14}
\phi_t-\phi_s+
(-J)\int_s^tf^{-1}(\phi_{\tau})\partial_{\tau}\mathbf{v}_g(\tau,\cdot)d\tau=
\int_s^t\mathbf{v}_{g,\tau}d\tau.
\end{equation}

Next we take the limit as $t\rightarrow s$. In the limit, we obtain the following:
\begin{prop}\label{p7.1}
For $\mathcal{L}^1-a.e$ $s\in[0,T]$, and for any $x\in\mathbb{T}^2$, it holds:
\begin{equation}\label{7.21new}
\partial_t\phi(s,\cdot)+(-J)f^{-1}(\phi_s)\partial_t\mathbf{v}_g(s,\cdot)=\mathbf{v}_g(s,\cdot),\\
\end{equation}
\begin{equation}\label{7.22new}
\begin{split}
\partial_t\phi(s,x)=\big[I+f^{-2}D^2p+f^{-1}\nabla p\otimes\nabla(f^{-1})\big]^{-1}&(s,\phi(s,x))\cdot(f^{-1}J\nabla p)(s,\phi(s,x))\\
&-f^{-2}\partial_t\nabla p(s,\phi(s,x)).
\end{split}
\end{equation}
\end{prop}

\begin{proof}
From (\ref{4.5-0}) we conclude that
$\lim_{t\rightarrow s+}\frac{\mathbf{v}_g(t,x)-\mathbf{v}_g(s,x)}{t-s}$ exists
for $\mathcal{L}^3$-a.e $(s,x)\in[0,T)\times\bT^2$.
Now we fix some $s\in[0,T)$, such that the above limit exists
for ${\mathcal L}^2$-a.e  $x\in\times \bT^2$. Note that this property
holds for a.e.  $s\in[0,T)$, by Fubini's theorem.

For $s$ chosen above and $t\ne s$,
we divide (\ref{4.14}) by $t-s$, and  take a limit as $t\to s$.
Then the first term in the left-had side of
(\ref{4.14}) will tend to $\partial_t\phi(s,x)$,
for $\mathcal{L}^2-$a.e $x\in\mathbf{T}^2$ because of our choice of $s$,
while the term on the right hand side of (\ref{4.14}) will tend to $\mathbf{v}_g(s,x)$,
because of the Lipschitz continuity of $\mathbf{v}_g$ with respect to time-variable.
It remains to consider the integral in the left-hand side.
For this purpose,  we calculate for each $x\in\bT^2$, and any $0<s<t<T$:
\begin{equation*}
\begin{split}
\int_s^tf^{-1}(\phi_{\tau})\partial_{\tau}\mathbf{v}_g(\tau,\cdot)d\tau
=&\int_s^t\big(f^{-1}(\phi_{\tau})-f^{-1}(\phi_s)\big)\partial_{\tau}\mathbf{v}_g(\tau,\cdot)
d\tau\\
+&f^{-1}(\phi_s)\int_s^t\partial_{\tau}\mathbf{v}_g(\tau,\cdot)
d\tau\\
=\int_s^t\big(f^{-1}(\phi_{\tau})-f^{-1}(\phi_s)\big)\partial_{\tau}&\mathbf{v}_g(\tau,\cdot)
d\tau+f^{-1}(\phi_s)(\mathbf{v}_{g,t}-\mathbf{v}_{g,s}).
\end{split}
\end{equation*}
In the last line above, we used the fact that for any
$x\in\bT^2$, $t\longmapsto\mathbf{v}_g(t,x)$ is Lipschitz in $t$, and so the usual
Newton-Leibnitz formula holds.
Divide las line above by $t-s$, and let $t\rightarrow s$.
We observe that the first term will tend to zero, because $f^{-1}(\phi(t,x))$ is
Lipschitz in $t$, with a Lipschitz constant independent of $x$,
while $\partial_{\tau}\mathbf{v}_g(\tau,\cdot)$ is bounded.
The second term will tend to $\partial_t\mathbf{v}_g(s,x)$,
by the choice of $(s, x)$.
Therefore, after dividing by $t-s$ and letting $t\rightarrow s$,
we see above will tend to $f^{-1}(\phi_s)\partial_t\mathbf{v}_g(s,\cdot)$.

Now we can conclude that for $\mathcal{L}^1-a.e$ $s\in[0,T]$:
\begin{equation}\label{4.15}
\partial_t\phi(s,\cdot)+(-J)f^{-1}(\phi_s)\partial_t\mathbf{v}_g(s,\cdot)=\mathbf{v}_g(s,\cdot),\textrm{ for $\mathcal{L}^2-$a.e $x\in\mathbb{T}^2$}.
\end{equation}

On the other hand, $\frac{\phi_t-\phi_s}{t-s}$ is uniformly bounded in
$C^{k,\alpha}$ for any $t>s$ by (\ref{insert}).
Hence for any sequence $t_j\rightarrow s$ there exists a subsequence $t_{j_m}$
such that
$\frac{\phi_{t_{j_m}}-\phi_s}{t_{j_m}-s}$ converges in $C^{k,\beta}$ for any
$\beta<\alpha$, and the limit function
is equal to $\partial_t\phi(s,x)$ for a.e. $x\in\bT^2$
by the choice of $s$. Thus $C^{k,\beta}$-limit on $\bT^2$ does not depend on the
choice of the sequence and subsequence.
This shows that $\partial_t\phi(s,x)$
exists for all $x\in\bT^2$. Thus we proved that
\begin{equation}\label{4.15-xx}
\partial_t\phi(s,\cdot)\in C^{k,\alpha}(\bT^2) \;\textrm{ with }\;
\|\partial_t\phi(s,\cdot)\|_{C^{k,\alpha}}\le C
\;\textrm{ for a.e. }s\in[0, T],
\end{equation}
where $C$ does not depend on $s$.
We can apply similar arguments to $p_t$ using (\ref{p-Lip}), thus obtain:
\begin{equation}\label{4.15-xx-p}
\begin{split}
&\textrm{ for a.e. }s\in[0, T], \;
\partial_tp(s,x) \textrm{
exists for all $x\in\bT^2$ and }\\
&\partial_t p(s,\cdot)\in C^{k+1,\alpha}(\bT^2) \;\textrm{ with }\;
\|\partial_t p(s,\cdot)\|_{C^{k+1,\alpha}}\le C,
\end{split}
\end{equation}
where $C$ does not depend on $s$.
From (\ref{4.15-xx}), (\ref{4.15-xx-p}), we conclude that
for a.e. $s\in[0, T)$,
$\partial_t\nabla p(s,\cdot), \mathbf{v}_g(s,\cdot)\in C^{k,\alpha}$, and the usual chain
rule holds.
Then we obtain from (\ref{4.4}):
\begin{equation}
\begin{split}\label{chainrule}
\partial_t\mathbf{v}_g(s,x)=&J\big[f^{-1}(\phi(s,x))\partial_t\nabla p(s,\phi(s,x))+
f^{-1}(\phi(s,x))D^2p(s,\phi(s,x))\partial_t\phi(s,x)\\
&\qquad+\nabla p(s,\phi(s,x))\otimes\nabla(f^{-1})(\phi(s,x))\partial_t\phi(s,x)\big],\\
&\qquad\textrm{for $\mathcal{L}^2-$a.e $x\in\mathbb{T}^2$.}
\end{split}
\end{equation}
Observe that each term in both sides of (\ref{chainrule}) are at least $C^{k,\alpha}$, we see that (\ref{chainrule}) holds for every $x\in\mathbb{T}^2$.
Similar arguments also works for (\ref{4.15}), and we can conclude that for $\mathcal{L}^1-a.e$ $s\in[0,T]$, (\ref{4.15}) holds for every $x\in\mathbb{T}^2$.

By plugging (\ref{chainrule}) into (\ref{4.15}), we obtain for $\mathcal{L}^1-a.e$ $s\in[0,T)$:
\begin{equation}
\begin{split}
\big[I+f^{-2}D^2p+f^{-1}\nabla p\otimes\nabla(f^{-1})\big]
&(s,\phi(s,x))\partial_t\phi(s,x)\\
&=(f^{-1}J\nabla p)(s,\phi(s,x))-f^{-2}\partial_t\nabla p(s,\phi(s,x)),\\
&\textrm{ for every $x\in\mathbb{T}^2$}.
\end{split}
\end{equation}

By our construction, we have $I+f^{-2}D^2p(s,\cdot)+f^{-1}\nabla p(s,\cdot)\otimes\nabla (f^{-1})>\frac{c_0}{2}$, so we can invert this matrix, and obtain
\begin{equation}
\partial_t\phi(s,x)=\big[I+f^{-2}D^2p+f^{-1}\nabla p\otimes\nabla(f^{-1})\big]^{-1}(s,\phi(s,x))\cdot(f^{-1}J\nabla p)(s,\phi(s,x))-f^{-2}\partial_t\nabla p(s,\phi(s,x)).
\end{equation}
\end{proof}

\subsection{Regularity in time}
\label{passageLim-Sect-sub3}

In this section we show that the limiting equations (\ref{7.21new}), (\ref{7.22new}) hold
for any $(s,x)\in[0,T)\times\mathbb{T}^2$. For that, we need to show regularity of
$p$ and $\phi$ in time.

Fix any $\zeta\in C^1(\mathbb{T}^2)$. Using the measure-preserving property of
the
 map $\phi_t$ in (\ref{flowMapMeas}), we obtain
\begin{equation}
\begin{split}
&0=\left(\frac{d}{dt}\int_{\mathbb{T}^2}\zeta(\phi(t,x))dx\right)_{|t=s}=\int_{\mathbb{T}^2}
\nabla\zeta(\phi(s,x))\partial_t\phi(s,x)dx\\
&=\int_{\mathbb{T}^2}\nabla\zeta(\phi(s,x))\big[I+f^{-2}D^2p+f^{-1}\nabla p\otimes\nabla(f^{-1})\big]^{-1}\cdot(f^{-1}J\nabla p-f^{-2}\partial_t\nabla p)(s,\phi(s,x))dx\\
&=\int_{\mathbb{T}^2}\nabla\zeta(x)\big[I+f^{-2}D^2p+f^{-1}\nabla p\otimes\nabla(f^{-1})\big]^{-1}\cdot(f^{-1}J\nabla p-f^{-2}\partial_t\nabla p)(s,x)dx.
\end{split}
\end{equation}
That is, $\partial_tp(t,\cdot)$ solves the following uniformly elliptic equation of divergence
form for a.e $t\in[0,T)$:
\begin{equation}\label{7-23}
\nabla\cdot[(I+f^{-1}D(f^{-1}Dp))^{-1}f^{-2}\nabla (\partial_t p)]=\nabla\cdot[(I+f^{-1}D(f^{-1}Dp))^{-1}(f^{-1}J\nabla p)].
\end{equation}
Writing this equation in the form
\begin{equation}\label{7-23a}
\nabla\cdot[A(t,x)\nabla (\partial_t p)]=
\nabla\cdot F(t,x),
\end{equation}
Here $A(t,x)=(I+f^{-1}D(f^{-1}Dp))^{-1}(t,x)$. We have that $\|A(t, \cdot), F(t, \cdot)\|_{C^{k, \alpha}(\bT^2)}\le C$ for a.e. $t\in[0,T)$.
Then the standard estimates for linear elliptic equations imply
$\|\partial_t p(t, \cdot)\|_{C^{k+1, \alpha}(\bT^2)}\le C$ for a.e. $t\in[0,T)$. Next we want to obtain higher regularity in time.
\begin{prop}
For any $(s,x)\in[0,T)\times\mathbb{T}^2$, and $1\leq m\leq k+1$,
$\partial_t^mp(s,x)$ exists,
$\partial_t^mp\in L^{\infty}([0,T);C^{k+2-m,\alpha}(\mathbb{T}^2))$, and
(\ref{7.21new}), (\ref{7.22new}) hold for any $(s,x)\in[0,T)\times\mathbb{T}^2$.
\end{prop}

\begin{proof}
Formally we can differentiate (\ref{7-23}) in $t$,
and obtain a similar elliptic equation for $\partial_t^2p_t$, and apply the $C^{k,\alpha}$ estimates. To proceed rigorously, we take any $t_1<t_2$ for which (\ref{7-23}) holds true, and take their difference:
\begin{equation}\label{7-24}
\begin{split}
&\nabla\cdot\bigg(A(t_1)f^{-2}\nabla\big(\frac{\partial_tp
(t_1)-\partial_tp(t_2)}{t_1-t_2}\big)\bigg)=-\nabla\cdot\bigg(\frac{A(t_1)-A(t_2)}{t_1-t_2}f^{-2}\partial_t\nabla p(t_2)\bigg)\\
&+\nabla\cdot\bigg(A(t_1)f^{-1}J\frac{\nabla p(t_1)-\nabla p(t_2)}{t_1-t_2}+\frac{A(t_1)-A(t_2)}{t_1-t_2}f^{-1}J\nabla p(t_2)\bigg),
\end{split}
\end{equation}
where $A(t)(x)=A(t,x)$ for $A(t,x)$ defined above.
Recall that we have the estimate $A(t_1)$ in $C^{k,\alpha}(\bT^2)$. Also,
by  (\ref{p-Lip})we have an
estimate for
$\frac{A(t_1)-A(t_2)}{t_1-t_2}$ in $C^{k-1,\alpha}$
independent of $t_1, t_2\in [0,T]$. Thus the right hand side
of (\ref{7-24}) has form $\nabla\cdot G(x)$, where
$\|G(\cdot)\|_{C^{k-1,\alpha}}\le C$
for $C$ independent of $t_1, t_2\in [0,T]$. By $C^{1,\alpha}$ estimates
for elliptic equations, we
obtain
\begin{equation}\label{p-seconDiffQuot-est}
\|\frac{\partial_tp
(t_1)-\partial_tp(t_2)}{t_1-t_2}\|_{C^{k,\alpha}(\bT^2)}\le C
\;\textrm{ for a.e. }\; t_1, t_2\in [0,T),
\end{equation}
where $C$ is independent of $t_1, t_2\in [0,T)$.


Hence  $\partial_tp$ exists for all
$(t,x)\in[0,T)\times\mathbb{T}^2$: indeed fix $x\in\bT^2$.
From (\ref{p-seconDiffQuot-est}), functions $g_h(t)=\frac{p(t+h, x)-p(t,x)}h$
defined on $[0, T-|h|]$
for $|h|<T/2$, are equiLipschitz on their common domains,
thus from any sequence
 $h_i\to 0$ one can extract a subsequence $h_{i_k}$ so that
 $g_{h_{j_k}}(\cdot)$
converge to a uniform limit $g^\infty\in C([0,T)$
 on compact subsets of $[0, T)$. Also, (\ref{4.15-xx-p})
 implies that $g^\infty(t)=\partial_tp(t,x)$ for a.e.
 $t$, thus  $g^\infty(t)$ is independent of the subsequence $h_{i_j}$
 for such $t$.
It follows that the uniform limit $g^\infty(\cdot)$ is independent of the subsequence
$h_{i_k}$, i.e. $\lim_{h\to 0}\frac{p(t+h, x)-p(t,x)}h$ exists for each
$t\in [0, T)$.
Thus we proved that for each $x\in \bT^2$, $\partial_tp(\cdot, x)$ exists for all
$t\in[0,T)$. Also, now (\ref{p-seconDiffQuot-est}) holds for
all $t_1, t_2\in [0,T)$.

Here we can see (\ref{7.22new}) holds for any $(s,x)\in[0,T)\times\mathbb{T}^2$, since the right hand side of (\ref{7.22new}) is continuous in $s$.
The continuity of the term $\partial_t\nabla p$ comes from the estimate (\ref{p-seconDiffQuot-est}).
From (\ref{chainrule}), we know $\partial_t\mathbf{v}_g$ is also continuous in $s$ for each fixed $x$. Hence (\ref{7.21new}) is true for all $s$.

Then we let $t_2\rightarrow t_1$
in   (\ref{p-seconDiffQuot-est}), and repeating the
proof of
 (\ref{4.15-xx})--(\ref{4.15-xx-p}), we obtain that
for a.e. $t\in[0,T)$ the
$\partial_t^2p(t,x)$ exists for all $x\in\mathbb{T}^2$, and
$\|\partial_t^2p(t,\cdot)\|_{C^{k,\alpha}(\bT^2)}\le C$ for those $t$,
with $C$ independent of $t$.
Then we can repeat the previous argument and obtain the claimed regularity for
$\partial_t^{k+2-m}p(t,\cdot)$ for $0\leq m\leq k+1$.
From (\ref{7.22new}) we see that
$\partial_t\phi(t)\in L^{\infty}(0,T;C^{k,\alpha}(\bT^2))$.
Also, recall that $\phi(t):\bT^2\to \bT^2$ is a diffeomorphism for each $t$,
and (\ref{flowMapDiffeo}) holds.
Hence we define
\begin{equation*}\mathbf{u}(t, x)=\partial_t\phi(t,\phi^{-1}(t,x)),\end{equation*}
and obtain
$\mathbf{u}\in L^{\infty}(0,T;C^{k,\alpha})$.
\end{proof}
In (\ref{4.15}), we replace $x$ by
$\phi^{-1}(t,x)$, thus get
\begin{equation}\label{4.16}
\mathbf{u}+(-J)f^{-1}(x)D_t\mathbf{u}_g=\mathbf{u}_g.
\end{equation}

This gives equation (\ref{1.2}), (\ref{1.3}). Also, we have remarked in (\ref{flowMapMeas}) that $\phi_t$ is measure preserving.
Then it follows that $\mathbf{u}$ is divergence free, i.e. (\ref{1.4}) holds.

\section{Extension to SG Shallow Water with variable $f$}
\label{varCoriolisPar-sect}
In this section, we will extend previous approach  to the SG shallow water equations with periodic boundary conditions. The SGSW equation is the following:
\begin{align}\label{5.1}
&(u_g,v_g)=f^{-1}(-\frac{\partial h}{\partial y},\frac{\partial h}{\partial x}),\\
\label{5.2}
&D_tu_g+\frac{\partial h}{\partial x}-fv=0,\\
\label{5.3}
&D_tv_g+\frac{\partial h}{\partial y}+fu=0,\\
\label{5.4}
&\frac{\partial h}{\partial t}+\nabla \cdot(h\mathbf{u})=0.
\end{align}
With initial data $h|_{t=0}=h_0$. Recall the assumptions made in Theorem \ref{1.9}.
\begin{equation}\label{5.5}
h_0\in C^{k+2,\alpha}(\bT^2) \textrm{ for some $k\geq2$};\,\,\,\,
I+f^{-1}D(f^{-1}Dh_0)>c_0;\,\,\,\, \int_{[0,1)^2} h_0=1;\,\,\,h_0\geq c_1.
\end{equation}

Then as in Section \ref{TimeSteppingProcSect}, we can discretize in time and
thus obtain a time-difference equation, and we need to solve:
\begin{equation}\label{5.6}
x+f^{-1}(x)f^{-1}(F^{-1}(x))\nabla q(x)=F^{-1}(x)+f^{-2}R_{f\delta t}(\nabla h)(F^{-1}(x)).
\end{equation}
In the above, $q(x)=h_{n+1}(x)$, $h(x)=h_n(x)$, and $F=F_{n+1}$. This is the same
equation as (\ref{1.30}). Now we require $\int_{[0,1]^2} q=\int_{[0,1]^2} h=1$.
In the shallow water setting, $F$ does not preserve Lebesgue measure, but should satisfy the following push-forward condition:
\begin{equation}\label{5.7}
F_\#(hdx)=qdy,
\end{equation}
which, in case of  continuous $h, q$ and $C^1$ diffeomorphism $F$, is equivalent to
\begin{equation}\label{5.8}
\det DF^{-1}=\frac{q(x)}{h(F^{-1}(x))}.
\end{equation}
Therefore, recalling (\ref{1.31}), (\ref{1.34}) (or its version (\ref{3.3}))
 we see the equation we need to solve is now
\begin{equation}\label{5.9}
\frac{\det(I+f^{-1}(x)\nabla q\otimes\nabla(f^{-1})+f^{-2}D^2q+A(x))}{\det[(I+f^{-1}\nabla h\otimes\nabla (f^{-1})+f^{-2}D^2h)(F^{-1}(x))+B(x)]}=\frac{q(x)}{h(F^{-1}(x))}.
\end{equation}
where $A(x)$ and $B(x)$ is the same (\ref{1.32}),(\ref{1.33}), with $p_{n+1}$ replaced by $q$, $p_n$ replaced by $h$.
Thus we need to solve (\ref{5.6}), (\ref{5.9}) for $\pNP$ and $F$ and show that $F$ is a
diffeomorphism.

Since equation (\ref{5.6}) differs from (\ref{3.2}) only by notations
($p$ changed to $h$), then all results of Section \ref{constMaps-Sect},
with $p$ replaced by $h$, holds for equation (\ref{5.6}). In particular,
map $\Mp$ is defined as in Lemma \ref{2.11}, and satisfies properties
in Lemmas \ref{2.11}, \ref{mapIsDiffeo}, and the map
$F=(Id+\Mp(h,\pNP, \delta t))^{-1}$ solves  equation (\ref{5.6})
for given
$(h,\pNP, \delta t)\in V_{\hat\eps_1'}(\pN_0)\times(-\hat\eps_1', \hat\eps_1')$,
where $\eps_1'$ is defined in Lemma \ref{mapIsDiffeo}.
Also, Lemma \ref{l5.1} holds.

Next we use implicit function theorem to solve (\ref{5.9}) for $\pNP$,
with $F$ defined by
$F^{-1}= Id+\Mp(\pN,\pNP, \delta t)$. Denote by $C_1^{2,\alpha}$ the
following affine
subspace of $C^{2,\alpha}(\bT^2)$:
$$
C_1^{2,\alpha}=\big\{w\in C^{2,\alpha}(\bT^2)\;:\;\int_{[0,1)^2} w=1\,\big\}.
$$
Consider the following subset $U_2$ of $C_1^{2,\alpha/2}$:
\begin{equation}\label{5.10}
U_2=\{w\in C_1^{2,\alpha/2}:\;I+f^{-1}D(f^{-1}Dw)>\frac{c_0}{2},\;\; w>\frac{c_1}{2},
\;\; \|w-h_0\|_{2,\alpha/2}<\epsTwo\}.
\end{equation}
Here $\epsTwo$ is from Lemma \ref{l5.1}.
Then $\Mp(h,q,\delta t)$ is defined, and $id+\Mp(h,q,\delta t)$ is a
$C^{1,\alpha/2}_{loc}$ diffeomorphism of $\bR^2$, as long as $\|h-h_0\|_{3,\alpha}<\epsTwo$, $\|q-h_0\|_{2,\alpha/2}<\epsTwo$, and $|\delta t|<\epsTwo$.

Let $\tilde{U_2}\subset C_1^{3,\alpha}$ be defined by
$$
\tilde{U_2}=\{w\in C_1^{3,\alpha}:\;I+f^{-1}D(f^{-1}Dw)>\frac{c_0}{2},\;\; w>\frac{c_1}{2},
\;\; \|w-h_0\|_{3,\alpha}<\epsTwo\}.
$$
 By Lemmas \ref{2.11} and \ref{mapIsDiffeo},
this guarantees that the map $F^{-1}=id+\Mp(h,q,\delta t)$ can
be defined and invertible if $q\in U_2$ and $h\in\tilde{U_2}$, and that
$F$ is a
Frechet differentiable map. Define a map
\begin{equation}\label{5.11}
\begin{split}
&\qquad\quad P:U_2\times\tilde{U_2}\times(-\epsTwo,\epsTwo)\rightarrow C_0^{0,\alpha/2}\\
&(q,h,\delta t)\longmapsto \frac{h\circ(id+\Mp(h,q,\delta t))\,\det(I+
f^{-1}\nabla q\otimes\nabla(f^{-1})+
f^{-2}D^2q+A))}{\det[(I+f^{-1}\nabla h\otimes\nabla (f^{-1})+
f^{-2}D^2h)\circ(id+\Mp(h,q,\delta t))+B]}-q,
\end{split}
\end{equation}
where the matrices $A$ and $B$ are given by (\ref{1.32}) and (\ref{1.33}) with
$(p_n,\,p_{n+1})$ replaced by $(h, q)$, and
$F^{-1}_{n+1}$ replaced by $id+\Mp(h,q,\delta t)$. It is clear from above formula
that $P$ maps into $C^{0,\alpha/2}$.

Now we need to show that the right hand
side of (\ref{5.11}) has integral zero. This is similar to
assertion (\ref{l5.2-i1}) of Lemma \ref{l5.2}. We first show the
related property in a more general setting for the application below.
Namely, fix
 $h_0\in U_1\cap C^{k+2,\alpha}$, and define
the versions of the sets $U_2$ and $\tilde U_2$ replacing
spaces $C^{m,\alpha}_1$ by spaces $C^{m,\alpha}$ (i.e. removing the condition
 $\int q_{[0,1]^2}=1$ in the definitions):
\begin{align*}
&U_2'=\{q\in C^{2,\alpha/2}:I+f^{-1}D(f^{-1}Dq)>\frac{c_0}{2}, \|q-p_0\|_{2,\alpha/2}<\epsTwo\};
\\
&\tilde{U_2}'=\{q\in C^{3,\alpha}:I+f^{-1}D(f^{-1}Dq)>\frac{c_0}{2}, \|q-p_0\|_{3,\alpha}<\epsTwo\}.
\end{align*}
Then, for $ (q,h,\delta t)\in U_2'\times\tilde{U_2}'\times(-\epsTwo,\epsTwo)$,
the expression defining $P$ in (\ref{5.11}) is well-defined by Lemmas \ref{2.11},
and \ref{l5.1}. Then we prove the following:

\begin{lem}\label{noZeroAverage}
Denote $Q(q,h,\delta t):=P(q,h,\delta t)+q$, i.e.
$$
Q(q,h,\delta t)= \frac{h\circ(id+\Mp(h,q,\delta t))\,\det(I+
f^{-1}\nabla q\otimes\nabla(f^{-1})+
f^{-2}D^2q+A))}{\det[(I+f^{-1}\nabla h\otimes\nabla (f^{-1})+
f^{-2}D^2h)\circ(id+\Mp(h,q,\delta t))+B]},
$$
where $A$ and $B$ are as in (\ref{5.11}). Then
for any $ (q,h,\delta t)\in U_2'\times\tilde{U_2}'\times(-\epsTwo,\epsTwo)$,
\begin{equation}\label{integralRatioZero}
\int_{[0,1)^2}Q(q,h,\delta t)(x)\,dx=\int_{[0,1)^2}h(x)dx.
\end{equation}
\end{lem}
\begin{proof}
Fix $ (q,h,\delta t)\in U_2'\times\tilde{U_2}'\times(-\epsTwo,\epsTwo)$.

From (\ref{1.31}), we see the map $Q$ can be written as
$$
Q(q,h,\delta t)= h\circ(id+\Mp(h,q,\delta t))\det D_x(id+\Mp(h,q,\delta t)).
$$
Also, $id+\Mp(h,q,\delta t):\bR^2\to \bR^2$ is a diffeomorphism by
Lemma \ref{mapIsDiffeo}.

Denote $G:=id+\Mp(h,q,\delta t)$, then, as we showed,
$G: \bR^2\to \bR^2$ is a diffeomorphism
and
$
Q(q,h,\delta t)= (h\circ G)\det D G$.
Now we calculate changing variables:
\begin{align*}
\int_{[0,1)^2}h(G(x))\det DG(x)\,dx=
\int_{G([0,1)^2)}h(y)dy=
\int_{[0,1)^2}h(x)dx,
\end{align*}
where the last equality follows from $\bZ^2$-periodicity of $G-id= \Mp(h,q,\delta t)$
and $h$, and from the fact that
$G: \bR^2\to \bR^2$ is a diffeomorphism, see Lemma \ref{chgVarOnTorus} for details.
\end{proof}

Now, for $(q,h,\delta t)\in U_2\times\tilde{U_2}\times(-\epsTwo,\epsTwo)$,
  the right hand
side of (\ref{5.11}) has integral zero by Lemma \ref{noZeroAverage}
and since $\int_{[0,1)^2} h=\int_{[0,1)^2} q=1$. Thus map $P$ indeed maps
$U_2\times\tilde{U_2}\times(-\epsTwo,\epsTwo)$
into $C_0^{0,\alpha/2}$.
Furthermore, similar to Lemma \ref{l5.2}(\ref{l5.2-i2}), one can show that the map $P$
in (\ref{5.11}) is continuously Frechet differentiable.

To apply the implicit function theorem, we need to show the linearized map
$D_qP(h_0,h_0,0): C_0^{2,\alpha/2}\rightarrow C_0^{0,\alpha/2}$ is invertible.
By the argument similar to the calculations before Lemma \ref{3.9}, we obtain
\begin{equation}\label{5.12}
D_qP(h_0,h_0,0)\hOne=a_{ij}\partial_{ij}\hOne+b_i\partial_i\hOne-\hOne:=L(\hOne).
\end{equation}
where the coefficients $a_{ij},b_i\in C^{0,\alpha/2}(\bT^2)$.
The injectivity follows from the strong maximum principle since the coefficient of
the zero-order term is negative.

To see that $L: C_0^{2,\alpha/2}\rightarrow C_0^{0,\alpha/2}$ is surjective,
we first we notice
$L:C^{2,\alpha/2}(\bT^2)\rightarrow C^{0,\alpha/2}(\bT^2)$ is
injective by the strong maximum principle. Then we show that $L$
surjective as a map
from $C^{2,\alpha/2}(\bT^2)$ to $C^{0,\alpha/2}(\bT^2)$. Indeed,
this follows
from the method of continuity
similar to the proof of Lemma \ref{3.9}, except that now we
use $L_0$ defined by $L_0\hOne=\Delta \hOne-\hOne$,
and show existence of solution to $L_0w=k$ by minimizing the functional
$I[v]=\int_{\mathbb{T}^2}\frac{1}{2}(|\nabla v|^2+v^2)+kv$ over the space
$H^1(\mathbb{T}^2):=\{v\in H_{loc}^1(\mathbb{R}^2):\textrm{ $v$
is  $\mathbb{Z}^2$-periodic}\}$.
Now, given $k_1\in C_0^{0,\alpha/2}$, we can find $w\in C^{2,\alpha/2}$ such
that $L(\hOne)=k_1$. It only remains to show $\hOne\in C_0^{2,\alpha/2}$,
i.e. that $\int_{\mathbb{T}^2}\hOne =0$.
We first note that $\int_{\mathbb{T}^2}L(\hOne)=\int_{\mathbb{T}^2}k_1=0$.
Thus, using the form of $L$ in (\ref{5.12}), it remains to show that
$\int_{\mathbb{T}^2} a_{ij}\partial_{ij}\hOne+b_i\partial_i\hOne=0$.
This can be seen as following.
Let $Q(q,h,\delta t)=P(q,h,\delta t)+q$ as in Lemma \ref{noZeroAverage}.
Then, from Lemma \ref{noZeroAverage} and since
$\int_{[0,1)^2} h=1$, we see that
 $Q$ maps $U_2'\times\tilde{U_2}'\times(-\epsTwo,\epsTwo)$
into $C^{0,\alpha/2}_1$. From (\ref{5.5}) and (\ref{5.10}), $h_0+\eps w\in U_2'$
if $|\eps|$ sufficiently
small.
Also, $a_{ij}\partial_{ij}\hOne+b_i\partial_i\hOne= D_qQ(h_0,h_0,0)\hOne$.
Hence
$$
\int_{\mathbb{T}^2}a_{ij}
\partial_{ij}\hOne+b_i\partial_i\hOne=
\left(\frac{d}{d\eps}\int_{\mathbb{T}^2}Q(h_0+\eps \hOne,h_0,0)\right)_{|\eps=0}
=0.
$$
This proves $\hOne\in C_0^{2,\alpha/2}$.

Above discussion shows the following proposition, which is
similar to Proposition \ref{5.4spm} in the SG case:
\begin{prop}\label{5.4spmsw}
There exist $\epsThree, \epsFour\in (0, \epsTwo]$ with
$\epsFour\le \epsThree$, such that for any $h\in C^{3,\alpha}_0(\bT^2)$ with
$\|h-h_0\|_{3,\alpha}<\epsFour$ and $|\delta t|<\epsFour$, there exists a unique
$q\in C^{2,\alpha/2}(\bT^2)$ which solves (\ref{3.3})
with $F^{-1}=Id+\Mp(\pN,\pNP, \delta t)$
 and satisfies
$\|q-h_0\|_{2,\alpha/2}<\epsThree$.

Thus, denoting $q:=\mathcal{H}(p,\delta t)$, we define a map
$$\mathcal{H}:\{\|p-p_0\|_{3,\alpha}<\epsFour\}\times(-\epsFour,\epsFour)\rightarrow U_2.
$$
\end{prop}

With this, we can set up the iteration as
in the beginning of Section \ref{constApprSol-1-Sect}.

Next we  establish estimates for the approximate solutions,
similar to Section \ref{constApprSol-1-Sect}. The proofs essentially repeat the corresponding proofs in
Section \ref{constApprSol-1-Sect}, thus below we will only include the parts of
the proofs which
differ from the SG case.
\begin{lem}\label{5.13}
Suppose $h_{n+1}=\mathcal{H}(h_n,\delta t)$ with $\|h_n-h_0\|_{3,\alpha}<\epsFour$, then
\begin{equation}\label{5.14}
\|h_{n+1}\|_{k+2,\alpha}\leq C_0(\|h_n\|_{k+2,\alpha}).
\end{equation}
\end{lem}
\begin{proof}
Similar to the proof of Lemma \ref{3.16}.
\end{proof}
\begin{lem}\label{5.15}
\begin{equation}\label{5.16}
\|h_{n+1}-h_n\|_{k+1,\alpha}\leq C_1\delta t.
\end{equation}
Here the constant $C_1=C_1(\|h_{n+1}\|_{k+2,\alpha},\|h_n\|_{k+2,\alpha})$.
\end{lem}
\begin{proof}
As in the proof of Lemma \ref{3.18}, the result follows from linearization.

We will write $q=h_{n+1}$, $h=h_n$. We start by rewritting the equation (\ref{5.9}) to be
\begin{equation}\label{5.17}
\begin{split}
&\quad\frac{\det(I+f^{-1}(x)\nabla q\otimes\nabla(f^{-1})+f^{-2}D^2q+A(x))}{q(x)}\\
&=\frac{\det[(I+f^{-1}\nabla h\otimes\nabla (f^{-1})+f^{-2}D^2h)(F^{-1}(x))+B(x)]}{h(F^{-1}(x))}.
\end{split}
\end{equation}
We subtract from both sides the quantity $\frac{\det(I+f^{-1}\nabla h\otimes\nabla(f^{-1})+f^{-2}D^2h)}{h(x)}$
The left hand side becomes
$$\frac{\det(I+f^{-1}D(f^{-1}Dq)+A(x))-\det(I+f^{-1}D(f^{-1}Dh))}{q(x)}-\frac{q-h}{qh}\cdot\det(I+f^{-1}D(f^{-1}Dh)).
$$

We can deal with the first term in the same way as we did in the proof of Lemma \ref{3.18}. That is, we can write the left hand side as $a_{ij}\partial_{ij}(q-h)+b_i\partial_i(q-h)-c(q-h)$, where $a_{ij},b_i,c\in C^{k,\alpha}$ and $c>0$. (Indeed $c=\frac{1}{qh}\det(I+f^{-1}D(f^{-1}Dh))$)

The right hand side becomes
\begin{equation*}
\begin{split}
&\frac{\det[(I+f^{-1}D(f^{-1}Dh)(F^{-1}(x))+B(x)]-\det(I+f^{-1}D(f^{-1}Dh))}{h(F^{-1}(x))}+\\
&\qquad\det(I+f^{-1}D(f^{-1}Dh))\cdot(\frac{1}{h(F^{-1}(x))}-\frac{1}{h(x)}).
\end{split}
\end{equation*}

In the same way as in the proof of Lemma \ref{3.18}, we can write the first term above as $\tilde{b}_i\partial_i(q-h)+\delta tf$, with $\tilde{b}_i\in C^{k-1,\alpha},f\in C^{k,\alpha}$. The second term can be reprensented in terms of $F^{-1}(x)-x$, hence in terms of $\nabla q-\nabla h$, as was done in the proof of Lemma \ref{3.11}. To summarize, we will get an elliptic equation of the form
\begin{equation*}
a_{ij}\partial_{ij}(q-h)+c_i\partial_i(q-h)-c(q-h)=\delta t f.
\end{equation*}
with $a_{ij}, c\in C^{k,\alpha}$, $c_i\in C^{k-1,\alpha}$,
$f\in C^{k,\alpha}$, and $c>0$. Noting the sign of the zero-order term,
we obtain the desired estimate.
\end{proof}
\begin{lem}\label{5.20}
\begin{equation}\label{5.21}
\|F^{-1}-id\|_{k,\alpha}\leq C_2\delta t.
\end{equation}
Here $C_2=C_2(\|h_{n+1}\|_{k+2,\alpha},\|h_n\|_{k+2,\alpha})$.
\end{lem}
\begin{proof}
Exactly the same proof as in Lemma \ref{3.29}.
\end{proof}
Denote
$$
\nu_n=\frac{\det(I+f^{-1}D(f^{-1}h_n))}{h_n}.
$$
\begin{lem}\label{5.22}
\begin{equation}\label{5.23}
\|\nu_{n+1}\|_{k,\alpha}\leq \|\nu_n\|_{k,\alpha}+C_3\delta t.
\end{equation}
Here $C_3=C_3(\|h_{n+1}\|_{k+2,\alpha},\|h_n\|_{k+2,\alpha})$. Recall $\nu_n=\frac{\det(I+f^{-1}D(f^{-1}Dh_n))}{h_n}$.

\end{lem}
\begin{proof}
We follow the proof of Lemma \ref{3.32}.

We use notations $q(x)=h_{n+1}(x)$, $h(x)=h_n(x)$. Define
\begin{equation*}
\begin{split}
&\nu_{n+1}^{\prime}=\frac{\det(I+f^{-1}D(f^{-1}Dq)+A(x))}{q(x)},\\
&\nu_n^{\prime}=\frac{\det[(I+f^{-1}D(f^{-1}Dh))(F^{-1}(x))]}{h(F^{-1}(x))},\\
&\nu_n^{\prime\prime}=\frac{\det[(I+f^{-1}D(f^{-1}Dh))(F^{-1}(x))+B(x)]}{h(F^{-1}(x))}.
\end{split}
\end{equation*}
By the equation, $\nu_{n+1}^{\prime}=\nu_n^{\prime\prime}$. We prove the result by showing
\begin{align}\label{5.24}
&\|\nu_{n+1}^{\prime}-\nu_{n+1}\|_{k,\alpha}\leq C\delta t,\\
\label{5.25}
&\|\nu_n^{\prime}\|_{k,\alpha}\leq\|\nu_n\|_{k,\alpha}+C\delta t,\\
\label{5.26}
&\|\nu_n^{\prime}-\nu_n^{\prime\prime}\|_{k,\alpha}\leq C\delta t.
\end{align}
(\ref{5.24}), (\ref{5.26}) are proved in exactly the same way as in Lemma \ref{3.32}. (\ref{5.25}) is the consequence of Lemma \ref{3.41} applied to $G=\frac{\det(I+f^{-1}D(f^{-1}Dp))}{p}$.\end{proof}
\begin{lem}\label{5.27}
\begin{equation}\label{5.28}
\|h_{n+1}\|_{k+2,\alpha}\leq C_4(\|\nu_{n+1}\|_{k,\alpha}+1).
\end{equation}
Here $C_4=C_4(\|p_{n+1}\|_{2,\alpha})$.
\end{lem}
So we can obtain exactly same estimates as in SG case.
Then we repeat the proof of Proposition \ref{approxSolut-est-Prop},
to get the following

\begin{prop}\label{approxSolut-est-Prop-sgsw}
For any $h_0\in C^{k+2,\alpha}$, $k\geq2$, satisfying $I+f^{-1}D(f^{-1}Dh_0)>c_0$,
there exists $T_1>0$, $\eps>0$, depending only on $\|h_0\|_{k+2,\alpha}$, $f$, $c_0$ and $k$,
such that for any $|\delta t|<\eps$, one can construct a sequence
$(h_n, F_n)$ of solutions to (\ref{5.6}), (\ref{5.8})
(with $h=h_n$, $q=h_{n+1}$ and $F=F_{n+1}$) for $n\leq \frac{T_1}{\delta t}-1$,
such that $h_n\in C_1^{k+2,\alpha}$ and $F_n:\bT^2\to \bT^2$ is a diffeomorphism satisfying $\det DF_{n+1}=\frac{h_n}{h_{n+1}(F_{n+1})}$.
Moreover, $(h_n, F_n)$ satisfy the following estimates:
\begin{align}\label{3.55-sgsw}
&\|h_n\|_{k+2,\alpha}\leq C,\\
\label{3.56new-sgsw}
&\|h_{n+1}-p_n\|_{k+1,\alpha}\leq C\delta t,\\
\label{3.57-sgsw}
&\|F_{n+1}^{-1}-id\|_{k,\alpha}\leq C\delta t,\\
\label{3.58-sgsw}
&\|F_{n+1}-id\|_{k,\alpha}\leq C\delta t.
\end{align}
 for $0\leq n\leq \frac{T_1}{\delta t}-1$.
Here $C$ depends only on $\|h_0\|_{k+2,\alpha}$, $f$, $c_0$ and $k$.
\end{prop}

Define the time-discretized map $\phi(n\delta t,\cdot)=F_n\circ\cdots \circ F_1$ as
in (\ref{3.59}). Then similar to SG case, we can show that $\phi(n\delta t,x+h)=\phi(n\delta t,x)+h$ for $h\in \bZ^2$ and the following estimate holds:
\begin{equation}\label{5.34}
\|\phi(n\delta t,\cdot)-\phi((n-1)\delta t,\cdot)\|_{k,\alpha}\leq C\delta t.
\end{equation}
for $1\leq n\leq\frac{T}{\delta t}$. Here $T$ depends only on $\|h_0\|_{k+2,\alpha}$, $f$, $c_0,c_1$ and $k$.

 Also $\phi(n\delta t,\cdot)_\#(h_0dx)=h_ndx$. All the convergence facts obtained
in Section \ref{passageLim-Sect-sub1}, i.e. (\ref{4.2}) and (\ref{4.3})
for $h$ instead of $p$,
hold true here. Similar to (\ref{4.4}), we define
\begin{equation*}
\mathbf{v}_{g,\delta t}( t,x)=
f^{-1}(\phi_{\delta t}(t,x))J\nabla h_{\delta t}(t,\phi_{\delta t}( t,x)),
\quad
\mathbf{v}_{g}( t,x)=
f^{-1}(\phi(t,x))J\nabla h(t,\phi( t,x)).
\end{equation*}
Then we obtain (\ref{4.5}) and (\ref{4.5-0}). With this,
 following the argument in Sections \ref{passageLim-Sect-sub2}, \ref{passageLim-Sect-sub3},
we obtain  as Proposition \ref{p7.1}:
\begin{equation}\label{5.35}
\partial_t\phi(t,x)+(-J)f^{-1}(\phi_t(x))\partial_t\mathbf{v}_g(t,x)=
\mathbf{v}_g(t,x),\textrm{ for any $(t,x)\in[0,T)\times\mathbb{T}^2$.}
\end{equation}
Back to Eulerian coordinates, this equation is equivalent to
\begin{equation}\label{5.36}
\mathbf{u}_t(x)+(-J)f^{-1}(x)D_t\mathbf{u}_g(t,x)=\mathbf{u}_g(t,x).
\end{equation}
This gives (\ref{5.2}),(\ref{5.3}). It only remains to show (\ref{5.4}). To show this, we only need to show
\begin{equation}\label{5.37}
\phi_t\# h_0=h_t.
\end{equation}
Since $\phi_t$ is a diffeomorphism, this is equivalent to
\begin{equation}\label{8.32new}
\det D\phi_t=\frac{h_0}{h_t(\phi_t)}.
\end{equation}
To see (\ref{8.32new}), notice $\phi_t$ is the limit of $\phi^{\delta t}_t$ under the norm $C^{k,\beta}$ for every $\beta<\alpha$, and $h(t)$ is the limit of $h^{\delta t}(t)$ under the norm $C^{k+1,\beta}$, with any $\beta<\alpha$. Under discrete setting, it holds: $\det D\phi(n\delta t,\cdot)=\frac{h_0}{h_n(\phi(n\delta t,\cdot))}$. Now one can use the Lipshitz continuity in time of $\phi_n$ and $h_n$ and above convergence to pass to limit and obtain (\ref{8.32new}).

\section{Uniqueness}
\label{uniq-Sect}
In this section, we prove uniqueness of smooth solutions to SG and SG shallow water. In the following, $C$ will denote a generic constant which may change from line to line, but its dependence will be made clear.
\subsection{Uniqueness for SG}
\label{uniqSG-Sect}
In this subsection, we prove the uniqueness part of Theorem \ref{1.5}.
Suppose we have two solutions of SG, $(p_1,\mathbf{u}_{1,g},\mathbf{u}_1)$ and $(p_2,\mathbf{u}_{2,g},\mathbf{u}_2)$ defined on $[0,T_0]\times\bT^2$,
with regularity $p_i\in L^{\infty}(0,T_0;C^{3}(\bT^2))$ and $\partial_tp_i\in L^{\infty}(0,T_0;C^2(\bT^2))$.
Then by (\ref{4.16}), we see $\mathbf{u}_i\in L^{\infty}(0,T_0;C^1(\bT^2))$.
This regularity allows us to define the flow map $\phi_i(t,x)$$(i=1,2)$ which
satisfies $\phi_{i}(t,x+h)=\phi_i(t,x)+h$ for any $h\in \bZ^2$.
 Also, from (\ref{1.7}) for $p_1$ and $p_2$, there exists a constant $\lambda >0$
 such that
\begin{equation}\label{assumptOnSOlnUniqSG}
 I+f^{-1}D(f^{-1}Dp_i)\ge \lambda I   \;\textrm{ for }\;i=1,2
 \;\textrm{ on }\;[0,T_0]\times\bT^2.
\end{equation}

In the following,
the argument is made for either one of the solutions if there is no subscript.
Also we will denote $\phi_t(x)=\phi(t,x)$, $\phi_{1,t}(x)=\phi_1(t,x)$ and similar notations for other functions.

Recall $\mathbf{v}_g=f^{-1}(\phi_t(x))J\nabla p_t(\phi_t(x))$, i.e,
$\mathbf{v}_g$ is
$\mathbf{u}_g$ expressed in Lagrangian variables. Then we obtain from (\ref{4.15})
(which holds for all $(x,t)$ by (\ref{4.16})) that
\begin{equation}\label{6.1}
\partial_t\phi_t(x)+f^{-1}(\phi_t(x)))\partial_t(f^{-1}(\phi_t(x))\nabla p_t(\phi_t(x))=f^{-1}J\nabla p(\phi_t(x)).
\end{equation}
We can solve
\begin{equation}\label{6.2}
\partial_t\phi_t(x)=(I+f^{-1}D(f^{-1}Dp))^{-1}(\phi_t(x))(f^{-1}J\nabla p_t-f^{-2}\partial_t\nabla p_t)(\phi_t(x)).
\end{equation}
On the other hand, we can rewrite (\ref{6.1}) as
\begin{equation}\label{6.3}
\partial_t\phi_t(x)+\partial_t(f^{-2}\nabla p_t(\phi_t(x)))=f^{-1}\nabla p_t\otimes\nabla(f^{-1})(\phi_t(x))\partial_t\phi_t(x)+f^{-1}J\nabla p_t(\phi_t(x)).
\end{equation}
Integrate in $t$ to get
\begin{equation*}
\begin{split}
\phi_t(x)
&+f^{-2}\nabla p_t(\phi_t(x))=x+f^{-2}\nabla p_0(x)\\
&+\int_0^tf^{-1}\nabla p_s\otimes\nabla(f^{-1})(\phi_s(x))\partial_s\phi_s(x)ds+
\int_0^tf^{-1}J\nabla p_s(\phi_s(x))ds.
\end{split}
\end{equation*}
Differentiating on both sides, we obtain
\begin{equation}\label{6.5}
\begin{split}
&(I+D(f^{-2}Dp_t)(\phi_t))D\phi_t(x)=(I+D(f^{-2}Dp_0))(x)\\
&\qquad +\int_0^tf^{-1}\nabla p_s\otimes\nabla(f^{-1})(\phi_s(x))\partial_sD\phi_s(x)ds
+\int_0^tD(f^{-1}J\nabla p_s)(\phi_s(x))D\phi_s(x)ds\\
&\qquad +\int_0^t\sum_{k,m}\partial_m(f^{-1}\partial_ip_s\partial_k(f^{-1}))(\phi_s(x))
\partial_{x_j}\phi_s^m\partial_s\phi_s^kds.
\end{split}
\end{equation}
We do an integration by parts in the first integral of right hand side:
\begin{equation*}
\begin{split}
&
\int_0^tf^{-1}\nabla p_s\otimes\nabla(f^{-1})(\phi_s(x))\partial_sD\phi_s(x)ds\\
&\quad=f^{-1}\nabla p_t\otimes\nabla(f^{-1})(\phi_t(x))D\phi_t(x)
-f^{-1}\nabla p_0\otimes\nabla(f^{-1})(x)\\
&\qquad-\int_0^t\partial_s[f^{-1}\nabla p_s\otimes\nabla(f^{-1})(\phi_s(x))]D\phi_s(x)ds.
\end{split}
\end{equation*}
Plugging this into (\ref{6.5}), we obtain
\begin{equation}\label{6.7}
\begin{split}
&(I+f^{-1}D(f^{-1}Dp_t))(\phi_t(x))D\phi_t(x)=(I+f^{-1}D(f^{-1}Dp_0))(x)\\
&\qquad-\int_0^tf^{-1}\partial_s\nabla p_s\otimes\nabla(f^{-1})(\phi_s(x))D\phi_s(x)ds\\
&\qquad+\int_0^tD(f^{-1}J\nabla p_s)(\phi_s(x))D\phi_s(x)ds\\
&\qquad+\int_0^t\sum_{k,m}\partial_m(f^{-1}\partial_ip_s\partial_k(f^{-1}))(\phi_s(x))
(\partial_{x_j}\phi_s^m\partial_s\phi_s^k-\partial_{x_j}\phi_s^k\partial_s\phi_s^m)ds.
\end{split}
\end{equation}
Define
\begin{align}\label{6.8}
&E^1(s,x)=-f^{-1}\partial_s\nabla p_s\otimes\nabla(f^{-1})(x),\\
\label{6.9}
&E^2(s,x)=D(f^{-1}J\nabla p_s)(x),\\
\label{6.10}
\begin{split}
&E^3_{ik}(s,x)=\sum_{m}\partial_s\phi_s^m(\phi_s^{-1}(x))[-\partial_m(f^{-1}\partial_ip_s\partial_k(f^{-1}))(x)\\
&\qquad\qquad\quad+\partial_k(f^{-1}\partial_ip_s\partial_m(f^{-1})(x))]
\\
&\qquad\qquad=\sum_{m}\mathbf{u}^m(x)[-\partial_m(f^{-1}\partial_ip_s\partial_k(f^{-1}))(x)+\partial_k(f^{-1}\partial_ip_s\partial_m(f^{-1})(x))].
\end{split}
\end{align}
and let $E(s,x)=E^1+E^2+E^3$. Notice that we can use (\ref{6.3}) to plug in (\ref{6.10}) for $\partial_s\phi_s^m$. So $E(s,x)$ involves space derivatives of $p$ up to second order and also $\partial_t\nabla p_t$. We can write (\ref{6.7}) as
\begin{equation}\label{6.11}
(I+f^{-1}D(f^{-1}Dp_t))(\phi_t(x))D\phi_t(x)=(I+f^{-1}D(f^{-1}Dp_0))(x)+\int_0^tE(s,\phi_s(x))D\phi_s(x)ds.
\end{equation}
Multiplying (\ref{6.11}) by $E(t,\phi_t(x))(I+f^{-1}D(f^{-1}Dp))^{-1}(\phi_t(x))$
from the left on both sides, we obtain
\begin{equation}\label{6.12}
\begin{split}
E(t,\phi_t(x))D\phi_t&(x)=E(t,\phi_t(x))\big(I+f^{-1}D(f^{-1}Dp))^{-1}(\phi_t(x)\big)\cdot\\
&[(I+f^{-1}D(f^{-1}Dp_0))(x)+\int_0^tE(s,\phi_s(x))D\phi_s(x)ds].
\end{split}
\end{equation}
 Denote
\begin{equation}\label{6.13}
H(t,x)=\int_0^tE(s,\phi_s(x))D\phi_s(x)ds.
\end{equation}
We see it solves the following initial value problem.
\begin{align}\label{6.14}
&\frac{dH(t,x)}{dt}=G(t,\phi_t(x))(D(x)+H(t,x)),\\
\label{6.15}
&H(0)=0,
\end{align}
where
\begin{equation}\label{6.16}
G(t,x)=E(t,x)(I+f^{-1}D(f^{-1}Dp))^{-1}(t,x),
\end{equation}
and
\begin{equation}\label{6.17}
D(x)=I+f^{-1}D(f^{-1}Dp_0)(x).
\end{equation}
Finally we conclude from the measure preserving
property of $\phi_t$ and (\ref{6.11}) that
\begin{equation}\label{6.18}
\det(I+f^{-1}D(f^{-1}Dp_t))=\det[I+f^{-1}D(f^{-1}Dp_0)(\phi_t^{-1}(x))+H(t,\phi_t^{-1}(x))].
\end{equation}

Let $\phi_1$, $\phi_2$ be the flow maps corresponding to the two solutions
introduced at the beginning of this section, with initial data $p_0$.
We start with some preliminary estimates:

\begin{lem}\label{6.19}
Let $H(t,x):[0,T]\times\bR^2\rightarrow M_{2\times2}(\bR)$ be the solution to the problem:
\begin{equation*}
\begin{split}
\frac{dH(t,x)}{dt}&=A(t,x)\cdot H(t,x)+B(t,x),\\
&H(0)=0.
\end{split}
\end{equation*}
Here $A(t,x),B(t,x), H(t,x)$ are $2\times 2$ matrix valued functions, periodic with respect to $\bZ^2$, and we assume $\sup_{0\leq t\leq T}\|A(t,\cdot)\|_{C^1}+\|B(t,\cdot)\|_{C^1}\leq M$. Then we have
\begin{equation}\label{6.20}
\|H(t,\cdot)\|_{L^{\infty}}+\|D_xH(t,\cdot)\|_{L^{\infty}}\leq Ct.
\end{equation}
Here $C$ depends only on $M$ and $T$.
\end{lem}
\begin{proof}
To estimate of $\|H(t,\cdot)\|_{L^{\infty}}$ follows from the Gronwall inequality.
To estimate $D_xH$, we first differentiate and then use Gronwall
inequality and the already obtained estimate on $H$.
\end{proof}
We use Lemma \ref{6.19} to estimate $C^1$-norm of the
function $H(t,x)$ defined in (\ref{6.13}):
\begin{lem}\label{6.21}
Let $H$ be as defined in (\ref{6.13}), then:
\begin{align}
\label{6.22-0}
&\|\phi(t,\cdot)\|_{C^1(\bT^2)}\leq C,\\
\label{6.22}
&\|H(t,\cdot)\|_{C^1(\bT^2)}\leq Ct.
\end{align}
Here $C$ depends on $\|p\|_{L^{\infty}(0,T_0;C^3)},\|\partial_tp\|_{L^{\infty}(0,T_0;C^2)},\lambda,T_0$.
\end{lem}
\begin{proof}
To show (\ref{6.22-0}), we observe that
due to (\ref{4.16}), $\mathbf{u}_t(\cdot)$ is bounded in $C^1(\bT^2)$ for each $t$,
with a bound
which has the same dependence as in the lemma.
Thus the flow map $\phi_t$ of $\mathbf{u}_t(\cdot)$
has estimate in $C^1(\bT^2)$ for each $t$ by a constant with same dependence.

To prove (\ref{6.22}), according to (\ref{6.14}),(\ref{6.15}) and Lemma \ref{6.19},
 we
only need to show
\begin{equation}\label{6.23}
\sup_{0\leq t\leq T_0}\|G(t,\phi_t(\cdot)\|_{C^1(\bT^2)}\leq C,
\end{equation}
since estimate of $\|D\|_{C^1(\bT^2)}$ follows from
(\ref{6.17}).

To prove (\ref{6.23}), we only need to show
$\sup_{0\leq t\leq T_0}\|G(t,\cdot)\|_{C^1}+\|\phi_t\|_{C^1}\leq C$.
The term $\|\phi_t\|_{C^1}$ is estimated in (\ref{6.22-0}).
For the estimate of $G(t,\cdot)$, we need to go back to (\ref{6.16}),
and (\ref{6.8})-(\ref{6.10}). From the explicit expressions given there
and (\ref{assumptOnSOlnUniqSG}),
we see that the $C^1$ norm of $G(t,\cdot)$ is bounded by a constant which has the
dependence as in the lemma.
\end{proof}

Below, we write $L^2$ and $H^2$ for $L^2(\bT^2)$ and $H^2(\bT^2)$.
\begin{lem}\label{6.24}
For any $t\in[0,T_0]$, we have
\begin{equation}\label{6.25}
\|\phi_{1,t}^{-1}-\phi_{2,t}^{-1}\|_{L^2}\leq C\|\phi_{1,t}-\phi_{2,t}\|_{L^2}.
\end{equation}
and
\begin{equation}\label{6.26}
\|\phi_{1,t}-\phi_{2,t}\|_{L^2}\leq C\|\phi_{1,t}^{-1}-\phi_{2,t}^{-1}\|_{L^2}.
\end{equation}
Here $C$ depends on $\|p_i\|_{L^{\infty}(0,T_0;C^{3}(\bT^2))},
\|\partial_tp_i\|_{L^{\infty}(0,T_0;C^2(\bT^2))}, \lambda,T_0$, $i=1,2$.
\end{lem}
\begin{proof}
From  estimate (\ref{6.22-0}) for $\phi_1$ and $\phi_2$, we can bound the
eigenvalues of
$D\phi_{i,t}$, $i=1,2$ from above. Since $\det D\phi_{i,t}=1$, we also
obtain the bound for the eigenvalues from below. This gives estimate for
$\|D\phi_t^{-1}\|_{L^{\infty}}$. Now, since $\phi_{i,t}$ is measure preserving, we have
\begin{equation}\label{6.27}
\|\phi_{1,t}^{-1}-\phi_{2,t}^{-1}\|_{L^2}=\|\phi_{1,t}^{-1}\circ\phi_{2,t}-id\|_{L^2}\leq \|D\phi_{1,t}^{-1}\|_{L^{\infty}}\|\phi_{1,t}-\phi_{2,t}\|_{L^2}.
\end{equation}
This gives (\ref{6.25}). The proof for (\ref{6.26}) is similar.
\end{proof}

\begin{lem}\label{6.28}
For each $t\in[0,T_0]$, we have
\begin{equation}\label{6.29}
\|\partial_t\nabla p_{1,t}-\partial_t\nabla p_{2,t}\|_{L^2}\leq C\|p_{1,t}-p_{2,t}\|_{H^2}.
\end{equation}
Here $C$ depends on $\lambda,\|p_{i,t}\|_{C^2},\|\partial_tp_{i,t}\|_{C^1}$, $i=1,2$.
\end{lem}
\begin{proof}
We take (\ref{2.20new}) for $p_1,p_2$ and subtract to get
\begin{equation}\label{6.30}
\begin{split}
&\nabla\cdot[(I+f^{-1}D(f^{-1}Dp_{1,t}))^{-1}f^{-2}(\partial_t\nabla p_{1,t}-\partial_t\nabla p_{2,t})]
\\
&\quad =\nabla\cdot[(I+f^{-1}D(f^{-1}p_{1,t}))^{-1}(f^{-1}J\nabla p_{1,t}-f^{-1}J\nabla p_{2,t})]
\\
&\qquad+\nabla\cdot[((I+f^{-1}D(f^{-1}Dp_{1,t}))^{-1}\\
&\qquad-(I+f^{-1}D(f^{-1}Dp_{2,t}))^{-1})(f^{-1}J\nabla p_{2,t}-f^{-2}\partial_t\nabla p_{2,t})].
\end{split}
\end{equation}
Multiplying both sides of (\ref{6.30}) $\partial_tp_{1,t}-\partial_tp_{2,t}$
and integrating by parts, and using (\ref{assumptOnSOlnUniqSG}), we obtain (\ref{6.29})
\end{proof}

\begin{prop}\label{6.32}
For any $t\in[0,T_0]$
\begin{equation*}
\sup_{0\leq s\leq t}\|p_{1,s}-p_{2,s}\|_{H^2}\leq C\sup_{0\leq s\leq t}\|\phi_1-\phi_2\|_{L^2}.
\end{equation*}
Here $C$ depends on $\|p_i\|_{L^{\infty}(0,T_0;C^3(\bT^2))},
\|\partial_tp_i\|_{L^{\infty}(0,T_0;C^2(\bT^2))},\lambda, T_0$, $i=1,2$.
\end{prop}

\begin{proof}
From (\ref{6.18}) for $p_1$ and $p_2$, by subtraction we get
\begin{equation}\label{6.34}
\begin{split}
a_{ij}f^{-1}\partial_j(f^{-1}\partial_i(p_{1,t}&-p_{2,t}))=
F(x),
\end{split}
\end{equation}
where
\begin{align*}
&a_{ij}=\int_0^1M_{ji}(I+\theta f^{-1}D(f^{-1}Dp_{1,t})+(1-
\theta)f^{-1}D(f^{-1}Dp_{2,t}))d\theta,\\
&F(x)=b_{ij}[f^{-1}\partial_j(f^{-1}\partial_ip_0)(\phi_{1,t}^{-1})-
f^{-1}\partial_i(f^{-1}\partial_jp_0)(\phi_{2,t}^{-1}(x))\\
&\qquad\qquad+H_{1,ij}(t,\phi_{1,t}^{-1}(x))-H_{2,ij}(t,\phi_{2,t}^{-1}(x))],
\end{align*}
with
\begin{equation*}
\begin{split}
b_{ij}=\int_0^1M_{ji}[I&+\theta f^{-1}D(f^{-1}Dp_0)(\phi_{1,t}^{-1}(x))+(1-\theta)f^{-1}D(f^{-1}Dp_0)(\phi_{2,t}^{-1}(x))\\
&+\theta H_1(t,\phi_{1,t}^{-1}(x))+(1-\theta)H_2(t,\phi_{2,t}^{-1}(x))]d\theta.
\end{split}
\end{equation*}
Here $[M_{ij}(A)]$ denotes the cofactor matrix
of the matrix $A$.

We assumed $p_1,p_2$ to be $C^{3}$ smooth, and $f\in C^2$,
from which we conclude that the coefficients on equation (\ref{6.34}),
written in the non-divergence form, are in $C^1(\bT^2)$, and the $C^1$ norms
depend on $\|p_1,p_2\|_{L^{\infty}(0,T_0;C^{3}(\bT^2))}$.
Also,
equation (\ref{6.34}) is uniformly elliptic (by (\ref{assumptOnSOlnUniqSG})),
and has no zero-th order term. Then we apply the $H^2$ estimate to get
\begin{equation}\label{6.37}
\|p_{1,t}-p_{2,t}\|_{H^2}\leq C\|F\|_{L^2}.
\end{equation}
Here and

Next notice that $b_{ij}$ is bounded. Indeed, we just need to observe $H_1,H_2$
are bounded, this follows from Lemma \ref{6.21}.
Then we estimate
\begin{equation*}
\begin{split}
\|F\|_{L^2}&\leq C\|f^{-1}D(f^{-1}Dp_0)(\phi_{1,t}^{-1})-f^{-1}D(f^{-1}Dp_0)(\phi_{2,t}^{-1})\|_{L^2}\\
&\quad+C\|H_{1,ij}(t,\phi_{1,t}^{-1}(\cdot))
-H_{2,ij}(t,\phi_{2,t}^{-1}(\cdot))\|_{L^2}
\\
&\leq C\|\phi_{1,t}-\phi_{2,t}\|_{L^2}+C\|D_xH_1(t,\cdot)\|_{L^{\infty}}\|\phi_{1,t}-\phi_{2,t}\|_{L^2}+C\|H_1-H_2\|_{L^2}.
\end{split}
\end{equation*}

Here we used that $p_0\in C^3$ and Lemma \ref{6.24}. Now we use that
$\|DH_1\|_{L^{\infty}}$ is bounded, which follows from Lemma \ref{6.21}, to obtain
\begin{equation}\label{6.39}
\|F\|_{L^2}\leq C\|\phi_{1,t}-\phi_{2,t}\|_{L^2}+C\|H_1-H_2\|_{L^2}.
\end{equation}
We can deduce from (\ref{6.14}) that
\begin{equation*}
\frac{d}{dt}(H_1-H_2)=G_1(t,\phi_{1,t}(x))(H_1-H_2)+(G_1(t,\phi_{1,t})-G_2(t,\phi_{2,t}))(H_2+D).
\end{equation*}
Thus we have
\begin{equation}\label{6.41}
\begin{split}
&\|H_{1,t}-H_{2,t}\|_{L^2}\leq\sup_{0\leq t\leq T_0}\|G_1(t,\cdot)\|_{L^{\infty}}\int_0^t\|H_{1,s}-H_{2,s}\|_{L^2}\\
&+(\sup_{0\leq t\leq T_0}\|H\|_{L^{\infty}}+\|D\|_{L^{\infty}})\int_0^t\|G_1(s,\phi_{1,s}(\cdot))-G_2(s,\phi_{2,s}(\cdot))\|_{L^2}.
\end{split}
\end{equation}
Again to deal with the second term, we have now
\begin{equation}\label{6.42}
\begin{split}
&\|G_1(s,\phi_{1,s}(\cdot))-G_2(s,\phi_{2,s}(\cdot))\|_{L^2}\\
&\qquad\leq
\|G_1(s,\phi_{1,s}(\cdot))-G_1(s,\phi_{2,s}(\cdot))\|_{L^2}+\|G_1(s,\phi_{2,s}(\cdot))-G_2(s,\phi_{2,s}(\cdot))\|_{L^2}
\\
&\qquad\leq \|DG_1\|_{L^{\infty}}\|\phi_{1,s}-\phi_{2,s}\|_{L^2}+\|G_1-G_2\|_{L^2}.
\end{split}
\end{equation}
Recall (\ref{6.16}) for definition of $G_1$, and the explicit expressions given in (\ref{6.8})-(\ref{6.10}), we see $\|D_xG_1\|_{L^{\infty}}\leq C$. We still need to estimate $\|G_1-G_2\|_{L^2}$. Now
\begin{equation*}
G_1-G_2=(E_1-E_2)(I+f^{-1}D(f^{-1}Dp_1))^{-1}+E_2[(I+f^{-1}D(f^{-1}Dp_1))^{-1}-(I+f^{-1}D(f^{-1}Dp_2))^{-1}].
\end{equation*}
Then we estimate
\begin{equation}\label{6.44}
\begin{split}
\|G_1-G_2\|_{L^2}&\leq\|(I+f^{-1}D(f^{-1}Dp))^{-1}\|_{L^{\infty}}\|E_1-E_2\|_{L^2}\\
&\quad+\|E_2\|_{L^{\infty}}\cdot
\|(I+f^{-1}D(f^{-1}Dp_1))^{-1}-(I+f^{-1}D(f^{-1}Dp_2))^{-1}\|_{L^2}\\
&\leq C\|E_1-E_2\|_{L^2}+C\|p_1-p_2\|_{H^2}.
\end{split}
\end{equation}
It only remains to estimate $\|E_1-E_2\|_{L^2}$. For this we go back to (\ref{6.8})-(\ref{6.10}). One sees easily from (\ref{6.8}) and previous lemma
\begin{equation}\label{6.45}
\|E^1_{1,s}-E^1_{2,s}\|_{L^2}\leq C\|\partial_s\nabla p_1-\partial_s\nabla p_2\|_{L^2}\leq C\|p_1-p_2\|_{H^2}.
\end{equation}
Also from (\ref{6.9}):
\begin{equation}\label{6.46}
\|E_{1,s}^2-E_{2,s}^2\|_{L^2}\leq C\|p_{1,s}-p_{2,s}\|_{H^2}
\end{equation}
One can see from (\ref{6.10}) that
\begin{equation}\label{6.47}
\|E^3_{1,s}-E^3_{2,s}\|_{L^2}\leq C\|\mathbf{u}_{1,s}^m-\mathbf{u}_{2,s}^m\|_{L^2}\|p_1\|_{C^{2}}+C\|\mathbf{u}_{2,s}^m\|_{L^{\infty}}\|p_1-p_2\|_{H^2}
\end{equation}
From (\ref{4.16}) we observe that $\|\mathbf{u}_t\|_{L^{\infty}}$ can be estimated by $\|p\|_{C^{2}}$ and $\|\partial_t p\|_{C^1}$.
On the other hand, from Lemma \ref{6.28}:
\begin{equation}\label{6.48}
\|\mathbf{u}_{1,t}-\mathbf{u}_{2,t}\|_{L^2}\leq C(\|p_1-p_2\|_{H^2}+\|\partial_t\nabla p_1-\partial_t\nabla p_2\|_{L^2})\leq C\|p_{1,t}-p_{2,t}\|_{H^2}.
\end{equation}
Collecting (\ref{6.44})-(\ref{6.48}) to get
\begin{equation*}
\|G_1-G_2\|_{L^2}\leq C\|p_1-p_2\|_{H^2}.
\end{equation*}
Recalling (\ref{6.41}) and (\ref{6.42}), we get
\begin{equation*}
\|H_{1,t}-H_{2,t}\|_{L^2}\leq C\int_0^t\|H_{1,s}-H_{2,s}\|_{L^2}ds+C\int_0^t\|\phi_{1,s}-\phi_{2,s}\|_{L^2}ds+C\int_0^t\|p_{1,s}-p_{2,s}\|_{H^2}ds.
\end{equation*}
Use Gronwall to get
\begin{equation*}
\|H_{1,t}-H_{2,t}\|_{L^2}\leq C\int_0^t\|\phi_{1,s}-\phi_{2,s}\|_{L^2}ds+C\int_0^t\|p_{1,s}-p_{2,s}\|_{H^2}ds.
\end{equation*}
Recalling (\ref{6.37}) and (\ref{6.39}), we have
\begin{equation*}
\|p_{1,t}-p_{2,t}\|_{H^2}\leq C\|\phi_{1,t}-\phi_{2,t}\|_{L^2}+\int_0^t\|\phi_{1,s}-\phi_{2,s}\|_{L^2}ds+\int_0^t\|p_{1,s}-p_{2,s}\|_{H^2}ds.
\end{equation*}
Now, using Gronwall inequality, we conclude the proof.
\end{proof}
We present a final lemma which follows directly from (\ref{6.2}).
\begin{lem}\label{6.53}
\begin{equation*}
\|\phi_{1,t}-\phi_{2,t}\|_{L^2}\leq C\int_0^t(\|p_{1,s}-p_{2,s}\|_{H^2}+\|\partial_t\nabla p_{1,s}-\partial_t\nabla p_{2,s}\|_{L^2})ds.
\end{equation*}
Here $C$ depends on $\|p_i\|_{L^{\infty}(0,T_0;C^3(\bT^2))},
\|\partial_tp_i\|_{L^{\infty}(0,T_0;C^2(\bT^2))},\lambda,T_0$, $i=1,2$.
\end{lem}
\begin{proof}
From (\ref{6.2}), using (\ref{assumptOnSOlnUniqSG}) and the measure-preserving property
of $\phi_{2,t}$, we have
\begin{equation*}
\begin{split}
\|\partial_t\phi_{1,t}&-\partial_t\phi_{2,t}\|_{L^2}\\
&\leq \|(I+f^{-1}D(f^{-1}Dp_{1,t}))^{-1}(f^{-1}J\nabla p_{1,t}-f^{-2}\partial_t\nabla p_{1,t})(\phi_{1,t})\\
&\quad-(I+f^{-1}D(f^{-1}Dp_{1,t}))^{-1}(f^{-1}J\nabla p_{1,t}-f^{-2}\partial_t\nabla p_{1,t})(\phi_{2,t})\|_{L^2}\\
&\quad+\|(I+f^{-1}D(f^{-1}Dp_{1,t}))^{-1}(f^{-1}J\nabla p_{1,t}-f^{-2}\partial_t\nabla p_{1,t})\\
&\quad-(I+f^{-1}D(f^{-1}Dp_{2,t}))^{-1}(f^{-1}J\nabla p_{2,t}-f^{-2}\partial_t\nabla p_{2,t})\|_{L^2}\\
&\leq\|D[(I+f^{-1}D(f^{-1}Dp_{1,t}))^{-1}
(f^{-1}J\nabla p_{1,t}-f^{-2}\partial_t\nabla p_{1,t})]
\|_{L^{\infty}}\|\phi_{1,t}-\phi_{2,t}\|_{L^2}\\
&\quad+C\|p_{1,t}-p_{2,t}\|_{H^2}+C\|\partial_t\nabla p_{1,t}-\partial_t\nabla p_{2,t}\|_{L^2}.
\end{split}
\end{equation*}
Now, applying the Gronwall inequality, we conclude the proof.
\end{proof}
Combining all above estimates, we can now conclude uniqueness.
Indeed, from Lemma \ref{6.53} and Lemma \ref{6.28}, we get
\begin{equation*}
\|\phi_{1,t}-\phi_{2,t}\|_{L^2}\leq C\int_0^t\|p_{1,s}-p_{2,s}\|_{H^2}ds.
\end{equation*}
Let $a(t)=\sup_{0\leq s\leq t}\|\phi_{1,s}-\phi_{2,s}\|_{L^2}$, and
$b(t)=\sup_{0\leq s\leq t}\|p_{1,s}-p_{2,s}\|_{H^2}$. Therefore
\begin{equation*}
a(t)\leq C\int_0^tb(s)ds.
\end{equation*}
Using Proposition \ref{6.32} we obtain
\begin{equation*}
a(t)\leq C\int_0^tb(s)ds\leq C\int_0^ta(s)ds.
\end{equation*}
Also $a(0)=0$ since $\phi_{i,0}=Id$. Then from Gronwall inequality, $a(t)=0$ for all $t>0$,
that is $\phi_{1,t}=\phi_{2,t}$ for all $t>0$, which implies
$\mathbf{u}_1=\mathbf{u}_2$.
Also, from $a(t)\equiv 0$ and Proposition \ref{6.32}, $b(t)=0$ for all $t>0$,
that is
$p_{1,t}= p_{2,t}$ for all $t>0$. Now uniqueness is proved.
\subsection{Uniqueness for SG Shallow Water with variable $f$}
\label{uniqShWat-Sect}
The uniqueness for SG Shallow Water case is very similar to SG case.
Let $\phi(t)$ be the flow map,  and $h(t,x)$ be the fluid depth field.
Assume that (\ref{1.11}) is satisfied. Then similar to (\ref{6.1}), we have
\begin{equation}\label{6.59}
\partial_t\phi_t(x)+f^{-1}(\phi_t(x))\partial_t(f^{-1}(\phi_t(x)))\nabla h_t(\phi_t(x)))=f^{-1}J\nabla h_t(\phi_t(x)).
\end{equation}
The same calculation as in SG case shows (\ref{6.2})-(\ref{6.7}) still holds with $p_t$ replaced by $h_t$. In particular, we have
\begin{equation}\label{6.60}
(I+f^{-1}D(f^{-1}Dh_t))(\phi_t(x))D\phi_t(x)=(I+f^{-1}D(f^{-1}Dh_0))(x)+H(t,x).
\end{equation}
Here $H$ has the same expression as given by (\ref{6.13}) and solves an
initial value problem similar to (\ref{6.14})-(\ref{6.17}) with $E$ given by
(\ref{6.8})-(\ref{6.10}), the only difference being $p$ replaced by $h$.
Similar to Lemma \ref{6.21}. We can conclude $\|H(t,\cdot)\|_{C^1}\leq Ct$.
Here $C$ depends on $\|h\|_{L^{\infty}(0,T_0;C^3(\bT^2))}$,
$\|\partial_t h\|_{L^{\infty}(0,T_0;C^2(\bT^2))}$, $\lambda$, $T_0$,
where $\lambda>0$ is such number that
$I+f^{-1}D(f^{-1}Dh)\ge\lambda I$ on $[0,T_0]\times\bT^2$,
which exists by (\ref{1.11}).

The difference arises in (\ref{6.18}). Since $\phi_t$ does not preserve Lebesgue
measure, but satisfies $(\phi_t)_\# h_0=h_t$, we have
\begin{equation}\label{6.61}
\det D\phi_t(x)=\frac{h_0(x)}{h_t(\phi_t(x))}.
\end{equation}

Taking determinant on both sides of (\ref{6.60}), we obtain
\begin{equation*}
\frac{\det(I+f^{-1}D(f^{-1}Dh_t)(\phi_t(x)))}{h_t(\phi_t(x))}=\frac{\det(I+f^{-1}D(f^{-1}Dh_0)(x)+H(t,x))}{h_0(x)}.
\end{equation*}
Taking  this at point $\phi_t^{-1}(x)$, we get
\begin{equation}\label{6.63}
\frac{\det(I+f^{-1}D(f^{-1}Dh_t)(x))}{h_t(x)}=
\frac{\det(I+f^{-1}D(f^{-1}Dh_0)(\phi_t^{-1}(x))+H(t,\phi_t^{-1}(x)))}{h_0(\phi_t^{-1}(x))}.
\end{equation}
As in SG case, we solve (\ref{6.59}) for $\partial_t\phi_t$:
\begin{equation*}
\partial_t\phi_t(x)=(I+f^{-1}D(f^{-1}Dh_t))^{-1}(\phi_t(x))[f^{-1}J\nabla h_t-f^{-2}\partial_t\nabla h_t](\phi_t(x)).
\end{equation*}
Or equivalently, writing this at  point $\phi_t^{-1}(x)$, we get expression
for physical velocity:
\begin{equation}\label{6.64}
\mathbf{u}_t=(I+f^{-1}D(f^{-1}Dh_t))^{-1}(f^{-1}J\nabla h_t-f^{-2}\partial_t\nabla h_t).
\end{equation}
Due to the equation $\partial_th+\nabla\cdot(\mathbf{u}_th_t)=0$, we obtain an
elliptic equation for $\partial_th$, namely
\begin{equation}\label{6.66}
\nabla\cdot[(I+f^{-1}D(f^{-1}Dh_t))^{-1}f^{-2}h_t\partial_t\nabla h_t]-\partial_th=\nabla\cdot[(I+f^{-1}D(f^{-1}Dh_t))^{-1}f^{-1}h_tJ\nabla h_t].
\end{equation}

 Suppose we have two solutions of SG shallow water equations,
 $(h_1,\mathbf{u}_{1,g},\mathbf{u}_1)$ and $(h_2,\mathbf{u}_{2,g},\mathbf{u}_2)$,
 which satisfy the assumptions of Theorem \ref{1.9}, and defined on
 $[0,T_0]\times\bR^2$. Then there exist constants $\lambda, \,\mu>0$
 such that
\begin{equation}\label{assumptOnSOlnUniqSGSW}
 I+f^{-1}D(f^{-1}Dh_i)\ge \lambda I \;\textrm{ and }\;
 h_i\ge \mu  \;\textrm{ for }\;i=1,2
 \;\textrm{ on }\;[0,T_0]\times\bT^2.
\end{equation}
 Due to (\ref{6.64}),
 we have $\mathbf{u}_i\in L^{\infty}(0,T_0;C^1)$, hence we can define the flow
 map $\phi_i(t)$.

Next we present the estimates in the shallow water case which lead to uniqueness,
parallel to previous subsection.

As a preparation, we first observe that the analogue of Lemma \ref{6.24} also holds for the shallow water setting.
\begin{lem}\label{9.8nn}
For any $t\in[0,T_0]$, we have
\begin{equation}\label{9.62new}
\|\phi_{1,t}^{-1}-\phi_{2,t}^{-1}\|_{L^2}\leq C\|\phi_{1,t}-\phi_{2,t}\|_{L^2}
\end{equation}
and
\begin{equation}\label{9.63new}
\|\phi_{1,t}-\phi_{2,t}\|_{L^2}\leq C\|\phi_{1,t}^{-1}-\phi_{2,t}^{-1}\|_{L^2}.
\end{equation}
\end{lem}
Here $C$ depends only on $\lambda$, $\mu$, $\|h_{i,t}\|_{L^{\infty}(0,T_0;C^2(\bT^2))}$,
$\|\partial_th_i\|_{L^{\infty}(0,T_0;C^1(\bT^2))}$, $i=1,2$.
\begin{proof}
From (\ref{6.64}), $\mathbf{u}_t\in L^{\infty}(0,T_0;C^1(\bT^2))$.
Since the flow map $\phi_i$ is generated by $\mathbf{u}_i$, it follows that
$\|D\phi_{i,t}\|_{L^{\infty}}$ are uniformly bounded in $t$,
with a bound having the dependence as in the lemma.
Also, from (\ref{6.61}) combined with
 (\ref{1.10}) and (\ref{assumptOnSOlnUniqSGSW}),
it follows that
 $\frac{c_1}{\|h_{i,t}\|_0}\leq\det D\phi_{i,t}\leq\frac{\|h_0\|_0}{\mu}$,
 for $i=1,2$.
From this we conclude $\|D\phi_{i,t}^{-1}\|_{L^{\infty}}$ are also uniformly
bounded in $t\in[0, T_0]$, with a bound having the dependence as in the lemma.

To prove (\ref{9.62new}), we can calculate
\begin{equation*}
\begin{split}
&\int_{[0,1)^2}|\phi_{1,t}^{-1}(x)-\phi_{2,t}^{-1}(x)|^2dx=\int_{[0,1)^2}|x-\phi_{2,t}^{-1}(\phi_{1,t}(x))|^2\det D\phi_{1,t}(x)dx\\
&=\int_{[0,1)^2}|\phi_{2,t}^{-1}(\phi_{2,t}(x))-\phi_{2,t}^{-1}(\phi_{1,t}(x))|^2\det D\phi_{1,t}(x)dx\\
&\leq \|D\phi_{2,t}^{-1}\|_{L^{\infty}}^2\|\det D\phi_{1,t}\|_{L^{\infty}}\|\phi_{1,t}-\phi_{2,t}\|_{L^2}^2.
\end{split}
\end{equation*}
The proof for (\ref{9.63new}) is similar.
\end{proof}

Similar to Lemma \ref{6.28}, we can show
\begin{lem}\label{6.67}
For each $t\in[0,T_0]$, we have
\begin{equation}\label{6.68}
\|\partial_t\nabla h_{1,t}-\partial_t\nabla h_{2,t}\|_{L^2}\leq C\|h_{1,t}-h_{2,t}\|_{H^2}.
\end{equation}
Here $C$ depends on $\lambda, \mu,\|h_{i,t}\|_{L^{\infty}(0,T_0;C^2)},\|\partial_th_i\|_{L^{\infty}(0,T_0;C^1)}$, $i=1,2$.
\end{lem}
\begin{proof}
Similar to the proof of Lemma \ref{6.28}, we take equation (\ref{6.66}) for
$h_1$ and $h_2$, and subtract. 
Multiplying both sides of the resulting equation by
$\partial_th_{1,t}-\partial_th_{2,t}$
and integrating by parts, we obtain (\ref{6.68}),
noticing that the  zero-th order term
has the right sign so we can argue in the
same way as in Lemma \ref{6.28}.
\end{proof}
The key proposition holds also in SG shallow water case. It is an analogue to
Proposition \ref{6.32}.
\begin{prop}\label{6.69}
For any $t\in[0,T_0]$
\begin{equation*}
\sup_{0\leq s\leq t}\|h_{1,s}-h_{2,s}\|_{H^2}\leq C\sup_{0\leq s\leq t}\|\phi_{1,s}-\phi_{2,s}\|_{L^2}.
\end{equation*}
Here $C$ depends on
$\lambda, \mu, \|h_i\|_{L^{\infty}(0,T_0;C^3)},\|\partial_th_i\|_{L^{\infty}(0,T_0;C^2)},T_0$, $i=1,2$.
\end{prop}
\begin{proof}
Similar to the proof of Proposition \ref{6.32}, we take equation (\ref{6.63})
for $h_1,h_2$ and subtract. Then we obtain equation
for $h_{1,t}-h_{2,t}$, with the left-hand side
\begin{equation*}
\begin{split}
&\quad\frac{\det(I+f^{-1}D(f^{-1}Dh_{1,t})(x))}{h_{1,t}(x)}-\frac{\det(I+f^{-1}D(f^{-1}Dh_{2,t})(x))}{h_{2,t}(x)}\\
&=\frac{a_{ij}f^{-1}\partial_j(f^{-1}\partial_i(h_{1,t}-h_{2,t}))}{h_{1,t}}-\frac{\det(I+f^{-1}D(f^{-1}Dh_{2,t}))}{h_{1,t}h_{2,t}}(h_{1,t}-h_{2,t}).
\end{split}
\end{equation*}
Here
\begin{equation*}
a_{ij}=\int_0^1M_{ji}(I+\theta f^{-1}D(f^{-1}Dh_{1,t})+(1-\theta)f^{-1}D(f^{-1}Dh_{2,t}))d\theta.
\end{equation*}
Property (\ref{assumptOnSOlnUniqSGSW}) and regularity assumptions for
$h_1$, $h_2$ imply that the equation is uniformly elliptic, and coefficients
are   in $C^1(\bT^2)$, with ellipticity constant and $C^1$ norms depending on
the parameters listed in the lemma.
The sign of the coefficient of zero-th order term is negative, by
property (\ref{assumptOnSOlnUniqSGSW}). Hence we can apply
the $H^2$ estimates to conclude
\begin{equation*}
\|h_{1,t}-h_{2,t}\|_{H^2}\leq C\|F\|_{L^2},
\end{equation*}
where $F$ is the right-hand side of the equation.
Now we can do a calculation as in the proof of Proposition \ref{6.32} to
estimate $\|F\|_{L^2}$.
\end{proof}
Finally,
\begin{lem}\label{6.74}
\begin{equation*}
\|\phi_{1,t}-\phi_{2,t}\|_{L^2}\leq C\int_0^t(\|h_{1,s}-h_{2,s}\|_{H^2}+\|\partial_t\nabla h_{1,s}-\partial_t\nabla h_{2,s}\|_{L^2})ds.
\end{equation*}
Here $C$ depends on $\|h_i\|_{L^{\infty}(0,T_0;C^3)},\|\partial_th_i\|_{L^{\infty}(0,T_0;C^2)}$, $i=1,2$.
\end{lem}
\begin{proof}
We argue exactly as in the proof of Lemma \ref{6.53}, using (\ref{6.61})
for $\phi_{2,t}$ instead of the measure-preserving property.
\end{proof}
With above estimates, we can conclude uniqueness in the same way as for SG case.

\section{Appendix}
\label{appendix-section}
\subsection{Frechet derivatives for maps between H\"{o}lder spaces}
\begin{lem}\label{CompositionFrechetDif}
Let $\alpha,\beta\in (0,1)$ and $\alpha>\beta$. Then,
\begin{enumerate}[(i)]
\item
 The map $G:C^{2,\alpha}(\bT^2)\times C^{1,\beta}(\bT^2;\bR^2)\rightarrow C^{1,\beta}(\bT^2)$
defined by
$$
G(p, w)= p\circ (id+w)
$$
is continuously Frechet-differentiable.
\item
The map $G$ above considered as a map $C^{1,\alpha}(\mathbb{T}^2)\times
C^{1,\beta}(\mathbb{T}^2;\mathbb{R}^2)\rightarrow C^{0,\beta}(\mathbb{T}^2)$
is also continuously Frechet differentiable.
\end{enumerate}
\end{lem}
\begin{proof}
In the argument below, the constant $C$ may change from line to line, and depends only on
$\alpha$, $\beta$, $\|p\|_{C^{2,\alpha}(\bT^2)}$ and
$\|w\|_{C^{1,\beta}(\bT^2;\bR^2)}$.

We only prove the first statement. The second statement is proved in a similar manner and is simpler.

That $G$ maps into $C^{1,\beta}(\mathbb{T}^2)$ is clear from the expression.  We show Frechet differentiability,
with differential $D G(p,w)$ for any $(p, w)\in C^{2,\alpha}(\mathbb{T}^2)\times C^{1,\beta}(\mathbb{T}^2;\mathbb{R}^2)$
given by
\begin{equation}\label{formOfFrechetDeriv-1}
\begin{split}
D&G(p,w):C^{2,\alpha}(\mathbb{T}^2)\times C^{1,\beta}(\mathbb{T}^2;\mathbb{R}^2)\rightarrow C^{1,\beta}(\mathbb{T}^2),\\
&(h_1,h_2)\longmapsto h_1\circ(id+w)+(\nabla p\circ(id+w))\cdot h_2.
\end{split}
\end{equation}
Here the second term above denotes the dot product of 2 vectors in $\mathbb{R}^2$.

By definition, we need to show
$$\frac{\|G(p+\eps h_1,w+\delta h_2)-G(p,w)-DG(p,w)(\eps h_1,\delta h_2)\|_{1,\beta}}{\eps+\delta}\rightarrow0,
$$
as $\eps,\delta\rightarrow0$, uniformly for any $\|h_1\|_{2,\alpha}, \|h_2\|_{1,\beta}\leq1$.
we can write
\begin{equation*}
\begin{split}
[G(p+\eps h_1,w+\delta h_2)& -G(p,w)-DG(p,w)(\eps h_1,\delta h_2)](x)\\
&=\int_0^1(\nabla p(x+w(x)+\theta\delta h_2(x))-\nabla p(x+w(x)))d\theta\cdot\delta h_2(x)\\
&\qquad\qquad+\eps(h_1(x+w(x)+\delta h_2(x))-h_1(x+w(x)).
\end{split}
\end{equation*}
Therefore,
\begin{equation}\label{9.1}
\begin{split}
\|G(& p+\eps h_1,w+\delta h_2)-G(p,w)-DG(p,w)(\eps h_1,\delta h_2)\|_{1,\beta}\\
&\leq\delta \sup_{0\leq\theta\leq1}\|\nabla p(id+w+\theta\delta h_2)-
\nabla p(id+w)\|_{1,\beta}\|h_2\|_{1,\beta}\\
&\qquad\qquad\qquad+\eps\|h_1(id+w+\delta h_2)-h_1(id+w)\|_{1,\beta}\\
&\leq\delta \sup_{0\leq\theta\leq1}\|\nabla p(id+w+\theta\delta h_2)-
\nabla p(id+w)\|_{1,\beta}+\eps\|h_1(id+w+\delta h_2)-h_1(id+w)\|_{1,\beta}.
\end{split}
\end{equation}
Now
\begin{equation*}
\begin{split}
 D[\nabla & p(id+w+\theta\delta h_2)-\nabla p(id+w)]\\
&=D^2p(id+w+\theta\delta h_2)(I+Dw+\theta\delta Dh_2)-D^2p(id+w)(I+Dw)\\
&=[D^2p(id+w+\theta\delta h_2)-D^2p(id+w)](I+Dw+\theta\delta Dh_2)+
\theta\delta D^2p(id+w) Dh_2.
\end{split}
\end{equation*}
Therefore,
\begin{equation*}
\begin{split}
\|D[\nabla & p(id+w+\theta\delta h_2)-\nabla p(id+w)]\|_{0,\beta}\\
&\leq\|D^2p(id+w+\theta\delta h_2)-D^2p(id+w)\|_{0,\beta}\|I+Dw+\theta\delta Dh_2\|_{0,\beta}\\
&\qquad\qquad\qquad+\delta\|D^2p(id+w)\|_{0,\beta}\|Dh_2\|_{0,\beta}\\
&\leq\|D^2p(id+w+\theta\delta h_2)-D^2
p(id+w)\|_{0,\beta}(1+\|w\|_{1,\beta}
+\delta)+\delta C\\
&\leq C\|D^2p(id+w+\theta\delta h_2)-D^2
p(id+w)\|_{0,\beta}+\delta C
\end{split}
\end{equation*}
Finally, we use interpolation inequality to estimate the first term above:
\begin{equation*}
\begin{split}
\|D^2p&(id+w+\theta\delta h_2)-D^2p(id+w)\|_{0,\beta}\\
&\leq\|D^2p(id+w+\theta\delta h_2)-D^2p(id+w)\|_{0,\alpha}^{\beta/\alpha}\cdot
\|D^2p(id+w+\theta\delta h_2)-D^2p(id+w)\|_0^{1-\beta/\alpha}\\
&\leq C(\|p\|_{2,\alpha}
\|h_2\|_0^{\alpha}\,\delta^{\alpha})^{1-\beta/\alpha}\\
&\le C\delta^{\alpha-\beta}
\rightarrow 0.
\end{split}
\end{equation*}
The second term in (\ref{9.1}) can be handled in a similar way.

Also using explicit expression  of $DG(p,w)$ in (\ref{formOfFrechetDeriv-1}), and
arguing in a similar way as above, we
show that $DG$ is continuous in $p,w$ under the operator norm.
\end{proof}
\begin{lem}\label{l9.2}
Let $f:\mathbb{R}^n\rightarrow\mathbb{R}$ be a $C^3$ function, then the map
\begin{equation*}
\begin{split}
F:&C^{0,\alpha}(\mathbb{T}^2)\times C^{0,\alpha}\cdots C^{0,\alpha}(\mathbb{T}^2)\rightarrow C^{0,\alpha}(\mathbb{T}^2)\\
&(p_1,p_2,\cdots,p_n)\longmapsto f(p_1,p_2,\cdots p_n).
\end{split}
\end{equation*}
is continuously Frechet differentiable.
\end{lem}
\begin{proof}
For simplicity, we just prove this lemma for $n=2$. The general $n$ can be handled in a similar way. From Taylor`s formula, one can write
$$f(p_1+\eps h_1,p_2+\delta h_2)=f(p_1,p_2)+\eps h_1\partial_1f(p_1)+\delta h_2\partial_2f(p_2)+\int_0^1A(x,t)(1-t)dt,
$$
where
$$
A(x,t)=\partial_{11}f(\eps h_1)^2+2\partial_{12}f(\eps h_1)(\delta h_2)+\partial_{22}f(\delta h_2)^2.
$$
Here $\partial_{ij}f$ is evaluated at $(p_1+t\eps h_1,p_2+t\delta h_2)$. Now observe that if
$\|h_1\|_{0,\alpha},\|h_2\|_{0,\alpha}\leq 1$, one has a uniform $C^{0,\alpha}$ bound for $\frac{A(x,t)}{\eps^2+\delta^2}$.
This proves the map $F$ is differentiable. The continuity of the differential follows readily from the
expression $D_{p_i}F(p_1,p_2,\cdots,p_n)h_i=\partial_if(p)h_i$.

\end{proof}

\begin{cor}\label{c9.3}
The map
$$\det:C^{0,\alpha}(M_{2\times2})\rightarrow C^{0,\alpha}
$$
$$A\longmapsto \det A.
$$
is continuously Frechet differentiable.Here $C^{0,\alpha}(M_{2\times2})$ denotes the space of $2\times2$ matrices with entries in $C^{0,\alpha}(\mathbb{T}^2)$.

\end{cor}

Similar to Lemma \ref{CompositionFrechetDif}, one can also show
\begin{lem}\label{l9.4}
Let $f\in C^2(\mathbb{T}^2)$, then the map
$$F:C^{1,\alpha}(\mathbb{T}^2;\mathbb{R}^2)\rightarrow C^{0,\alpha}(\mathbb{T}^2)
$$
$$w\longmapsto f(id+w).
$$
is continuously Frechet differentiable.
\end{lem}
\begin{proof}
Let $\|h\|_{1,\alpha}\leq 1$, we can write
\begin{equation*}
\begin{split}
&f(id+w+\eps h)-f(id+w)-\nabla f(id+w)\eps h\\
&\qquad=\eps h\int_0^1[\nabla f(id+w+\eps\theta h)-\nabla f(id+w)]d\theta.
\end{split}
\end{equation*}
Therefore,
\begin{equation*}
\begin{split}
&\|f(id+w+\eps h)-f(id+w)-\nabla f(id+w)\eps h\|_{0,\alpha}\\
&\qquad\leq\eps\|h\|_{0,\alpha}\sup_{0\leq\theta\leq1}\|\nabla f(id+w+\eps\theta h)-\nabla f(id+w)\|_{0,\alpha}.
\end{split}
\end{equation*}
Now we use interpolation inequality:
\begin{equation*}
\begin{split}
\|\nabla f( & id+w+\eps\theta h)-\nabla f(id+w)\|_{0,\alpha}\\
&\leq\|\nabla f(id+w+\eps\theta h)-\nabla f(id+w)\|_{C^1}^{\alpha}\cdot\|\nabla f(id+w+\eps\theta h)-\nabla f(id+w)\|_0^{1-\alpha}.\\
&\leq C(\eps\|D^2f\|_0\|h\|_0)^{1-\alpha}.
\end{split}
\end{equation*}
\end{proof}

\subsection{Integration on $\bT^2$}
We note the following elementary property of $\bZ^2$-periodic functions:
\begin{lem}\label{chgVarOnTorus}
Let $G: \bR^2\to \bR^2$ be a $C^1$ diffeomorphism, $h\in C(\bR^2)$, and let
 $G-id$ and $h$ be $\bZ^2$-periodic. Then
$$
\int_{G([0,1)^2)}h(y)dy=
\int_{[0,1)^2}h(x)dx.
$$
\end{lem}
\begin{proof}
We can write
$$
G([0,1)^2)=\cup_{(i,j)\in\bZ^2} A_{ij}, \;\;\textrm{ where $A_{ij}$ are
disjoint, and }\;
A_{ij}-(i,j)\subset [0,1)^2.
$$
Denote $B_{ij}:=A_{ij}-(i,j)$.
We show that
\begin{equation}\label{periodicShift}
B_{ij}\;\textrm{  are disjoint, and }\; [0,1)^2=\cup_{(i,j)\in\bZ^2} B_{ij}.
\end{equation}
Indeed, suppose $x\in B_{ij}\cap B_{kl}$. Then
$x+(i,j)=G(y)$ and $x+(k,l)=G(z)$ for $y,z\in[0, 1)^2$. Then using
$\bZ^2$-periodicity of  $G-id$, we have  $G(y)-y=G(y-(i,j))-(y-(i,j))$ and
 $G(z)-z=G(z-(k,l))-(z-(k,l))$,
from which $x=G(y)-(i,j)=G(y-(i,j))$ and similarly
$x=G(z-(k,l))$. Since $G: \bR^2\to \bR^2$ is a diffeomorphism, we get
$y-(i,j))=z-(k,l)$, which implies $y=z$ and $(i,j)=(k,l)$ since
$y,z\in[0, 1)^2$. Thus $B_{ij}$  are disjoint. Next, let $x\in [0, 1)^2$.
Since $G:\bR^2\to \bR^2$ is a diffeomorphism, then $x=G(y)$ for some
$y\in\bR^2$. Then $y=z-(i,j)$ for some $z\in[0,1)^2$ and $(i,j)\in \bZ^2$.
Using again $\bZ^2$-periodicity of  $G-id$, we get $x=G(z-(i,j))=G(z)-(i,j)$,
from which, using that $x,z\in [0, 1)^2$, we get $x\in B_{ij}$.
Thus $[0,1)^2=\cup_{(i,j)\in\bZ^2} B_{ij}$. Now (\ref{periodicShift}) is proved.

With this, using  $\bZ^2$-periodicity of $h$, we get
$$
\int_{G([0,1)^2)}h(y)dy=\sum_{(i,j)\in\bZ^2}\int_{ A_{ij}}h(y)dy
=\sum_{(i,j)\in\bZ^2}\int_{ B_{ij}}h(x)dx=
\int_{[0,1)^2}h(x)dx.
$$
\end{proof}

\section*{Acknowledgements}
The work of Jingrui Cheng was supported in part by the National Science
Foundation under Grant DMS-1401490.
The work of Mikhail Feldman was supported in part by the National Science
Foundation under Grant DMS-1401490, and the
Van Vleck
Professorship Research Award by the University of Wisconsin-Madison.

\end{document}